\documentclass[english, 11pt]{amsart}
\usepackage{amssymb}
\usepackage{amsfonts}
\usepackage{amscd}
\usepackage[T1]{fontenc}
\usepackage{enumerate}%
\usepackage{amsthm,lipsum}
\usepackage[colorlinks, linkcolor=red]{hyperref}
\usepackage{color}
\usepackage[utf8]{inputenc}

\usepackage{geometry}
 \usepackage{imakeidx}
 \makeindex

\definecolor{immi}{rgb}{0,.6,.1}

\newbox\removebox
\newcommand\remove[2]{%
\setbox\removebox=\ifmmode\hbox{$#2$}\else\hbox{#2}\fi%
\leavevmode
\rlap{\textcolor{#1}{\vrule height0.8ex depth-0.6ex width\wd\removebox}}%
\box\removebox
}
\long\def\bigremove#1{%
\par\setbox\removebox=\vbox{#1}%
\vbox{%
\vbox to0pt{\hbox{\tikz\draw[color=blue,thick] (0,0) -- (\wd\removebox,-\ht\removebox)  (\wd\removebox,0) -- (0,-\ht\removebox);}}
\box\removebox
}
}

\usepackage{mathrsfs} 

\makeatletter
\def\RFss@@#1{\RF^*_{\!*#1}}
\def\RFss@_#1{\RFss@@{,#1}}
\def\RFss{\@ifnextchar_{\RFss@}{\RFss@@{}}}
\makeatother

\newcommand{\RF}{{\rm RF}}
\def\BSing{\mathrm{BSing}}
\def\Ch{\mathrm{char}}

\def\Supp{\operatorname{Supp}}

\def\lct{\operatorname{lct}}
\def\moi{\operatorname{moi}}

\def\deg{\operatorname{deg}}

\def\Supp{\operatorname{Supp}}

\def\rank{\operatorname{rank}}

\DeclareMathOperator*{\depth}{depth}

\DeclareMathOperator*{\cond}{\mathfrak{c}\mathfrak{o}\mathfrak{n}\mathfrak{d}}

\def\Sing{\operatorname{Sing}}

\def\11{{\mathbf 1}}
\def\AA{{\mathbb A}}

\def\CC{{\mathbb C}}

\def\FF{{\mathbb F}}

\def\NN{{\mathbb N}}

\def\PP{{\mathbb P}}
\def\QQ{{\mathbb Q}}
\def\RR{{\mathbb R}}

\def\ZZ{{\mathbb Z}}

\def\cA{{\mathcal A}}
\def\cB{{\mathcal B}}

\def\cE{{\mathcal E}}

\def\cG{{\mathcal G}}

\def\cI{{\mathcal I}}
\def\cJ{{\mathcal J}}

\def\cL{{\mathcal L}}
\def\cM{{\mathcal M}}
\def\cN{{\mathcal N}}
\def\cO{{\mathcal O}}

\def\cR{{\mathcal R}}
\def\cS{{\mathcal S}}

\def\cU{{\mathcal U}}
\def\cV{{\mathcal V}}

\def\cX{{\mathcal X}}
\def\cY{{\mathcal Y}}
\def\cZ{{\mathcal Z}}

\newcommand{\spec}{\operatorname{Spec}}

\newtheorem{thm}{Theorem}[section]
\newtheorem{lem}[thm]{Lemma}
\newtheorem{lemdefn}[thm]{Lemma-Definition}
\newtheorem{cor}[thm]{Corollary}
\newtheorem{prop}[thm]{Proposition}
\newtheorem{conj}[thm]{Conjecture}

\newtheorem{Op}[thm]{Open Question}

\theoremstyle{definition}
\newtheorem{defn}[thm]{Definition}

\newtheorem{example}[thm]{Example}

\newtheorem{def-prop}[thm]{Proposition-Definition}
\newtheorem{def-theorem}[thm]{Theorem-Definition}
\newtheorem{def-lem}[thm]{Lemma-Definition}

\theoremstyle{remark}
\newtheorem{remark}[thm]{Remark}

\theoremstyle{plain}
\newtheorem*{namedthm}{\namedthmname}
\newcounter{namedthm}

\makeatletter
\newenvironment{named}[1]
  {\def\namedthmname{#1}%
   \refstepcounter{namedthm}%
   \namedthm\def\@currentlabel{#1}}
  {\endnamedthm}
\makeatother

\theoremstyle{plain}

  {\par\medskip\noindent #1\par\begingroup%
    \advance\leftskip by 1em\advance\rightskip by 1em}%
  {\par\endgroup}

\DeclareMathOperator*{\Spec}{Spec}
\DeclareMathOperator*{\loc}{loc}

\newcommand{\ord}{\operatorname{ord}}

\newcommand{\Oi}{\operatorname{Oi}}
\newcommand{\Specm}{\operatorname{Specmax}}

\def\cI{\mathcal{I}}
\def\cJ{\mathcal{J}}

\def\cA{\mathcal{A}}

\def\cG{\mathcal{G}}
\def\cL{\mathcal{L}}
\def\cO{\mathcal{O}}

\def\cU{\mathcal{U}}

\renewcommand{\phi}{\varphi}
\renewcommand{\epsilon}{\varepsilon}
\renewcommand{\theta}{\vartheta}

\renewcommand{\and}{ \quad \text{and} \quad }

\begin{document}

\setcounter{tocdepth}{1} 

\author[K.~H.~Nguyen]
{Kien Huu Nguyen}

\address{KU Leuven, Department of Mathematics,
Celestijnenlaan 200B, B-3001 Leu\-ven, Bel\-gium}
\email{kien.nguyenhuu@kuleuven.be}
\address{Institute of Mathematics, Vietnam Academy of Science and Technology, 18 Hoang Quoc Viet, Nghia Do, Hanoi, Vietnam}
\email{nhkien@math.ac.vn}

%
%

\begin{abstract}
In 2006, Budur, Musta\c{t}\v{a} and Saito introduced the notion of Bernstein-Sato polynomial of an arbitrary scheme of finite type over fields of characteristic zero. By the strong monodromy conjecture, it should have a corresponding picture on the arithmetic side of ideals in polynomial rings. In this paper, we try to address this problem. By using an idea inspired by the Hardy-Littlewood circle method, we introduce the notions of abstract exponential sums modulo $p^m$ and  motivic oscillation index of an arbitrary ideal of polynomial rings over number fields. In the arithmetic picture, abstract exponential sums modulo $p^m$ and the motivic oscillation index of an ideal should play the role of the Bernstein-Sato polynomial of the corresponding scheme and its maximal non-trivial root respectively. We will provide some properties of the motivic oscillation index of ideals in this paper. On the other hand,  based on Igusa's conjecture for exponential sums, we formulate an averaged Igusa conjecture for exponential sums of ideals. In particular, this conjecture and the motivic oscillation index of ideals will have many interesting applications. We will introduce these applications and prove a variant of this conjecture.

\end{abstract}
\title[Motivic oscillation index of arbitrary ideals]{Exponential sums and motivic oscillation index of arbitrary ideals and their applications}


\maketitle
\tableofcontents
\section{Introduction}\label{intro}
\subsection{Motivation}Let $f\in\ZZ[x_1,\dots,x_n]$ be a non-zero polynomial. One of the most important questions for number theorists is to know whether the Diophantine equation $f(x)=0$ has an integral solution. If $f$ is homogeneous, then $0$ is always a solution of $f$. Thus, we should ask about the existence of a non-zero integral solution of $f$. The local-global principle (or Hasse principle for short) is a useful tool to answer this question in many cases. More precisely, if $f$ has an integral solution, then $f$ has a real solution and the congruence equation $f(x)\equiv 0 \mod p^m$ has a solution for each prime $p$ and each positive integer $m$. The Hasse principle asks whether the reverse is also true.  In other words, it aims to find certain sufficient conditions for the reverse to
hold. Whenever the Hasse principle holds for $f$, we have many tools to check whether $f$ has an integral solution. Indeed, we can use the quantifier elimination theorem of real closed fields (Tarski-Seidenberg theorem) to verify the existence of a real solution of $f$ by looking at the coefficients of $f$. Moreover, we can use the Lang-Weil estimate and Hensel's lemma to check the existence of solutions for the congruence equations of $f$ modulo $p^m$. 


Igusa approached the Hasse principle by an adèlic method. His idea is based on an expectation that a deep enough understanding of the information about $f$ over local fields might help to study the Hasse principle of $f$. One of the most important pieces of information about $f$ is its densities over local fields. Let $p$ be a prime then the $p$-adic density $\mathfrak{d}_p(f)$ of $f$ equals to
\[\lim_{m\to +\infty} p^{-m(n-1)}\#\{x\in(\ZZ/p^m\ZZ)^n\mid f(x)=0\in \ZZ/p^m\ZZ\}.\]
Thus, we wish to know when this limit exists. We set $N_{p,m}=\#\{x\in(\ZZ/p^m\ZZ)^n|f(x)=0\in \ZZ/p^m\ZZ\}$. Then the information about $\mathfrak{d}_p(f)$ is captured by the Poincaré series
$$P_{p,f}(T)=\sum_{m\geq 0}p^{-mn}N_{p,m}T^{m}.$$
In \cite{Bosha}, Borewicz and Shafarevich conjectured that $P_{p,f}(T)$ is a rational function of $T$. Igusa proved this conjecture by using Hironaka's embedded resolution of singularities (see \cite{Igusa3}) for the hypersurface of equation $f=0$. Note that Denef gave alternative proof of this conjecture by using model theoretic tools (see \cite{Denef}). Igusa's proof is based on the fact that
\begin{equation}\label{Igu00}
P_{p,f}(p^{-s})=\frac{1-p^{-s}\cZ_{p,f}(s)}{1-p^{-s}},
\end{equation}
where $\cZ_{p,f}(s)$ is the Igusa local zeta function associated to $f$ defined by 
\[\cZ_{p,f}(s)=\int_{\ZZ_p^n}\left| f(x)\right|^s\left|dx\right|,\]
for $s\in\CC, \Re(s)>0$, $\left| .\right|$ is the $p$-adic norm and $\left|dx\right|$ is the normalized Haar measure on $\QQ_p^n$ such that the volume of $\ZZ_p^n$ is $1$ (see \cite{Igusa3,DenefBour} for more information about Igusa local zeta functions). In fact, by the definition of $\mathfrak{d}_p(f)$, the rationality of $P_{p,f}(p^{-s})$, partial fraction decomposition of rational functions and (\ref{Igu00}),  to know about the existence of $\mathfrak{d}_p(f)$, it suffices to look at the real part and the multiplicity of poles of $\cZ_{p,f}(s)$. 
On the other hand,  it is known that if $\mathfrak{d}_p(f)$ exists then one has 
\begin{equation}\label{density}
\mathfrak{d}_p(f)=\sum_{m\geq 0}p^{-nm}\sum_{\overline{a}\in(\ZZ/p^m\ZZ)^{\times}}\sum_{\overline{x}\in(\ZZ/p^m\ZZ)^n}\exp\left( \frac{2\pi i af(x)}{p^m}\right),
\end{equation}
moreover, $\mathfrak{d}_p(f)$ exists if and only if the series on the right-hand side of (\ref{density}) is convergent (see, e.g.,  \cite[Lemma 1.18]{BrowningC}). This shows the role of the exponential sums on the right hand side of (\ref{density}) in this picture. In fact, these exponential sums play a central role in the approach of Igusa. Recall that if $f$ is a quadratic form over the ring of integers in a number field then the Hasse principle is true for $f$  by the Hasse-Minkowski theorem. In order to obtain a version of  the Hasse-Minkowski theorem for a homogeneous polynomial $f$ of higher degree over $\ZZ$, Igusa hoped to establish a certain Poisson summation formula of Siegel-Weil type associated to $f$ on the ring of adeles over $\QQ$ (see \cite{Igusa3}), following the work of Weil (see \cite{Weil}) on abstract Poisson formula of continuous maps between locally compact abelian groups. Moreover, he expected that this Poisson formula might help to obtain a Siegel formula that relates a theta series with an Eisenstein series (or roughly speaking, this Siegel formula relates the number of integer solutions of $f(x)=0$ with the densities of $f$ over the completions of $\QQ$, see \cite{Weil,Igusa3,Haris}). To have this Poisson summation formula for a homogeneous polynomial $f$, Igusa needed a constant $\sigma>2$ and a constant $c>0$ such that
\begin{equation}\label{inequ}
\bigg| p^{-mn}\sum_{\overline{x}= (x \mod p^m) \in(\ZZ/p^m\ZZ)^n}\exp\left(\frac{2\pi iaf(x)}{p^m}\right)\bigg|\leq cp^{-m\sigma}
\end{equation}
for all large enough primes $p$, all integers $m\geq 1$ and all $a\in\ZZ_p^{*}$. Igusa's conjecture for exponential sums and its generalization predict that (\ref{inequ}) holds for all large enough primes $p$, all $m\geq 1$ (resp. $m\geq 2$), all $a\in\ZZ_p^{*}$ and a constant $c$ independent of $p,m$ if $f$ is homogeneous (resp. $f$ is not homogeneous) provided that $\sigma$ is less than the motivic oscillation index $\moi_\QQ(f)$ of $f$ over $\QQ$ (see \cite{Igusa3, CMN, NguyenVeys}). By (\ref{density}) and the Chinese remainder theorem, if Igusa's conjecture for exponential sums holds for $f$ then the condition $\moi_\QQ(f)>1$ implies that $\mathfrak{d}_p(f)$ exists for all large enough primes $p$ (thus for all primes $p$ as a consequence of the results in \cite{Igusa3,CMN}, which say we can provide a good embedded resolution of singularities of $\{f=0\}$ over $\QQ$ to compute $\cZ_{p,f}(s)$ and show that $N_{p,m}=\mathfrak{d}_p(f)p^{m(n-1)}+o(p^{m(n-1)})$ when $m\to+\infty$ for all $p$) and the condition $\moi_\QQ(f)>2$ implies that the singular series (see Definition \ref{defsing})
$$\mathfrak{S}(f)=\sum_{N\geq 1}N^{-n}\sum_{\overline{a}\in(\ZZ/N\ZZ)^{\times}}\sum_{\overline{x}\in(\ZZ/N\ZZ)^n}\exp\left(\frac{2\pi i af(x)}{N}\right)=\prod_{p}\mathfrak{d}_p(f)$$
converges absolutely.

In another approach to the Hasse principle of $f,$ the analytic number theorists use the Hardy-Littlewood circle method  (see, e.g.,  \cite{Birch}). This method may also help to understand the distribution of integral and rational solutions of $f$ of bounded height, i.e., one asks about the existence of  an asymptotic formula of the function $$N_{f,a}(B)=\#\{x=(x_1,...,x_n)\in \ZZ^n\mid |x_i|\leq B \hspace{0.1cm}\forall i, f(x)=a\}$$ when $B\to +\infty$ uniform in $a\in [-B^{\deg(f)}, B^{\deg(f)}]$.  In particular, the existence of the $p$-adic density $\mathfrak{d}_p(f)$ of $f$ and the convergence of $\mathfrak{S}(f)$ are also needed for this method. Therefore, although Igusa's method has a huge gap to achieve due to the absence of  metaplectic groups of Weil associated with homogeneous polynomials of degree at least $3$ (see \cite[Section 4]{Igusa3}), Igusa's conjecture for exponential sums may still be useful for the Hardy-Littlewood circle method towards the Hasse principle and the distribution of integral and rational solutions of Diophantine equations.  

As seen above, the motivic oscillation index $\moi_\QQ(f)$ is an important arithmetic threshold of $f$. Therefore, understanding $\moi_\QQ(f)$ is needed. Note that $-\moi_\QQ(f)$ is the real part of a pole of a certain Igusa local zeta function associated to $f-a$ for a critical value $a$ of $f$ (see \cite{CMN,NguyenVeys}). Igusa proposed the strong monodromy conjecture to better understand the poles of local zeta functions in terms of the geometric invariants of singularities (see \cite{DenefBour}).  More precisely, he expected that the real part of every pole of every Igusa local zeta function associated to $f$ is a root of the Bernstein-Sato polynomial $b_f(s)$ of $f$ (see Section \ref{BM} for the definition of $b_f(s)$). This means that the geometric side of singularities of $f$ provides many important information about the arithmetic side of singularities of $f$ and vice versa.  For instance,  motivated by the strong monodromy conjecture, if $f$ has $0$ as its only critical value, then we may conjecture that $-\moi_\QQ(f)$ is the largest solution $-\tilde{\alpha}_f$ of the reduced Bernstein-Sato polynomial $\tilde{b}_f(s)=b_f(s)/(s+1)$ of $f$. In fact, skipping the part of the above story related to the Hasse principle and concentrating on the mysterious relation between the arithmetic side and the geometric side of singularities of $f$ is the main aspect of this paper.

Now, it is natural to ask whether we can generalize the above story from a single polynomial $f$ (thus with its associated scheme $\Spec(\ZZ[x]/(f))$) to the case of the ideal $\cI\subset \ZZ[x_1,\dots,x_n]$ generated by non-constant polynomials $f_1,\dots,f_r$ and its associated scheme $X=\Spec(\ZZ[x]/\cI)$. In fact, many parts of this question were treated.  Firstly, note that the Hardy-Littlewood circle method towards the Hasse principle and the distribution of integral points uses an arbitrary $X$ as above (see \cite{Birch, Browning-HB}). In particular, the local density $\mathfrak{d}_p(X)$ can be defined as well. Secondly, one can define the Igusa local zeta function $\cZ_{p,\cI}(s)$ of $\cI$ over $\QQ_p$ by
$$\cZ_{p,\cI}(s):=\int_{\ZZ_p^n}p^{-s\ord_p(\cI(x))}\left|dx\right|,$$
where $\ord_p(\cI(x)):=\min_{1\leq i\leq r}\ord_p(f_i(x))=\min_{f\in\cI}\ord_p(f(x))$ depends only on $\cI$ (see \cite{VeysZ}). Thirdly, Budur, Musta\c{t}\v{a} and Saito introduced the Bernstein-Sato polynomial $b_X(s)$ associated to $X_\QQ$ in \cite{BMS} (see Section \ref{BM}). Of course, one may discuss about the strong monodromy conjecture for poles of $\cZ_{p,\cI}(s)$ and roots of $b_X(s)$. 

In another direction of generalization,  Loeser introduced the notion of multiple variables  Igusa local zeta function associated to the morphism $F=(f_1,\dots,f_r):\AA_\QQ^n\to\AA_\QQ^r$ (see \cite{Loesermulti}). On the other hand, some first attempts toward exponential sums and a Poisson summation formula of Siegel-Weil type associated to $F$ were presented in the work of Cluckers (see \cite{Cluckers-multi}) and the work of Lichtin (see \cite{LichBon,Lichtin-Ig1,Lichtin-Ig2,Lichtin-Ig3}).

\subsection{The theoretical goals and ideas}
\subsubsection{Exponential sums and motivic oscillation indexes of an arbitrary ideal}
One of the main points of this paper is to introduce the notions of exponential sums and motivic oscillation indexes of an arbitrary ideal of the ring of polynomials with integral coefficients. The key to dealing with this is Lemma \ref{eqzeta} where we establish a link between the exponential sums of such an ideal and its associated Igusa local zeta function. In particular, we can compute the exponential sums of ideals by using log resolutions defined in Section \ref{SJ}. Our work may correspond to the work of Budur, Musta\c{t}\v{a} and Saito in \cite{BMS} via the philosophy of the strong monodromy conjecture. Independently of this paper, a recent work by Chen, Dirks, Musta\c{t}\v{a} and Olano used quite an analogous idea to define the minimal exponent of locally complete intersection singularities (see \cite{CDMO}). The reader can also relate this paper to the work of Musta\c{t}\v{a} and Popa in \cite{MuPominimal}. 

Now, let us explain our ideas in this paper. Let $\cI$ be a non-zero ideal of $\ZZ[x_1,\dots,x_n]$. Suppose that $\cI$ does not contain any non-zero constant.  In a naive way, we can define the motivic oscillation index of $\cI$ by generalizing the work in \cite{CMN,NguyenVeys} as follows.  Let $p$ be a prime and $L$ be a finite extension of $\QQ_p$, then we denote by $\sigma_L^{\textnormal{naive}}(\cI)$ the minimum taken over all $\sigma\in\RR\cup\{+\infty\}$ such that either $\sigma=+\infty$ or $-\sigma$ is the real part of a pole of the function 
$$\cG_{L,\cI}(s)=\left(1-q_L^{-(s+r)}\right)\int_{\cO_L^n}q_L^{-s\ord_L(\cI(x))}\left|dx\right|,$$
where $\cO_L$ is the ring of integers in $L$, $q_L$ is the cardinality of the residue field $k_L$ of $L$, $\ord_L:L\to \ZZ\cup \{+\infty\}$ is the valuation map of $L$, $\ord_L(\cI(x))=\min_{f\in \cI}\ord_L(f(x))$ and $r=n-\dim(\spec(\QQ[x_1,\dots,x_n]/\cI\otimes\QQ))$. Then the naive motivic oscillation index $\moi_{\QQ}^{\textnormal{naive}}(\cI)$ of $\cI$ over $\QQ$ is given by 
$$\moi_{\QQ}^{{\textnormal{naive}}}(\cI)=\liminf_{p\to +\infty, \QQ_p\subset L}\sigma_L^{\textnormal{naive}}(\cI).$$
However, the value $\moi_{\QQ}^{\textnormal{naive}}(\cI)$ does not provide enough important 
information about 
the minimal number of generators of $\cI\otimes\QQ$. For instance, let $\cI$ such that $\cI\otimes\QQ$ defines a smooth affine variety $X\subset\AA_\QQ^n$ (e.g., $X$ is a hyperplane) then it is easy to show that $\moi_{\QQ}^{{\textnormal{naive}}}(\cI)=+\infty$. But it is well-known that $\cI\otimes\QQ$ is not necessary to be generated by $n-\dim(X)$ generators (see, e.g, \cite{Kuma}). That is why we want to modify our definition.

Now, we suppose that $\cI\otimes \QQ$ is generated by non-constant polynomials $f_1,\dots,f_r\in\ZZ[x_1,\dots,x_n]$.  Motivated by the definition of the singular series $\mathfrak{S}(f_1,\dots,f_r)$ appearing in \cite{Birch} (see also Definition \ref{defsing}), we associate to a prime $p$, a positive integer $m$ and polynomials $f_1,\dots,f_r$  the exponential sum
\begin{align*}
E_{f_1,\dots,f_r}(p,m)&=\int_{(y,x)\in (\ZZ_p^r\setminus p\ZZ_p^r)\times \ZZ_p^n}\exp\left(\frac{2\pi i\sum_{1\leq i\leq r}y_if_i(x)}{p^m}\right)\left|dy\wedge dx\right|\\
&=\frac{1}{p^{-m(n+r)}}\sum_{\overline{y}\in (\ZZ/p^m\ZZ)^r\setminus (p\ZZ/p^m\ZZ)^r, \overline{x}\in (\ZZ/p^m\ZZ)^n}\exp\left(\frac{2\pi i\sum_{1\leq i\leq r}y_if_i(x)}{p^m}\right).
\end{align*}
Note that this exponential sum depends only on $p,m$, the ideal $\cI$ and the number $r$ when $p$ is large enough as shown in Section \ref{exponential sum}. Thus, it is more convenient to write  $E_{\cI}^{(r)}(p,m)$ instead of  $E_{f_1,\dots,f_r}(p,m)$ and call this the $r^{\textnormal{th}}$ exponential sum modulo $p^m$ of $\cI$.
\begin{defn} The $r^{\textnormal{th}}$-motivic oscillation index  of the ideal $\cI$ over $\QQ$, denoted by $\moi_{\QQ}^{(r)}(\cI)$, is the supremum of all real numbers $\sigma$ such that for each large enough prime $p$, there is a constant $c_{\sigma,p}>0$ satisfying $|E_{\cI}^{(r)}(p,m)|\leq c_{\sigma,p}p^{-\sigma m}$ for all $m\geq 1$.
\end{defn}
In the sense of \cite{NguyenVeys}, $E_{\cI}^{(r)}(p,m)$ is the exponential sum modulo $p^m$ associated to the scheme $\mathfrak{D}_\ZZ^{n,r}:=\AA_\ZZ^n\times(\AA_\ZZ^r\setminus\spec(\ZZ[y_1,...,y_r]/(y_1,...,y_r)))$ and the polynomial $g(x,y):=\sum_{1\leq i\leq r}y_if_i(x)$ while $\moi_{\QQ}^{(r)}(\cI)$ agrees with the motivic oscillation index $\moi_{\mathfrak{D}_\ZZ^{n,r}, \QQ}(g(x,y))$ of $g(x,y)$ at the scheme $\mathfrak{D}_\ZZ^{n,r}$ over $\QQ$. We will define the motivic oscillation index $\moi_{\QQ}(\cI)$ of $\cI$ over $\QQ$ to be the $r^{\textnormal{th}}$-motivic oscillation index $\moi_{\QQ}^{(r)}(\cI)$  if moreover $\cI\otimes \QQ$  can not be generated by $r-1$ polynomials (see Lemma-Definition \ref{moidef}). Similarly, let $Z\subset \AA_\ZZ^n$ be a $\ZZ$-scheme of finite type, we can define the motivic oscillation index $\moi_{\QQ, Z}(\cI)$ of $\cI$ at $Z$ over $\QQ$ by using $Z$ instead of $\AA_\ZZ^n$ in the above definition. Suppose that $\cI\otimes \QQ$ can not be generated by $r'$ elements for some $r'<r$. It may happen that for an affine neighbourhood $U$ of $Z_\QQ$, $\cI\otimes \QQ|_U$ can be generated by $r'$ elements. Thus, it is better to have the notions of $r'^{\textnormal{th}}$-exponential sum modulo $p^m$ of $\cI$ at $Z$ and local motivic oscillation index $\moi_{\QQ, Z}^{\loc}(\cI)$ of $\cI$ at $Z$ over $\QQ$ (see Definitions \ref{moiloc1}, \ref{expodef}). More generally, if $X$ is a $\ZZ$-scheme of finite type, $P\in X(\ZZ)$ and $\tau$ is a closed embedding of an affine neighbourhood $W$ of $P$ in $X$ to an affine space $\AA_\ZZ^n$, we can define the local motivic oscillation index of $X$ at $P$ over $\QQ$ to be the local motivic oscillation index of $\tau(W)$ at $\tau(P)$ over $\QQ$ . However this definition always depends on $\tau$. To remove this dependence, we introduce the notion of absolutely local motivic oscillation index $\moi_{\QQ, P}^{\textnormal{aloc}}(X)$ of $X$ at $P$ over $\QQ$ (see Definition \ref{moiloc2}). For instance, if $X=\spec(\ZZ[x_1,\dots,x_n]/\cI)$  then $\moi_{\QQ, P}^{\textnormal{aloc}}(\cI)=\moi_{\QQ, P}^{\loc}(\cI)-n$. All details about these definitions will be given in Section \ref{exponential sum}. 

\subsubsection{Averaged Igusa conjecture for exponential sums}For the exponential sums and the motivic oscillation indexes of ideals, we formulate the following conjecture.
\begin{conj}[Averaged Igusa conjecture for exponential sums]\label{aveIgu}Let $\cI$ be a non-zero ideal of $\ZZ[x_1,\dots,x_n]$ such that $\cI\otimes \QQ\neq (1)$. Suppose that $\cI\otimes \QQ$ is generated by polynomials $f_1,\dots,f_r$. For each $\sigma<\moi_\QQ^{(r)}(\cI)$, there is a constant $c(\sigma)$ depending only on $\cI,r$ and $\sigma$ such that 
$$\left|E_{\cI}^{(r)}(p,m)\right|\leq c(\sigma)p^{-m\sigma}$$
provided that $p$ is large enough and $m\geq 2$.
\end{conj}
This conjecture is a special case of Igusa's conjecture for exponential sums stated in \cite[page 2]{CMN} and \cite[Conjecture 1.1]{NguyenVeys} (see also Conjecture \ref{Iguconj}) since we are working with exponential sums associated to the special scheme $\mathfrak{D}_\ZZ^{n,r}$ and polynomials of the special form $g(x,y)=\sum_{i=1}^ry_if_i(x)$. In Section \ref{Conjexp}, we will formulate a strong and general version of Conjecture \ref{aveIgu} (see Conjectures \ref{avelocIgu}). Moreover, motivated by the work in \cite{Browning-HB}, \cite{MustPopa}  and \cite{CluckerNguyen}, a variant of Conjecture \ref{avelocIgu} will also be introduced in Section \ref{Conjexp} (see Conjecture \ref{conjeff}).
\subsubsection{Some properties of the motivic oscillation indexes of ideals}
Notice that the most important part of Conjecture \ref{aveIgu} concerns the situation where $\moi_\QQ^{(r)}(\cI)=\moi_\QQ(\cI)$. Understanding $\moi_\QQ(\cI)$ and Conjecture $\ref{aveIgu}$ better is one of the main goals of this paper. We will show some important properties of $\moi_\QQ(\cI)$. The fact is that the motivic oscillation index $\moi_{\QQ}(\cI)$ will provide some important geometric properties of the scheme $X=\spec(\ZZ[x_1,\dots,x_n]/\cI)$. In addition, some important properties of the Igusa local zeta functions of $\cI$ are also captured by $\moi_{\QQ}(\cI)$. These things are demonstrated by the following theorem which is the version of higher codimension of \cite[Proposition 3.10]{CMN}.
\begin{named}{Theorem A}\label{pole1}Let $\cI$ be the ideal of $\ZZ[x_1,\dots,x_n]$ generated by non-zero polynomials $f_1,\dots,f_r$ such that $\cI\otimes \QQ\neq (1)$. Let $X=\spec(\ZZ[x_1,\dots,x_n]/\cI)$. The following assertions hold:
\begin{itemize}
\item[i,] $\moi_{\QQ}(\cI)>r$ if and only if  $\moi_{\QQ}(\cI)=\moi_{\QQ}^{(r)}(\cI)>r$ if and only if the scheme $X_\QQ$ is a complete intersection of codimension $r$ in $\AA^n_\QQ$ and has only rational singularities.
\item[ii,] $\moi_{\QQ}(\cI)\leq r$ if and only if $\moi_{\QQ}^{(r)}(\cI)$ equals to $\lct(\cI\otimes\QQ)$ the log canonical threshold of $\cI\otimes\QQ$ (see Section \ref{SJ} for the definition of log canonical threshold).
\item[iii,] If $\moi_{\QQ}(\cI)=r$ then either $\moi_{\QQ}^{(r)}(\cI)=\moi_{\QQ}(\cI)$ or $X_\QQ$ is a complete intersection in $\AA^n_\QQ$ and has only rational singularities.
\end{itemize} 
\end{named}
\begin{cor}\label{corpol} With the notation and assumption of \ref{pole1}, if we have $\moi_{\QQ}(\cI)>r$, then for all non-Archimedean local fields $L$ of characteristic zero, the following claims hold:
\begin{itemize}
\item[(\textit{i}),] $s_0=-r$ is a pole of multiplicity at most $1$ of the Igusa local zeta function 
$$\cZ_{L,\cI}(s)=\int_{\cO_L^n}q_L^{-s\ord_L\left(\cI(x)\right)}\left|dx\right|,$$
\item[(\textit{ii}),]if $s_0$ is a pole of $\cZ_{L,\cI}(s)$ and $\Re(s_0)\neq -r$, then $\Re(s_0)<-r$.
\end{itemize}
\end{cor}

Note that the first assertion of \ref{pole1} is analogous to \cite[Corollary 1.7]{CDMO} via the philosophy of the strong monodromy conjecture. Corollary \ref{corpol} and many parts of \ref{pole1} will be proved in Proposition \ref{equcondi}. 
In Section \ref{property}, we will generalize \ref{pole1} by using the local motivic oscillation indexes of ideals (see \ref{locpole}).
\subsection{Some applications}
We introduce some first applications of the motivic osicllation indexes of ideals and Conjecture $\ref{aveIgu}$.
\subsubsection{Counting points of schemes over finite rings and the arithmetic aspect of singularities}
 Suppose that Conjecture \ref{aveIgu} holds, the following theorem enables us to improve the result in \cite[Theorem A]{AizenAvni} on counting points over finite rings of complete intersections having only rational singularities.
\begin{named}{Theorem B}\label{count}
Let $\cI$ be the ideal of $\ZZ[x_1,\dots,x_n]$ generated by non-zero polynomials $f_1,\dots,f_r$ such that $\cI\otimes \QQ\neq (1)$. Let $X=\spec(\ZZ[x_1,...,x_n]/\cI)$. Suppose that $X_{\CC}$ is a complete intersection of dimension $n-r$ in $\AA_{\CC}^n$ and has only rational singularities. Let $\sigma_0\in (r,+\infty]$. If Conjecture \ref{aveIgu} holds for $\cI,r$ and all $\sigma<\sigma_0$, then there is a positive constant $C$ such that for all non-Archimedean local fields $L$ of large enough residue field characteristic and all positive integers $m$ we have 
\begin{equation}\label{abcount}
\left|\frac{1}{q_L^{m(n-r)}}\#X(\cO_L/(\varpi_L^m))-\frac{1}{q_L^{n-r}}\# X(\cO_L/(\varpi_L))\right|\leq Cq_L^{-\lceil 2(\sigma_0-r)\rceil},
\end{equation}
where $\varpi_L$ is an arbitrary uniformizing parameter of $\cO_L$ and $\lceil a\rceil=\min \{\ell\in \ZZ, \ell\geq a\}$.
\end{named}
By Conjecture \ref{aveIgu}, we expect that one can take $\sigma_0=\moi_\QQ^{(r)}(\cI)$ in \ref{count}. However, Conjecture \ref{aveIgu} has not been proved yet, thus any progress on increasing $\sigma_0$ would improve the counting points of $X$ over finite rings. \ref{count} is also an improved absolute analogue of \cite[Theorem 4.11]{C-G-H} in which they studied the notion of $E$-smooth morphisms between $\QQ$-schemes (see Definition \ref{esm}) and a slight modification of their argument probably implies that the structure morphism from $X_\QQ$ to $\spec(\QQ)$ is $E$-smooth if and only if we can replace $\lceil 2(\sigma_0-r)\rceil$ in the right hand side of (\ref{abcount}) by $E$. In other words, \ref{count} helps to show that the structure morphism $X_\QQ\to \spec(\QQ)$ is $\lceil 2(\sigma_0-r)\rceil$-smooth whenever $\sigma_0>r$. In Section \ref{uniadd}, we will introduce a more general version of \ref{count} for locally complete intersections by using the local motivic oscillation indexes of ideals (see \ref{locCounting}). Moreover,  by using the proof of \cite[Thm 4.11]{C-G-H}, we have a version of \ref{locCounting} for flat families of locally complete intersections having only rational singularities (see Theorem \ref{improE}).

In \cite{AizenAvni}, Aizenbud and Avni characterized rational singularities by counting points over finite rings. In our approach, \ref{locpole} also allows us to have an arithmetic characterization in terms of the local motivic oscillation indexes of ideals for the property of a scheme being a locally complete intersection having only rational singularities or log canonical singularities.  Moreover, by combining \ref{locCounting} with the result in \cite{C-G-H}, we also have an arithmetic criterion in terms of the local motivic oscillation indexes of ideals for the property of a scheme being a locally complete intersection having only terminal singularities (see  Proposition \ref{FRST}). On the other hand, rational, log canonical and terminal singularities  were characterized in terms of geometric properties of jet schemes in \cite{Mustata1,EMY,Ein-Must}. Recently, the work of  Musta\c{t}\v{a} and Popa in \cite{Must-Pop-k} together with the work of Chen, Dirks and Musta\c{t}\v{a} in \cite{CDM-k} characterized $e$-rational singularities and $e$-Du Bois singularities in terms of the (local) minimal exponents for the case of locally complete intersections (see similar results in \cite{JKSY-k,FL-k1,FL-k2}). However, the work of  Aizenbud and Avni seems to be hard to extend to the case of $e$-rational/$e$-Du Bois singularities. Nevertheless, by the philosophy of the strong monodromy conjecture, the local motivic oscillation indexes of ideals should  be equal to the local minimal exponents of the corresponding schemes. Thus, the new perspective of this paper is expected to give an arithmetic characterization of $e$-rational singularities and $e$-Du Bois singularities (see Conjecture \ref{k-rati} in Section \ref{OP}).    

In order to use \ref{count}, \ref{locCounting} and Theorem \ref{improE}, we need to find a large enough number $\sigma_0$ and prove Conjecture \ref{aveIgu} for all real numbers $\sigma<\sigma_0$. When $\cI$ is a principal ideal generated by a non constant polynomial $f$ of degree $d>1$, the result in \cite{Nguyennsd} implies Conjecture \ref{aveIgu} with   $$\sigma<\sigma_0=\tilde{\sigma}_0(f):=\frac{n-s(f_d)}{2(d-1)},$$
where $f_d$ is the homogeneous part of highest degree of $f$ and $s(f_d)$ is the dimension of the singular locus of $f_d$ in $\AA^n$. In this paper, we are able to generalize this result to a very general setting. More precisely, let $w=(w_1,\dots,w_n)\in\NN_{\geq 1}^n$ and $f$ be a polynomial then we will write $f=f_{0w}+...+f_{dw}$, where $f_{iw}$ is the $w$-weighted homogeneous part of $w$-degree $i$ of $f$ and $d=\deg_w(f)$ is the $w$-degree of $f$ (see Section \ref{modp}). Let $J$ be a non-empty finite subset of $\NN_{\geq 1}$. For each $i\in J$, let $r_i$ be a positive integer. For each $1\leq j\leq r_i$, let $f_{ij}(x_1,\dots,x_n)$ be a non-constant polynomial. We put $r=\sum_{i\in J} r_i$. Suppose that $f_{ij}$ is of $w$-degree $i$ and $F_{ijw}$ is the $w$-weighted homogeneous part of highest $w$-degree of $f_{ij}$. For each $i\in J$, let $s_{wi}$ be the dimension of the Birch singular locus
$$\BSing(F_{i1w},\dots,F_{ir_iw})=\{x\in \CC^n|\rank\left(\left(\frac{\partial F_{ijw}}{\partial x_\ell}(x)\right)_{1\leq j\leq r_i, 1\leq \ell\leq n}\right)<r_i\}.$$
We set 
$$\tilde{\sigma}_{0w}\big(\left(f_{ij}\right)_{i\in J, 1\leq j\leq r_i}\big)=\min_{i\in J}\frac{n-s_{wi}}{2(i-1)},$$
with the convention  that $n/0=+\infty$ and $0/0=0$.
In Section \ref{exsums}, we will prove the following theorem.
\begin{named}{Theorem C}\label{wnsd}
Let $w=(w_1,...,w_n)\in \NN_{\geq 1}^n$. Let $J$ be a non-empty finite subset of $\NN_{\geq 1}$, let $(r_i)_{i\in J}\in \NN_{\geq 1}^J$ and put $r=\sum_{i\in J}r_i$. For each $i\in J$ and each $1\leq j\leq r_i$, let $f_{ij}(x_1,\dots,x_n)$ be a non-constant polynomial in $\ZZ[x_1,\dots,x_n]$ such that $\deg_w(f_{ij})=i$. We set $\cI=\sum_{i\in J, 1\leq j\leq r_i}(f_{ij})\subset \ZZ[x_1,\dots,x_n]$ and assume that $\cI\otimes \QQ\neq (1)$. With the above notation, we have 
$$\moi_{\QQ}^{(r)}(\cI)\geq \tilde{\sigma}_{0w}\big(\left(f_{ij}\right)_{i\in J, 1\leq j\leq r_i}\big).$$
In addition, if $\sigma<\tilde{\sigma}_{0w}\big(\left(f_{ij}\right)_{i\in J, 1\leq j\leq r_i}\big)$, then for each prime $p$ there is a constant $c(\sigma,\cI,r,p)$ depending only on $\sigma,\cI, r, p$ such that $c(\sigma,\cI,r,p)=1$ if $p$ is large enough and
\begin{equation}\label{thm}
\left|E_{\cI}^{(r)}(p,m)\right|\leq c(\sigma,\cI,r,p)p^{-m\sigma}
\end{equation}
for all primes $p$ and all $m\geq 1$.
\end{named}
To prove \ref{wnsd} for $w=(1,\dots,1)$, we also need some materials as in \cite{Nguyennsd} such as the transfer principle for exponential sums modulo $p^m$ in \cite{NguyenVeys} and bounds on exponential sums over finite fields in  \cite{CDenSperlocal,Katz}. However, the strategy in \cite{Nguyennsd} is no longer available for our new setting. In fact,  it turns out that we need to have a uniform bound on exponential sums over finite fields of polynomials sharing the weighted degree and the dimension of the singular locus of the weighted homogeneous part of highest weighted degree. We will introduce such a uniform bound in Section \ref{modp}. On the other hand, we also need to work with flexible weights of variables when using this uniform bound to relate $\tilde{\sigma}_{0w}\big(\left(f_{ij}\right)_{i\in J, 1\leq j\leq r_i}\big)$ to our bound for $\big|E_{\cI}^{(r)}(p,m)\big|$. 
 To play with an arbitrary weight $w=(w_1,\dots,w_n)$, we associate each polynomial $f_{ij}(x_1,\dots,x_n)$ with a polynomial $f_{ij}^{(w)}$ in the variables $x_{e\ell}$ for $1\leq e\leq n, 1\leq \ell\leq w_e$ by the torus transformation in \cite[Theorem 7.4]{CDenSperlocal}, namely, $f_{ij}^{(w)}=f_{ij}(x_{11}\cdot...\cdot x_{1w_1},\dots,x_{n1}\cdot...\cdot x_{nw_n})$. Then we try to relate the exponential sums of $\cI=\sum_{i\in J, 1\leq j\leq r_i}(f_{ij})$ to the exponential sums of $\cI_w=\sum_{i\in J, 1\leq j\leq r_i}(f_{ij}^{(w)})$ and run the above process for  the weight $w^{(0)}=(1,\dots,1)\in\ZZ_{>0}^{w_1+...+w_n}$ and polynomials $f_{ij}^{(w)}$ for $i\in J, 1\leq j\leq r_i$. Further details will be provided in  Section \ref{exsums}. 

The following consequence of \ref{wnsd} and Theorem \ref{improE} will give a uniform version for counting solutions of systems of congruence equations. 
\begin{cor}\label{countinguni}Let $J, (r_i)_{i\in J}, (f_{ij})_{i\in J, 1\leq j\leq r_i}, r, w, \tilde{\sigma}_{0w}\big(\left(f_{ij}\right)_{i\in J, 1\leq j\leq r_i}\big)$ be as in \ref{wnsd}. For each $\ell\in J$ and each $1\leq j\leq r_\ell$, we denote by $\cA_{n\ell j}$ the $\ZZ$-affine space of all polynomials in $n$ variables of $w$-degree at most $\ell-1$. We put $\cA=\prod_{\ell\in J, 1\leq j\leq r_\ell}\cA_{n\ell j}$. For each $m\geq 1$, let $\pi_m:\cA(\ZZ/p^m\ZZ)\to \cA(\ZZ/p\ZZ)$ be the map induced by the canonical map $\ZZ/p^m\ZZ\to \ZZ/p\ZZ$.   If $\tilde{\sigma}_{0w}\big(\left(f_{ij}\right)_{i\in J, 1\leq j\leq r_i}\big)>r$, then there is an integer $M$ and a constant $C$ such that for all $p>M$, all tuples $g=(g_{\ell j})_{\ell,j}\in \cA(\ZZ/p\ZZ)$ and all tuples $\tilde{g}=(\tilde{g}_{\ell j})_{\ell,j}\in A(\ZZ/p^m\ZZ)$, the condition $\pi_m(\tilde{g})=g$ implies
\begin{align*}
&\left|\frac{\#\{x\in (\ZZ/p^m\ZZ)^n| f_{\ell j}(x)=\tilde{g}_{\ell j}(x) \hspace{0.1cm}\forall \ell, j\}}{p^{m(n-r)}}-\frac{\#\{x\in (\ZZ/p\ZZ)^n| f_{\ell j}(x)=g_{\ell j}(x)\hspace{0.1cm}\forall \ell,j\}}{p^{n-r}}\right|\\ \leq& Cp^{-\lceil 2\left(\tilde{\sigma}_{0w}\left(\left(f_{ij}\right)_{i\in J, 1\leq j\leq r_i}\right)-r\right)\rceil}.
\end{align*}
\end{cor}

\subsubsection{FRS morphisms}Let $K$ be a field of characteristic zero. Let $\overline{K}$ be an algebraic closure of $K$. Recall that a morphism $F$ between $K$-schemes of finite  type is $FRS$ if it is flat and every non-empty fiber of $F$ has only rational singularities. As mentioned above,  Theorem \ref{improE} is a uniform version of \ref{locCounting} for fibers of $FRS$ morphisms. 

In concern to $FRS$ morphisms, Glazer and Hendel studied singularity properties  of convolutions of algebraic morphisms in \cite{G-H1, G-H2}. Let $(G,.)$ be an algebraic group and $X_1, X_2$ be varieties over $K$. If $\varphi_i: X_i \to G, 1\leq i\leq 2$ are morphisms, then the convolution $\varphi_1*\varphi_2: X_1\times X_2\to G$ is defined by $(x_1,x_2)\mapsto \varphi_1(x_1).\varphi_2(x_2)$. Glazer and Hendel showed in \cite{G-H1} that if $(G,.)=(\AA_K^r,+)$ and $X$ is a $K$-smooth variety together with a morphism $\varphi:X\to G$ such that $\varphi_{\overline{K}}(X_i)$ is not contained in any proper affine subspace of $G_{\overline{K}}$ for all irreducible components $X_i$ of $X_{\overline{K}}$, then the $\ell^{th}$ convolution power $$\varphi^{*\ell}=\underbrace{\varphi*...*\varphi}_{\ell \textnormal{ times}}$$ of $\varphi$ is an $FRS$ morphism if $\ell$ is large enough. In particular, for each $D\in\NN_{\geq r+1}$, there is a constant $N(D)$ such that if  $X_1,\dots,X_\ell$ are smooth varieties of complexity at most $D$ and $\varphi_i:X_i\to G$ is a strongly dominant morphism of complexity at most $D$ for all $1\leq i\leq \ell$, then $\varphi_1*\varphi_2*...*\varphi_\ell$ is an $FRS$ morphism if $\ell>N(D)$. Here, the notion of complexity of a $K$-variety $X$ or a $K$-morphism $\varphi$ of finite type arises when we describe $X$ or $\varphi$ by using polynomial functions on affine spaces, namely, it is bounded in terms of the number and the degree of these polynomial functions together with the number and the dimension of these affine spaces. In the same way, we can define the notion of degree complexity of $X$ or $\varphi$ if we do not restrict the number and the dimension of the affine spaces that we are using (see precise definition of (degree) complexity in Section \ref{uniadd} or a stronger definition in \cite{G-H1}). On the other hand, a morphism $\varphi:X\to Y$ is strongly dominant if the restriction of $\varphi_{\overline{K}}$ on each irreducible component of $X_{\overline{K}}$ is also dominant. 

Similarly, one can ask the same question about flat morphisms with geometrically irreducible fibers ($FGI$ morphism) or fibers having only terminal singularities ($FTS$ morphism) (see a special case in \cite{G-H3}).  Interestingly, for the situation of \cite{G-H1}, the local motivic oscillation indexes of ideals can help to make an explicit constant $N(D)$ and relax the strongly dominant condition of $\varphi_i$. Furthermore, we obtain a similar version with $FTS$  morphisms under the assumption that Conjecture \ref{avelocIgu} holds. In particular, by using the exponential sums of ideals, we also obtain a similar result for $FGI$ morphisms. Let us formulate the results. 
\begin{named}{Theorem D}\label{FRS}
Let $K$ be a field of characteristic zero . Let $r,R,D$ be positive integers. Let $X_1,\dots,X_\ell$ be non-empty $K$-varieties. For each $1\leq i\leq \ell$ we take a morphism $\varphi_i:X_i\to \AA_K^r$ such that $\varphi_{i\overline{K}}|_{X_{ij}}$ is not contained in any proper affine subspace of $\AA_{\overline{K}}^r$ for every $i$ and every irreducible component $X_{ij}$ of $X_{i}\otimes\overline{K}$. The following claims hold:
\begin{itemize}
\item[(\textit{i}),]If $X_{i}$ is geometrically irreducible and a locally complete intersection for all $i$, then the convolution $\varphi_1*\varphi_2*...*\varphi_\ell$ is an $FGI$ morphism if $\ell>2r$. 

\item[(\textit{ii}),]Suppose that $X_i$ is smooth and  $X_i, \varphi_i$ are of degree complexity at most $(R,D)$ for all $1\leq i\leq \ell$. Then $\varphi_1*\varphi_2*...*\varphi_\ell$ is an $FRS$ morphism if $\ell>N(r,R,D)=2r(D^{R+1}-1)$. In addition, suppose that Conjecture \ref{avelocIgu} holds, then $\varphi_1*\varphi_2*...*\varphi_\ell$ is an $FTS$ morphism if $\ell>N'(r,R,D)=2(r+\frac{1}{2})(D^{R+1}-1)$. In particular, if $X_i$ is an affine space for each $i\geq 1$, then we can take $N(r,R,D)=rD, N'(r,R,D)=(r+\frac{1}{2})D$.
\end{itemize}
\end{named}
By \cite[Theorem B]{G-H2} and \cite[Lemma 3.8]{G-H3}, $\varphi_1*\varphi_2*...*\varphi_\ell$ is surjective if $\ell>2r$, but it has not been proved that this morphism is $FGI$ yet. On the other hand, if $X_i=\AA_K^n$ for al $i$, then a similar statement for $FRS$ morphisms of \ref{FRS} was claimed in \cite[Proposition 5.7]{G-H3}. However, it is not clear how to relax the smoothness of $X_i$ in \ref{FRS}. But we will provide  some rough ideas in Open Question \ref{non-smooth}. We also emphasize that the dominant term $rD^{R+1}$ of $N(r,R,D)$ in the statement of \ref{FRS} is optimal. Indeed, we can look at the following example. 
\begin{example}Let $n\geq(R+1)r$. Let $X_i=X$ for all $i$ and $X$ be the subvariety of $\AA_K^n$ associated to  the ideal generated by $x_{11}-x_{12}^D, x_{12}-x_{13}^D,\dots,x_{1R}-x_{1,R+1}^D,\dots,x_{r1}-x_{r2}^D,\dots,x_{rR}-x_{r,R+1}^D$. Let $g(x)=(x_{11}^D,\dots,x_{r1}^D): X\to \AA_K^r$. Then it is easy to verify that $g^{*\ell}$ is an $FRS$ morphism if and only if $\ell>rD^{R+1}$.
\end{example}
However, the notion of (degree) complexity is not an intrinsic notion of varieties and morphisms, as it depends on how we use polynomials to describe varieties and morphisms. In particular, in the situation of \ref{FRS}, given a family $(X_i,\varphi_i)_{i\geq 1}$, the constants $N(r,R,D), N'(r,R,D)$ can be very large due to such a bad description of $(X_i,\varphi_i)_{i\geq 1}$ by polynomials. Thus, we may ask to replace $N(r,R,D), N'(r,R,D)$ by invariants of $(X_i,\varphi_i)_{i\geq 1}$.
The recent result of  Glazer, Hendel and Sodin in \cite[Corollary 1.9]{GHS} states that if $\varphi$ is a dominant morphism  from an irreducible smooth variety $X$ to a connected algebraic group $G$ then there is an invariant $N(X,\varphi)$ depending only on the isomorphism class of $(X,\varphi)$ in the category of $K$-varieties over $G$  such that $\varphi^{*\ell}$ is an $FRS$-morphism whenever $\ell>N(X,\varphi)$. Here, $N(X,\varphi)$ is computed in terms of the log canonical threshold of the Jacobian ideal of $\varphi$ (see Section \ref{SJ}). In our approach, we may also expect an analogous result of \cite[Corollary 1.9]{GHS} such that $N(X,\varphi)$ is replaced  by another invariant that can be computed in terms of the motivic oscillation index of fibers of $\varphi$. We leave it to future research.

In order to prove the first part of \ref{FRS}, we use a criterion for the geometric irreducibility of an affine variety in terms of the exponential sums modulo $p$ of the corresponding ideal (see Proposition \ref{delta} below). This leads us to estimate the exponential sums modulo $p$ of ideals defining locally the fibers of $\varphi_1*\varphi_2*...*\varphi_\ell$. The main idea to prove the second part of \ref{FRS} is to show that the local motivic oscillation indexes of fibers of $\varphi_1*\varphi_2*...*\varphi_\ell$ will tend to infinity if $\ell$ tends to infinity.

\subsubsection{Some Waring type problems}
From the point of view of additive combinatorics, the Waring problem can be expressed as follows. Let $(G,\cdot)$ be a group or a semi group. Let $A_1,\dots,A_\ell$ be non-empty subsets of $G$. The set $A_1\cdot...\cdot A_\ell$ is the set of all elements of $G$ in the form $a_1\cdot...\cdot a_\ell$ with $a_i\in A_i$ for all $i$. If $A=A_1=...=A_\ell$, we write simple $A^{.\ell}$ instead of $A_1\cdot...\cdot A_\ell$. If $(G,\cdot)$ is either an abelian group or a semi abelian group, we use the notations $(G,+)$ and $\ell A$ instead of $(G,\cdot)$ and $A^{.\ell}$. Let $A$ be a subset of $G$, then the Waring problem for $A$ is to determine the minimal positive integer $\ell$ such that $A^{.\ell}=G$. For instance, the original Waring problem works with the case $(G,\cdot)=(\NN,+)$ and $A=\{x^n| x\in \NN\}$. In this paper, let $(G,\cdot)$ be a group scheme of finite type over $\ZZ$ and $A_1,\dots,A_\ell\subset G$ be definable subsets in the language of rings. We aim to ask which conditions on $G,\ell$ and $A_1,\dots,A_\ell$ are sufficient to have $G(\ZZ/p^m\ZZ)=A_1(\ZZ/p^m\ZZ)\cdot...\cdot A_\ell(\ZZ/p^m\ZZ)$ for all large enough primes $p$ and  all $m\geq 1$. Note that if $G=\AA_\ZZ^1$ and $\sum_{1\leq i\leq \ell}\#A_i(\ZZ/p\ZZ)\geq \ell+p-1$ then $A_1(\ZZ/p\ZZ)+...+A_\ell(\ZZ/p\ZZ)=G(\ZZ/p\ZZ)$ (see \cite[Thm 2.3]{Nathan}). More generally, since $G(\ZZ/p^m\ZZ)$ is a finite group, one may try to give a condition on $\#G(\ZZ/p^m\ZZ), \#A_i(\ZZ/p^m\ZZ)$ and $\ell$ to have $G(\ZZ/p^m\ZZ)=A_1(\ZZ/p^m\ZZ)\cdot...\cdot A_\ell(\ZZ/p^m\ZZ)$. However, such a condition in terms of $\#G(\ZZ/p^m\ZZ), \#A_i(\ZZ/p^m\ZZ), \ell$ will probably depend on $p$ and $m$. In addition, such a condition is not effective if $A_1,\dots,A_\ell$ possess extra geometric properties. For instance, let $G=\AA_\ZZ^r$ and $A_i$ be the affine line of $G$ associated to the ideal $(x_1,..., x_{i-1},x_{i+1},...,x_r)$ then the claim $A_1(\ZZ/p^m\ZZ)+...+A_r(\ZZ/p^m\ZZ)=G(\ZZ/p^m\ZZ)$ holds independent of $p,m$. Thus, we may expect that if $A_1,\dots,A_\ell$ have good enough geometric properties, then one should get a geometric answer to this question. In Section \ref{OP}, we will formulate  a more detailed question about this problem (see Open Question \ref{defi}). 

In this paper, we also work with the  probabilistic Waring type problem (see \cite{G-H3}). Here, we recall some notation. If $A$ is a non-empty finite set then we denote by $U_A$ the uniform distribution of $A$, i.e. for each $x\in A$, $U_A(x)=\frac{1}{\# A}$ is the probability that $x=y$ when $y\in A$. If $f:A\to B$ is a map between two non-empty finite sets then  $f_*(U_{A})$ is the distribution of $B$ assigning $y\in B$ to the probability that $f(x)=y$, i.e., $f_*(U_{A})(y)=\frac{\#f^{-1}(y)}{\#A}$.
If $A$ is a non-empty finite set then the usual $L^a$ norm ($a\geq 1$) for functions on $A$ is given by
$$\|f\|_{L^a}=\bigg(\#A^{a-1}\sum_{x\in A}\big|f(x)\big|^a\bigg)^{1/a}$$
if $a<+\infty$ and  
 $$\|f\|_{L^{a}}=\#A \max_{x\in A}\left|f(x)\right|$$
 if $a=+\infty$. Let $X$ and $G$ be $\ZZ$-schemes of finite type such that $X_\QQ, G_\QQ$ are smooth and geometrically irreducible. Let $f:X\to G$ be a $\ZZ$-morphism. Then for each prime $p$ and each positive integer $m$, $f$ induces a map $f_{p,m}:X_{p,m}=X(\ZZ/p^m\ZZ)\to G_{p,m}=G(\ZZ/p^m\ZZ)$. One asks the probabilistic Waring type problem for $f$. Namely, when do we have
$$\lim_{p\to+\infty}\sup_{m}\|U_{G_{p,m}}-(f_{p,m})_*(U_{X_{p,m}})\|_{L^a}=0?$$

Now, let $G=\AA_\ZZ^r$ for some positive integer $r$. Let $\ell\geq 1$ and $X_1,\dots,X_\ell$ be $\ZZ$-schemes of finite type such that $X_{i}\otimes\QQ\neq \emptyset$ for all $i$. For each $1\leq i\leq \ell$, let $\varphi_i:X_i\to\AA_\ZZ^r$ be a  $\ZZ$-morphism. We put $A_i=\varphi_i(X_i)$ and try to find a certain condition of $X_i$ and $\ell$ such that for all large enough primes $p$ and all $m\geq 1$, one has $(\ZZ/p^m\ZZ)^r=\sum_{1\leq i\leq \ell}A_i(\ZZ/p^m\ZZ)$. On the other hand, we study the probabilistic Waring type problem for the convolution $f:=\varphi_1*...*\varphi_\ell$. In our setting, by using \ref{FRS}, we obtain a result as follows.
\begin{named}{Theorem E}\label{adduni}Let $\ell$ be a positive integer. For each $1\leq i\leq\ell$, let $X_i$ be a scheme over $\ZZ$ of finite type  and  $\varphi_i:X_i\to G=\AA_\ZZ^r$ be a morphism of $\ZZ$-schemes. We set $A_i=\varphi_i(X_i)$ as above.  We suppose that $X_{i}\otimes\QQ$ is geometrically irreducible and  $A_i(\CC)$ is not contained in any proper affine subspace of $\AA_\CC^r$. If $\ell>2r$, then there is an integer $M$ depending on $X_1,\dots,X_\ell, \varphi_1,\dots,\varphi_\ell$ such that 
$$(\ZZ/p^m\ZZ)^r=\sum_{1\leq i\leq \ell}A_i(\ZZ/p^m\ZZ)$$
for all $p>M$ and all $m\geq 1$. 

Additionally, if we assume moreover that $X_{i}\otimes\QQ$ is smooth of degree complexity at most $(R,D)$ for all $i$ and $\ell>N(r,R,D)=2r(D^{R+1}-1)$, then 
$$\lim_{p\to+\infty}\sup_{m}\|U_{G_{p,m}}-(f_{p,m})_*(U_{X_{p,m}})\|_{L^{\infty}}=0,$$
where $f=\varphi_1*...*\varphi_\ell: X=X_1\times...\times X_\ell\to G$.
\end{named}
Note that \ref{adduni} is of quite similar form to the results in \cite{LST1,LST} which deal with some Waring type problems of simple finite groups. If $\ell>2r$ then $\varphi_{1\overline{\QQ}}*\varphi_{2\overline{\QQ}}*...*\varphi_{\ell\overline{\QQ}}$ is surjective by using \cite[Theorem B]{G-H2} and \cite[Lemma 3.8]{G-H3} as mentioned above. However, it is not clear how to obtain the first part of \ref{adduni} for $m>1$ from this fact. While we can use  Hensel's lemma and the claim of \ref{FRS} for $FGI$-morphisms to prove this part. On the other hand, the second part of \ref{adduni} will follow immediately from \cite[Theorem 9.3 (ii)]{G-H3} and \ref{FRS}.



\section{Preliminaries}\label{nota}
This section will give an overview of the notations and concepts used in this paper.
\subsection{Non-Archimedean local fields}\label{nuloc}Let $K$ be a number field and $\cO_K$ be its ring of integers. In this paper, a \index{local field}local field $L$ over $\cO_K$ is a finite extension of $\QQ_p$ or $\FF_p((t))$ which is endowed with a ring homomorphism $\theta:\cO_K\to L$. Let $L$ be a local field over $\cO_K$, we denote by $\cO_L$ the ring of integers in $L$ and $\cM_L$ the maximal ideal of $\cO_L$. Let $k_L=\cO_L/\cM_L$ be the residue field of $L$,  $p_L$ be the characteristic of $k_L$ and $q_L$ be the cardinality of $k_L$. If $z\in\cO_L$, we denote by $\overline{z}$ the image of $z$ by the reduction map $\cO_L\to k_L$. Let $\ord_L: L\to \ZZ\cup\{+\infty\}$ be the valuation map of $L$. We fix a uniformizing element $\varpi_L$ of $\cO_L$ and denote by $ac:L \to\cO_L^*\cup\{0\}$ the angular component map given by $ac(0)=0$ and $ac(z)=z\varpi_L^{-\ord_L(z)}$. Let $i\geq j$ be positive integers, we denote by $\pi_i:\cO_L\to\cO_L/(\varpi_L^i)$ and $
\pi_{ij}:\cO_L/(\varpi_L^i)\to \cO_L/(\varpi_L^j)$ the canonical projections. If $x\in\cO_L$, we set $\overline{x}=\pi_1(x)$. For each $\cO_L$-scheme $X$ and $i\geq j\geq 1$, $\pi_{ij}$ induces the canonical projection $\pi_{ij}^X: X(\cO_L/(\varpi_L^i))\to X(\cO_L/(\varpi_L^j))$. We will write $\pi_{ij}$ instead of $\pi_{ij}^X$ if $X$ is understood clearly in context.  If $\cI$ is an ideal of $\cO_L[x_1,\dots,x_n]$ generated by $f_1,...,f_r$ and $x\in \cO_L^n$, then we set $\ord_L\left(\cI(x)\right):=\min_{f\in \cI}\ord_L(f(x))=\min_{1\leq i\leq r}\ord_L(f_i(x))$.

An \index{additive character}additive character $\psi$ of $L$ is a continuous homomorphism $\psi: L\to \CC^{*}$. Let $\psi$ be a non-trivial additive character of $L$, then the \index{conductor of additive characters}conductor $m_\psi$ of $\psi$ is the integer $m$ such that $\psi$ is trivial on $\varpi_L^m\cO_L$, but non-trivial on $\varpi_L^{m-1}\cO_L$.

Let $L$ be a local field over $\ZZ$ and $n$ be a positive integer, we denote by $\left|dx_1\wedge...\wedge dx_n\right|$ or $\mu_{L^n}$ the normalized Haar measure on $L^n$ such that the volume of $\cO_L^n$ is $1$.

Let $M$ be an integer, we denote by $\cL_{K,M}$ the set of all local fields over $\cO_K$ of characteristic zero and of residue field characteristic at least $M$. We also denote by $\cL'_{K,M}$ the set of all local fields over $\cO_K$ of characteristic at least $M$. We will put $\tilde{\cL}_{K,M}=\cL_{K,M}\cup\cL'_{K,M}$.

Let $\mathfrak{p}\in\Specm(\cO_K)$, then we denote by $K_\mathfrak{p}$  the completion of $K$ with respect to the place $\mathfrak{p}$ and set $\cN_\mathfrak{p}:=\left|N_{K/\QQ}(\mathfrak{p})\right|=\#(\cO_K/\mathfrak{p})=q_{K_\mathfrak{p}}$. 

\subsection{Schemes and varieties}\label{SV}Let $R$ be a commutative ring with $1\neq 0$, $\cI$ be an ideal of $R[x_1,\dots,x_n]$, $R'$ be an $R$-algebra, then we will write $\cI_{R'}$ instead of $\cI\otimes_R R'$. Similarly, if $X, Y$ are $R$-schemes, $\varphi:X\to Y$ is a $R$-morphism and $R'$ is a $R$-algebra then we write $\varphi_{R'}:X_{R'}=X\times_{\spec(R)}\spec(R')=X\otimes_R R'\to Y_{R'}=Y\times_{\spec(R)}\spec(R')=Y\otimes_R R'$ for the extension of scalars  to $R'$ of $\varphi$. Throughout this paper, if $R$ is understood, we will use both notations $X_{R'}$ and  $X\otimes R'$ to denote $X\otimes_R R'$ depending on the appearance of subscripts. More precisely, we will write $X_{R'}, X_i\otimes R'$. 
 We set $\mathfrak{D}_{R}^{r}:=\AA_R^r\setminus\Spec(R[y_1,...,y_r]/(y_1,...,y_r))$. Let $Z$ be a subscheme of $\AA_R^n$, we set $Z\mathfrak{D}_R^{r}:=Z\times_{\Spec(R)}(\AA_R^r\setminus\Spec(R[y_1,...,y_r]/(y_1,...,y_r)))$. We will write $\mathfrak{D}_R^{n,r}$ instead of $\AA_R^n\mathfrak{D}_R^{r}$. 

Let $\mathsf{F}$ be a field. We denote by $\overline{\mathsf{F}}$ a fixed algebraic closure of $\mathsf{F}$.  An $\mathsf{F}$-variety is a reduced and separated $\mathsf{F}$-scheme of finite type. If $X$ is an $\mathsf{F}$-scheme, we denote by $X_{red}$ the reduced scheme associated to $X$. Let $X$ be an $\mathsf{F}$-variety and $V$ be a subscheme of $X$, we denote by $\dim_V(X)$ the dimension of $X$ at $V$, i.e., $\dim_V(X)=\dim (U)$ for every small enough neighbourhood $U$ of $V$ in $X$. We say that an $\mathsf{F}$-variety $X$ is equi-dimensional if $\dim_x(X)=\dim(X)$ for all $x\in X$. If this is the case, we also say that $X$ is of pure dimension $\dim(X)$.

Let $\mathsf{F}$ be a field and $X$ be a scheme of finite type over $\mathsf{F}$. If $X$ is a closed subscheme of $\AA_\mathsf{F}^n$, we say that $X$ is a \index{complete intersection}complete intersection in $\AA^n_\mathsf{F}$ if the ideal $\cI$ associated to $X$ can be generated by $n-\dim(X)$ elements. More generally,  we say that $X$ is a complete intersection if there is an affine and smooth $\mathsf{F}$-scheme $Y$, a closed embedding $X\to Y$ and regular functions $f_1,\dots,f_r\in \cO_Y(Y)$ such that $f_i$ is not a zero divisor in $\cO_Y(Y)/(f_1,\dots,f_{i-1})$ for all $1\leq i\leq r$ and the ideal of $X$ in $Y$ (via the embedding) is generated by $f_1,\dots,f_{r}$. We say that $X$ is a \index{locally complete intersection}locally complete intersection at $x\in X$ if we can find an open subscheme $U$ of $X$ containing $x$ such that $U$ is a complete intersection. If $X$ is a locally complete intersection at every point $x\in X$, then we say that $X$ is a locally complete intersection. By \cite[Section 6.3.2, Proposition 3.20 and Remark 3.23]{LiuQ}, if $X$ is  an $\mathsf{F}$-scheme  of finite type and $\mathsf{F}'$ is a field extension of $\mathsf{F}$, then $X$ is  a locally complete intersection if and only if  $X_{\mathsf{F}'}$ is a locally complete intersection.
\subsection{Singularities and jet schemes}\label{SJ}
Let $\mathsf{F}$ be a field of characteristic zero. Let $\cI$ be a non-zero ideal of $\mathsf{F}[x_1,\dots,x_n]$ such that $\cI\neq (1)$ and $X$ be the closed subscheme of $\AA_\mathsf{F}^n$ associated to $\cI$. Recall that a projective, birational $\mathsf{F}$-morphism $\pi: Y\to \AA_{\mathsf{F}}^n$, with a smooth variety $Y$, is a \index{log resolution}log resolution of $\cI$ if the inverse image $\cI\cO_Y$ of $\cI$ is the ideal of a Cartier divisor $D$ such that $D + K_{Y/\AA_{\mathsf{F}}^n}$ is a divisor with simple normal crossings, where the relative canonical divisor $K_{Y/\AA_{\mathsf{F}}^n}$ is locally defined by the determinant of the Jacobian matrix of $\pi$. Such a log resolution of $\cI$ exists as shown by the work of Hironaka in \cite{Hir:Res}. Let $\pi:Y\to \AA_\mathsf{F}^n$ be a log resolution of $\cI$. We can write $\cI\cO_Y=\cO_Y(-D)$ with $D=\sum_{i\in I}N_iE_i$  and $K_{Y/\AA_\mathsf{F}^n}=\sum_{i\in I}(\nu_i-1)E_i$ such that $\sum_{i\in I}E_i$ is a divisor with simple normal crossings. We call $((\nu_i, N_i))_{i\in I}$ the numerical data of $\pi$. Let $V$ be a subscheme of $\AA_\mathsf{F}^n$, then the \index{log canonical threshold}log canonical threshold $\lct_{V}(\cI)$ of $\cI$ at $V$ is defined by
$$\lct_V(\cI)=\min_{i\in I, V\cap\pi(E_i)\neq\emptyset}\frac{\nu_i}{N_i}=\min_{x\in V(\overline{\mathsf{F}})}\lct_x(\cI).$$
In particular, we set $\lct_V(\cI)=+\infty$ if $V\cap\pi(E_i)=\emptyset$ for all $i\in I$. This definition does not depend on the choice of $\pi$ (see \cite{Mustata2}). Let $U$ be a small enough neighbourhood of $V$ in $\AA_\mathsf{F}^n$, one can use the blowing up of $\AA_\mathsf{F}^n$ along the smooth locus of  $\overline{X\cap U}_{red}$ (the Zariski closure of $(X\cap U)_{red}$ in $\AA_{\mathsf{F}}^n$)  or use the description of log canonical threshold in terms of Jet schemes in \cite{MustJAMS} to see that $\lct_V(\cI)\leq n-\dim_V(X)$. If $X=V$, we write simple $\lct(\cI)$ instead of $\lct_X(\cI)$. Let $R$ be a subalgebra of $\mathsf{F}$, $\cI$ be an ideal of $R[x_1,...,x_n]$ and $V$ be a subscheme of $\AA_R^n$ then we set $\lct_V(\cI)=\lct_{V_\mathsf{F}}(\cI_\mathsf{F})$ and $\lct(\cI)=\lct(\cI_\mathsf{F})$.

Let $\mathsf{F}$ be a field of characteristic zero and  $X$ be an $\mathsf{F}$-variety. Let $h:Y\to X$ be a \index{resolution of singularities}resolution of singularities of $X$, i.e. $h$ is a proper birational morphism, $Y$ is a smooth variety, $h$ is an isomorphism outside the singular locus $\Sing(X)$ of $X$ and $(h^{-1}(\Sing(X)))_{red}$ is a divisor with simple normal crossings. The existence of such a resolution of singularities of $X$ also follows  from \cite{Hir:Res}. Let us write $h^{-1}(\Sing(X))=\sum_{i\in I}N_iE_i$ and $K_{Y/X}=\sum_{i\in I} a_iE_i$, where $\{E_i|i\in I\}$ is the set of irreducible components of $(h^{-1}(\Sing(X)))_{red}$. Suppose that $X$ is normal and the canonical divisor $K_X$ of $X$ is $\QQ$-Cartier, i.e., $mK_X$ is a Cartier divisor on $X$ for some positive integer $m$. We say that
\begin{itemize}
\item[(\textit{i}),] $X$ has only \index{terminal singularity}terminal singularities if $a_i>0$ for all $i\in I$.
\item[(\textit{ii}),]$X$ has only \index{canonical singularity}canonical singularities if $a_i\geq 0$ for all $i\in I$.
\item[(\textit{iii}),]$X$ has only \index{log terminal singularity}log terminal singularities if $a_i>-1$ for all $i\in I$.
\item[(\textit{iv}),]$X$ has only \index{log canonical singularity}log canonical singularities if $a_i\geq -1$ for all $i\in I$.
\item[(\textit{v}),]$X$ has only \index{rational singularity}rational singularities if the higher direct images $\cR^jh_*(\cO_Y)=0$ for all $j>0$.
\end{itemize}
These definitions do not depend on the choice of $h$ (see, e.g.,  \cite[Definition 2.8]{Kollar1}).

Let $\mathsf{F}$ be a field. Recall that an $\mathsf{F}$-variety $X$ satisfies Serre's condition $S_m$ if $\depth(\cO_{X,x})\geq \min\{m,\dim_{x,X}\}$ for all $x\in X$, where $\depth(\cO_{X,x})$ is the maximal length of a regular sequence of $\cO_{X,x}$, i.e., a sequence of elements  $g_1,...,g_r$ in the maximal ideal of $\cO_{X,x}$ such that $g_1$ is not a zero divisor in  $\cO_{X,x}$ and $g_i$ is not a zero divisor in $\cO_{X,x}/(g_1,...,g_{i-1})$ for all $2\leq i\leq r$. An $\mathsf{F}$-variety $X$ is demi-normal if $X$ is $S_2$ and its codimension $1$ points are either regular points or nodes. By Serre's condition for normality, $X$ is normal if and only if $X$ is $S_2$ and regular in codimension 1.  Thus, if $X$ is normal then $X$ is demi-normal. If $X$ is demi-normal, we have the notion of canonical divisor $K_X$ of $X$ (see \cite[page 191]{Kollar1}).  Let $Z$ be a reduced and demi-normal $\mathsf{F}$-variety whose normalization is $\tilde{Z}$. The conductor of $Z$, denoted by $\cond(Z)$, is the largest ideal sheaf on $Z$ such that it is also an ideal sheaf on $\tilde{Z}$. We put $\tilde{D}=\spec(\cO_{\tilde{Z}}/\cond(Z))$ then $\tilde{D}$ is a reduced divisor of pure codimension $1$ on $\tilde{Z}$ (see \cite[page 189]{Kollar1}). Suppose moreover that the canonical divisor $K_{Z}$ of $Z$ is $\QQ$-Cartier. Let $h:Y\to \tilde{Z}$ be a birational morphism then we can write $K_{Y/\tilde{Z}}=h^{*}(\tilde{D})+\sum_{i\in I_h}a_iE_i$ for rational numbers $a_i$ and distinct prime divisors $E_i$ on $Y$. \index{semi-log canonical singularity}We say that $Z$ has only semi-log canonical singularities if $a_i\geq -1$ for all $i\in I_h$ and all $(Y,h)$ such that $Y$ is regular at the generic point of $E_i$ for every $i$ (see \cite[Definition-Lemma 5.10]{Kollar1}). In particular, if $Z$ has only log canonical singularities then $Z$ has only semi-log canonical singularities.


Let $S$ be a scheme and $X$ be an $S$-scheme. \index{jet schemes}We recall the definition of the \textit{jet} schemes of $X$ in \cite[Chapter 3, Section 2]{CNic}. For each $m\in\NN$, the $m^{\textnormal{th}}$ jet schemes of $X$ over $S$, denoted by $X_{m}/S$, is the $S$-scheme representing the functor from the category of $S$-schemes to the category of  sets sending an $S$-scheme $V$ to $\textnormal{Hom}_{S-\textnormal{schemes}}(V\times_{\Spec(\ZZ)}\Spec(\ZZ[t]/(t^{m+1})),X)$.  Note that $X_0/S$ is identified with $X$ in an obvious way.  If $i\geq j\geq 0$, we have a canonical morphism $\pi_X^{ij}: X_{i}/S\to X_{j}/S$ induced by the ring homomorphism $\ZZ[t]/(t^{i+1})\to \ZZ[t]/(t^{j+1})$. We call $\pi_X^{ij}$ a truncation morphism. For each $m\in \NN$, the homomorphism $\ZZ\to\ZZ[t]/(t^{m+1})$ induces a zero section $s_{X,m}:X\to X_m/S$ so that $\pi_X^{m0}\circ s_{X,m}=\textnormal{id}_X$.  If $\varphi:X\to Y$ is a morphism of $S$-schemes then $\varphi$ induces naturally an $S$-morphism $\varphi_m:X_m/S\to Y_m/S$ for each $m\geq 0$  satisfying $\pi_Y^{ij}\circ\varphi_i=\varphi_j\circ\pi_X^{ij}$ for all $i\geq j\geq 0$. We call $\varphi_m$ the $m^{\textnormal{th}}$ jet morphism of $\varphi$. Let $T$ be an $\cS$-scheme, then $X_T=X\times_S T$ is a $T$-scheme and we have an isomorphism between $(X_T)_m/T$ and $(X_m/S)\times_S T$ for all $m\in\NN$. From now on, we will write $X_m$ instead of $X_{m}/S$ if $S$ is understood.
\begin{remark}\label{jetloc} Let $K$ be a number field. Let $S=\Spec(\cO_K)$ and $X$ be an $S$-scheme of finite type. Let $k$ be a finite field endowed with a structure of $\cO_K$-algebra then $L:=k((t))\in\tilde{\cL}_{K,1}$, $\cO_L=k[[t]]$ and we can take $\varpi_L=t$. By the definition,  $X_m(k)$ can be identified with $X(k[t]/(t^{m+1}))=X(k[[t]]/(t^{m+1}))$ and if $i\geq j\geq 1$ then the morphism $\pi_X^{(i-1)(j-1)}: X_{i-1}\to X_{j-1}$ induces the map $$\pi_{ij}^X:X(k[[t]]/(t^i))=X_{i-1}(k)\to X(k[[t]]/(t^j))=X_{j-1}(k)$$ mentioned in Section \ref{nuloc}. 

\end{remark}

We sketch some results on the relation of singularities and jet schemes. 
\begin{prop}\label{singjet}
Let $\mathsf{F}$ be a field of characteristic zero and $X$ be an $\mathsf{F}$-variety of dimension $n\geq 1$. Suppose that $X$ is  equi-dimensional and a locally  complete intersection. Then the following assertions hold:
\begin{itemize}
\item[(\textit{i}),] If $X$ is normal then $X$ has only semi-log canonical singularities if and only if $X$ has only log canonical singularities.
\item[(\textit{ii}),] $X_m$ is equi-dimensional for all $m\geq 1$ if and only if $\dim(X_m)\leq n(m+1)$ for all $m\geq 1$ if and only if $X_m$ is a locally complete intersection for all $m$ if and only if $X_m$ is a locally complete intersection  of dimension $n(m+1)$ for all $m\geq 1$ if and only if $X$ has only semi-log canonical singularities. 
\item[(\textit{iii}),]The number of irreducible components of $X_m\otimes \overline{\mathsf{F}}$ equals to the number of irreducible component of $X_{\overline{\mathsf{F}}}$  for all $m\geq 1$ if and only if $X$ has only rational singularities if and only if $X$ has only canonical singularities. If this is the case, each irreducible component of $X_m\otimes \overline{\mathsf{F}}$  is of dimension $n(m+1)$.
\item[(\textit{iv}),]$X_m$ is normal for all $m\geq 1$  if and only if $X$ has only terminal singularities.
\end{itemize}
\end{prop}
\begin{proof}
By the definition, the first item is trivial. The second item follows from \cite[Theorem 1.4]{Mustata1}, \cite[Corollary 10.2.9]{Ishs} and \cite[Corollary 3.17]{Eishi}. The third item  follows from  \cite[Theorem 2.1]{Mustata1} and \cite{Elk}. The last item was proved in \cite{EMY,Ein-Must}.
\end{proof}
\begin{cor}\label{jetsi}Let $\mathsf{F}$ be a field of characteristic zero and $X$ be the closed subscheme of $\AA_\mathsf{F}^n$ of codimension $r$ associated to a non-zero ideal $\cI\neq (1)$ of $\mathsf{F}[x_1,...,x_n]$. Suppose that $X$ is equi-dimensional and a locally complete intersection. Let $h:Y\to\AA_\mathsf{F}^n$ be a log resolution of $\cI$ associated to irreducible divisors $(E_i)_{i\in I}$ and numerical data $((\nu_i,N_i))_{i\in I}$. Then the following assertions hold:
\begin{itemize}
\item[(\textit{a}),]$X$ has only semi-log canonical singularities if and only if $\nu_i\geq rN_i$ for all $i\in I$.
\item[(\textit{b}),]$X$ has only rational singularities if and only if the set $\{i\in I\mid\nu_i<rN_i\}\cup \{i\in I\mid \nu_i=rN_i>r\}\cup\{i\in I\mid (\nu_i,N_i)=(r,1),  \exists j\in I \textnormal{ such that } j\neq i, (\nu_j,N_j)=(r,1) \textnormal{ and } E_j\cap E_i\neq\emptyset\}$ is empty.
\end{itemize}
\end{cor}
\begin{proof}By \cite[Theorem 2.1]{ELM}, the condition that $\nu_i\geq rN_i$ for all $i\in I$ is equivalent to $\dim(X_m)=n(m+1)$. Thus, we can use the item (\textit{ii}) of Proposition \ref{singjet} to deduce the item (\textit{a}).

The item (\textit{b}) follows from the proof of \cite[Theorem 2.1]{Mustata1}.

\end{proof}
To finish this section, we recall the following result on the relation between log-canonical threshold and jet schemes. 
\begin{prop}[Corollary 3.4, \cite{MustJAMS}]\label{jetlog}Let $\mathsf{F}$ be a field of characteristic $0$ and $X$ be a closed subscheme of $\AA_{\mathsf{F}}^n$ associated to a non-zero ideal $\cI\neq (1)$. Let $Z$ be a closed subset of $X$. Then 
$$\lct_Z(\cI)\leq n-\frac{\dim_{s_{X,m}(Z)}(X_m)}{m+1}$$
for all $m\geq 0$ and the equality holds for infinitely many $m$.
\end{prop}
\subsection{Bernstein-Sato polynomial and minimal exponent}\label{BM}
Let $\mathsf{F}$ be a field of characteristic zero. Let $f$ be a non-constant polynomial in $\mathsf{F}[x_1,...,x_n]$. Then there is an element $P\in \mathsf{F}[x_1,\dots,x_n,\frac{\partial}{\partial x_1},\dots, \frac{\partial}{\partial x_n},s]$ and a polynomial $b(s)\in \mathsf{F}[s]\setminus\{0\}$ such that $Pf^{s+1}=b(s)f^s$ (see \cite{Bern}). The monic polynomial $b_f(s)\neq 0$ of smallest degree satisfying this property is called the \index{$b_f(s)$, Bernstein-Sato polynomial of a regular function $f$}Bernstein-Sato polynomial of $f$. It is easy to see that $(s+1)\mid b_f(s)$.  The \index{$\tilde{b}_f(s)$, reduced Bernstein-Sato polynomial of a regular function $f$}reduced Bernstein-Sato polynomial of $f$ is  $\tilde{b}_f(s)=b_f(s)/(s+1)$. More generally, if $X$ is a smooth $\mathsf{F}$-variety and $f$ is a regular function on $X$ then we also can define  $b_f(s),\tilde{b}_f(s)$ in a similar way. The fact is that every root of $\tilde{b}_f(s)$ is a negative rational number (see \cite{Kashiwara}). If $X$ is a smooth $\mathsf{F}$-variety and $f$ is a regular function on $X$, then the \index{$\tilde{\alpha}_f$, minimal exponent of a regular function $f$}minimal exponent $\tilde{\alpha}_f$ of $f$ is the smallest root of $\tilde{b}_f(-s)$. Locally, if $X$ is a smooth  $\mathsf{F}$-variety, $x_0\in X$ and $f$ is a regular function on $X$ such that $f(x_0)=0$, then the minimal exponent $\tilde{\alpha}_{x_0}(f)$ of $f$ at $x_0$ is the  minimal exponent $\tilde{\alpha}_{f\mid_V}$ for a small enough neighbourhood  $V$ of $x_0$. Moreover, one has $\tilde{\alpha}_f=\min_{x\in X, f(x)=0}\tilde{\alpha}_x(f)$ (see, e.g.,  \cite[Lemma 2.5.2]{Gyo}). We also set $\tilde{\alpha}_{x_0}(f)=+\infty$ if $f(x_0)\neq 0$. Hence, $\tilde{\alpha}_f=\min_{x\in X}\tilde{\alpha}_x(f)$. 

Let $X$ be a smooth variety over a field $\mathsf{F}$ of characteristic zero. The above notion of Bernstein-Sato polynomials has been extended to arbitrary non-zero coherent ideal sheaf $\mathfrak{a}$ in $\cO_X$ (see \cite{BMS}).  For instance, if $X=\AA_{\mathsf{F}}^n$ and $\mathfrak{a}$ is generated  by non-zero polynomials  $f_1,...,f_r$ then \index{$b_\mathfrak{a}(s)$, Bernstein-Sato polynomial of a non-zero coherent ideal sheaf $\mathfrak{a}$} the Bernstein-Sato polynomial $b_\mathfrak{a}(s)$ of $\mathfrak{a}$ is the monic polynomial of smallest degree such that
$b_\mathfrak{a}(s)f_1^{s_1}\cdots f_r^{s_r}$ belongs to $$\sum_{|u|=1}\mathsf{F}[x_1,\dots,x_n,\frac{\partial}{\partial x_1},\dots, \frac{\partial}{\partial x_n},s_1,...,s_r]\prod_{u_i<0}\binom{s_i}{-u_i}f_1^{s_1+u_1}\cdots f_r^{s_r+u_r},$$
where the sum is over all $u=(u_1,...,u_r)\in\ZZ^r$ such that $|u|:=u_1+...+u_r=1$, $s=s_1+...+s_r$ for independent variables $s_1,...,s_r$, and  $\binom{s_i}{m}:=\frac{1}{m!}\prod_{j=0}^{m-1}(s_i-j)$ for each positive integer $m$ and each variable $s_i$. By \cite[Theorem 1.1]{Mustideal}, we have $b_{\mathfrak{a}}(s)=\tilde{b}_{\sum_{i=1}^ry_if_i}(s)$ if $\mathfrak{a}=(f_1,...,f_r)\subset \mathsf{F}[x_1,...,x_n]$. On the other hand, if $\mathfrak{\mathfrak{a}}$ is a non-zero coherent ideal sheaf of $\cO_X$, the notion of \index{$\tilde{\alpha}_{\mathfrak{a}}$, minimal exponent of a non-zero coherent ideal sheaf $\mathfrak{a}$}minimal exponent $\tilde{\alpha}_{\mathfrak{a}}$ of $\mathfrak{a}$ was introduced in \cite{CDMO} by using $V$-filtration. In particular, if  $\mathfrak{a}$ is the ideal of $\mathsf{F}[x_1,...,x_n]$ generated by $f_1,...f_r$ such that the closed subscheme of $\AA_{\mathsf{F}}^n$ associated to $\mathfrak{a}$ is of dimension $n-r$ then $\tilde{\alpha}_{\mathfrak{a}}=\tilde{\alpha}_{(\sum_{i=1}^ry_if_i)|_{U}}$ as showed in \cite[Theorem 1.1]{CDMO}, where $U=(\AA_{\mathsf{F}}^r\setminus\{0\})\times \AA_{\mathsf{F}}^n$. If $\mathfrak{a}$ defines a closed subscheme $Y$ of $X$, then we also write $b_Y(s)$ and $\tilde{\alpha}_Y$ instead of $b_\mathfrak{a}(s)$ and $\tilde{\alpha}_\mathfrak{a}$ respectively. If $Y$ is a closed subscheme of pure codimension $r$ in $X$, it follows from \cite[Proposition 6.1]{CDMO} that $b_Y(-r)=0$. Thus, as with the minimal exponent of polynomials, we also expect that $\tilde{\alpha}_Y$ is the smallest root of $b_Y(-s)/(r-s)$ if $Y$ is a closed subscheme of pure codimension $r$ in $X$.
\subsection{Igusa local zeta functions and exponential sums modulo power of primes}\label{IE}
In this section, we fix a number field $K$, a positive integer $n$ and an $\cO_K$-scheme of finite type $Z\subset\AA_{\cO_K}^n$. 

If $L$ is a local field over $\cO_K$, then we set $$\phi_{L,Z}:=\textbf{1}_{\{x\in \cO_L^n\mid \overline{x}\in Z(k_L)\}}.$$ 
\begin{defn}
Let $\cI$ be an ideal of $\cO_K[x_1,...,x_n]$. Let $L\in\cL_{K,1}$. We associated  $L,Z,\cI$ to the Igusa local zeta function $\cZ_{L,Z,\cI}(s)$ defined by \index{$\cZ_{L,Z,\cI}(s)$, Igusa local zeta function of an ideal $\cI$ over a $p$-adic field $L$ at an $\cO_L$-scheme $Z$}
$$\cZ_{L,Z,\cI}(s):=\int_{\cO_L^n}\phi_{L,Z}(x)q_L^{-s\ord_L\left(\cI_{\cO_L}(x)\right)}\left|dx\right|,$$
for $\Re(s)>0$. If $\cI$ is generated by $f$, we will write $\cZ_{L,Z,f}(s)$ instead of $\cZ_{L,Z,(f)}(s)$. 
\end{defn}
We recall the following result in \cite{VeysZ}.
\begin{prop}[\cite{VeysZ}, Theorem 2.4]\label{VZ}Let $\cI$ be an ideal of $\cO_K[x_1,...,x_n]$ such that $\cI_K\neq (1)$. Let $h$ be a log resolution of $\cI_K$ and $((\nu_i,N_i))_{i\in I}$ be the numerical data of $h$. Let $L\in\cL_{K,1}$. Then $\cZ_{L,Z,\cI}(s)$ admits a meromorphic continuation to the complex plane as a rational function of $q_L^{-s}$.  Moreover $\cZ_{L,Z,\cI}(s)\prod_{i\in I}(1-q_L^{-N_is-\nu_i})$ is a holomorphic function.  In particular, $-\Re(s_0)\geq \lct_Z(\cI)$ for all poles $s_0$ of $\cZ_{L,Z,\cI}(s)$. In addition, let $K'$ be a finite extension of $K$ such that there is $i\in I$ satisfying $\lct_{Z}(\cI)=\frac{\nu_i}{N_i}$ and $E_i\otimes K'$ has a geometrically irreducible component, then there exists an integer $M$ such that $\cZ_{L,Z,\cI}(s)$ has a pole $s_0$ satisfying $\Re(s_0)=-\lct_Z(\cI)$ provided that $L\in\cL_{K',M}$.  
\end{prop}

Let $L\in\tilde{\cL}_{K,1}$, $f$ be a non-constant polynomial in $\cO_L[x_1,\dots,x_n]$ and $\psi$ be an additive character of $L$.  The exponential sum $E_{L,Z,f}(\psi)$ associated to $f,L,Z,\psi$ is given by\index{exponential sum modulo $p^m$ of polynomials}
$$E_{L,Z,f}(\psi):=\int_{\cO_L^n}\phi_{L,Z}(x)\psi(f(x))\left|dx\right|.$$
\begin{defn}\label{L-os} Let $f\in\cO_K[x_1,...,x_n]$ be a non-constant polynomial and $L\in \cL_{K,1}$. The \index{$\sigma_{L,Z}(f)$, $L$-oscillation index of a polynomial $f$ over a $p$-adic field $L$ at an $\cO_L$-scheme $Z$}$L$-oscillation index $\sigma_{L,Z}(f)$ of $f$ at $Z$ is the supremum of all positive real numbers $\sigma$ such that there exists a constant $c=c(L,f,Z,\sigma)$ satisfying $\left|E_{L,Z,f}(\psi)\right|\leq cq_L^{-m_\psi\sigma}$
for all non-trivial additive characters $\psi$ of $L$.
\end{defn} Let us recall the following definition and proposition.
\begin{defn}[\cite{CMN,NguyenVeys}]\label{dmoif}Let $f\in\cO_K[x_1,...,x_n]$ be a non-constant polynomial. The constant
\begin{equation}\label{mdef0} 
\moi_{K,Z}(f):=\liminf_{L\in\cL_{K,M},M\to +\infty} \sigma_{L,Z}(f)\leq +\infty
\end{equation}
is so called the \index{$\moi_{K,Z}(f)$, motivic oscillation index of a polynomial $f$ over a number field $K$ at an $\cO_K$-scheme $Z$} motivic oscillation index of $f$ at $Z$ over $K$.
\end{defn} 

\begin{prop}[\cite{NguyenVeys}, Corollary 2.4 and Section 2.4]\label{finic}Let $f\in\cO_K[x_1,...,x_n]$ be a non-constant polynomial and $h:Y\to\AA_K^n$ be a log resolution of $(f)$. Let $((\nu_i,N_i))_{i\in I}$ be the numerical data of $h$. Suppose that $f(Z(\CC))=0$ then $\moi_{K,Z}(f)$ and  $\sigma_{L,Z}(f)$ belong to the set $\{+\infty\}\cup\{\frac{\nu_i}{N_i}\mid i\in I\}$ for all $L\in\cL_{K,1}$.
\end{prop}

 We also recall a general form of Igusa's conjecture for exponential sums in \cite[Page 25]{CMN}.
\begin{conj}[Igusa's conjecture for exponential sums]\label{Iguconj}Let $f$ be a non-constant polynomial in $\cO_K[x_1,...,x_n]$. Then there is an integer $M$ and a positive constant $c$ depending only on $f$ such that for all local fields $L\in\tilde{\cL}_{K,M}$ and all additive characters $\psi$ of $L$ of conductor $m_{\psi}\geq 2$, we have
$$\left|E_{L,Z,f}(\psi)\right|\leq cm_{\psi}^{n-1}q_L^{-m_\psi\moi_{K,Z}(f)}.$$
\end{conj}
\section{Motivic oscillation indexes of ideals and conjectures for exponential sums}\label{exponential sum}
In this section, the story in the introduction will be formulated under a more general setting. Namely, we play with an arbitrary number field $K$ instead of the field of rational numbers and develop the ideas in Section \ref{intro} for polynomials over the ring $\cO_K$ of integers of $K$.
\subsection{Motivic oscillation indexes of ideals}\label{mideal}
\begin{lemdefn}\label{moidef}Let $\cI$ be a non-zero ideal of $\cO_K[x_1,\dots,x_n]$ such that $\cI_K\neq (1)$. Let $r$ be a positive integer and $Z\subset \AA_{\cO_K}^n$ be an $\cO_K$-scheme of finite type.  If $\cI_K$ can not be generated by $r$ elements, we put $\moi_{K,Z}^{(r)}(\cI)=0$. Otherwise, suppose that $\cI_K$ can be generated by $r$ polynomials $f_1,\dots,f_r\in \cO_K[x_1,\dots,x_n]$, with the notation of Definition \ref{dmoif}, $$\moi_{K,Z}^{(r)}(\cI):=\moi_{K, Z\mathfrak{D}_{\cO_K}^r}\left(\sum_{i=1}^r y_if_i(x)\right)$$
does not depend on the choice of $f_1,\dots,f_r$, where $Z\mathfrak{D}_{\cO_K}^r$ was defined in Section \ref{SV}. 
We will call $\moi_{K,Z}^{(r)}(\cI)$ the $r^{\textnormal{th}}$-motivic oscillation index  of $\cI$ at $Z$ over $K$.

We will set $$\moi_{K,Z}(\cI):=\sup_{r\geq 1}\moi_{K,Z}^{(r)}(\cI)$$ and call it the \index{$\moi_{K,Z}(\cI)$, motivic oscillation index of an ideal $\cI$ over a number field $K$ at an $\cO_K$-scheme $Z$}motivic oscillation index of $\cI$ at $Z$ over $K$. We will write $\moi_{K}^{(r)}(\cI)$, $\moi_{K}(\cI)$ instead of $\moi_{K,\AA_{\cO_K}^n}^{(r)}(\cI)$ and $\moi_{K,\AA_{\cO_K}^n}(\cI)$ respectively.
\end{lemdefn}
\begin{proof}
If $\cI_K=(f_1,\dots,f_r)$, then it is easy to verify that there exists an integer $M$ such that $\cI_{\cO_L}=(f_1,\dots,f_r)$ if $L\in\tilde{\cL}_{K,M}$. We put $X=\spec(\cO_K[x_1,\dots,x_n]/\cI)$ and $$g(x,y)=\sum_{i=1}^ry_if_i(x).$$  Let $Z\subset\AA_{\cO_K}^n$ be an $\cO_K$-scheme of finite type. By the definition of $\moi_{K,Z\mathfrak{D}_{\cO_K}^r}(g)$, it is sufficient to show that if $L\in\tilde{\cL}_{K,M}$ and $\psi$ is an additive character  of conductor $m\geq 1$ of $L$ then $E_{L,Z\mathfrak{D}_{\cO_K}^r,g}(\psi)$ is independent of the choice of $f_1,\dots,f_r$. We can use the Fourier transform on the finite groups $(\cO_L/(\varpi_L^m))^r$ for $m\geq 1$ to prove it. However, we present here a proof using a direct calculation to provide more intuition on exponential sums to the reader. 

By orthogonality of characters, if $\psi$ is an additive character of $L$ of conductor $m\geq 2$, we have 
\begin{equation}\label{orth1}\int_{y\in\cO_L^r\setminus \varpi_L\cO_L^r}\psi\left(\sum_{i=1}^r y_if_i(x)\right)\left|dy\right|=0
\end{equation}
if $\min_{1\leq i\leq r}\ord_L(f_i(x))<m-1$,
and 
\begin{equation}\label{ortho2}
\int_{y\in\cO_L^r}\psi\left(\sum_{i=1}^r y_if_i(x)\right)\left|dy\right|=0
\end{equation}
if $\min_{1\leq i\leq r}\ord_L(f_i(x))= m-1$.
Let $L\in\tilde{\cL}_{K,1}$ and $m\geq 0$, we set $$A_{mZL}=\{x\in\cO_L^n\mid\overline{x}\in Z(k_L), \min_{1\leq i\leq r}\ord_L(f_i(x))<m-1\},$$ 
$$B_{mZL}=\{x\in\cO_L^n\mid\overline{x}\in Z(k_L), \min_{1\leq i\leq r}\ord_L(f_i(x))=m-1\}$$ and $$C_{mZL}=\{x\in\cO_L^n\mid\overline{x}\in Z(k_L), \min_{1\leq i\leq r}\ord_L(f_i(x))\geq m\}.$$
Then $A_{mZL},B_{mZL}$ and $C_{mZL}$ are disjoint sets satisfying $$A_{mZL}\cup B_{mZL}\cup C_{mZL}=\{x\in\cO_L^n\mid\overline{x}\in Z(k_L)\},$$
$$B_{mZL}\cup C_{mZL}=C_{(m-1)ZL}$$ and  
$$\int_{x\in C_{(m-1)ZL}}|dx|=q_L^{-(m-1)n}\#\pi_{(m-1)1}^{-1}(X\cap Z)(k_L)$$
for all $L\in\tilde{\cL}_{K,1}$ and all $m\geq 2$, where if $i\geq j\geq 0$ then the map $\pi_{ij}:X(\cO_L/(\varpi_L)^i)\to X(\cO_L/(\varpi_L)^j)$ was defined in Section \ref{nuloc}.
These together with  (\ref{orth1}), (\ref{ortho2}) and Fubini's theorem imply

\begin{align}
&E_{L,Z\mathfrak{D}_{\cO_K}^r,g}(\psi)\nonumber\\=&\int_{y\in\cO_L^r\setminus \varpi_L\cO_L^r}\int_{x\in\cO_L^n}\phi_{L,Z}(x)\psi\left(\sum_{i=1}^r y_if_i(x)\right)\left|dx\right|\left|dy\right|\nonumber\\
=&\int_{\{x\in\cO_L^n\mid\overline{x}\in Z(k_L)\}}\int_{y\in\cO_L^r\setminus \varpi_L\cO_L^r}\psi\left(\sum_{i=1}^r y_if_i(x)\right)\left|dy\right|\left|dx\right|\nonumber\\
=&\int_{x\in A_{mZL}}\underbrace{\left(\int_{y\in\cO_L^r\setminus \varpi_L\cO_L^r}\psi\left(\sum_{i=1}^r y_if_i(x)\right)\left|dy\right|\right)}_{=0}\left|dx\right|+\int_{x\in B_{mZL}}\underbrace{\left(\int_{y\in\cO_L^r}\psi\left(\sum_{i=1}^r y_if_i(x)\right)\left|dy\right|\right)}_{=0}\left|dx\right|\nonumber\\
&-\int_{x\in C_{(m-1)ZL}}\int_{y\in \varpi_L\cO_L^r}\underbrace{\psi\left(\sum_{i=1}^r y_if_i(x)\right)}_{=1}\left|dy\right|\left|dx\right|+\int_{x\in C_{mZL}}\int_{y\in\cO_L^r}\underbrace{\psi\left(\sum_{i=1}^r y_if_i(x)\right)}_{=1}\left|dy\right|\left|dx\right|\nonumber\\
=&\int_{x\in C_{mZL}}\left|dx\right|-q_L^{-r}\int_{x\in C_{(m-1)ZL}}\left|dx\right|\nonumber\\
=&q_L^{-mn}\left(\#\pi_{m1}^{-1}(X\cap Z)(k_L)-q_L^{n-r}\#\pi_{(m-1)1}^{-1}(X\cap Z)(k_L)\right)\label{exdm}
\end{align}if $L\in\tilde{\cL}_{K,M}$ and $\psi$ is an additive character of conductor $m\geq 2$ of $L$. Similarly, if $L\in\cL_{K,M}$ and $\psi$ is an additive character of conductor $1$ of $L$, then we have 
\begin{align}
E_{L,Z\mathfrak{D}_{\cO_K}^r,g}(\psi)=&\int_{x\in C_{1mL}}\left|dx\right|-q_L^{-r}\int_{x\in C_{0mL}}\left|dx\right|
=q_L^{-n}\left(\#(X\cap Z)(k_L)-q_L^{-r}\#Z(k_L)\right).\label{exd1}
\end{align}
 Therefore, $E_{L,Z\mathfrak{D}_{\cO_K}^r,g}(\psi)$ is independent of the choice of polynomials $f_1,\dots,f_r$.  
\end{proof}
By the proof of Lemma-Definition \ref{moidef}, it is easy to verify the following corollary.
\begin{cor}
Let $\cI, Z$ be as in Lemma-Definition \ref{moidef} and $X$ be the $\cO_K$-scheme associated to $\cI$, then  we have 
$$\moi_{K,Z}^{(r)}(\cI)=\moi_{K,Z\cap X}^{(r)}(\cI)$$
and
$$\moi_{K}^{(r)}(\cI)=\moi_{K,\AA_{\cO_K}^n}^{(r)}(\cI)= \moi_{K,X}^{(r)}(\cI).$$

\end{cor}
As mentioned in Section \ref{intro}, we will define the \index{$\moi_{K,Z}^{\textnormal{naive}}(\cI)$, naive motivic oscillation index of an ideal $\cI$ over a number field $K$ at an $\cO_K$-scheme $Z$}naive motivic oscillation indexes of a non-zero ideal $\cI$ of $\cO_K[x_1,\dots,x_n]$ as follows.
\begin{defn}\label{naiv}
Let $\cI$ be a non-zero ideal of $\cO_K[x_1,\dots,x_n]$ such that $\cI_K\neq (1)$. Let $X=\spec(\cO_K[x_1,\dots,x_n]/\cI)$ and $Z\subset\AA_{\cO_K}^n$ be an $\cO_K$-scheme of finite type. Let $r$ be a positive integer.  For each local field $L\in\cL_{K,1}$, let $\sigma_{L,Z}^{r,\textnormal{naive}}(\cI)$ be the minimum taken over all real numbers $\sigma$  such that $-\sigma$ is the real part of a pole of the function (see Proposition \ref{VZ})
$$F_{r,L,Z,\cI}(s)=\left(1-q_L^{-(s+r)}\right)\int_{\cO_L^n}\phi_{L,Z}(x)q_L^{-s\ord_L\left(\cI_{\cO_L}(x)\right)}\left|dx\right|.$$
We use the convention that $\sigma_{L,Z}^{r,\textnormal{naive}}(\cI)=+\infty$ if $F_{r,L,Z,\cI}(s)$ is an entire function. We call $\sigma_{L,Z}^{r,\textnormal{naive}}(\cI)$ the  $r^{\textnormal{th}}$-naive $L$-oscillation index of $\cI$ at $Z$.

The $r^{\textnormal{th}}$-naive motivic oscillation index $\moi_{K,Z}^{\textnormal{naive}}(\cI)$ of $\cI$ at $Z$ over $K$ is defined by
$$\moi_{K,Z}^{r,\textnormal{naive}}(\cI):=\liminf_{L\in\cL_{K,M}, M\to +\infty}\sigma_{L,Z}^{r,\textnormal{naive}}(\cI).$$
The naive motivic  oscillation index  of $\cI$ at $Z$ over $K$ (resp. the naive \index{$\sigma_{L,Z}^{\textnormal{naive}}(\cI)$, naive $L$-oscillation index of an ideal $\cI$ over a $p$-adic field $L$ at an $\cO_L$-scheme $Z$}$L$-oscillation index of $\cI$ at $Z$ for $L\in \cL_{K,1}$), denoted by $\moi_{K,Z}^{\textnormal{naive}}(\cI)$ (resp. $\sigma_{L,Z}^{\textnormal{naive}}(\cI)$), is  defined to be  $\moi_{K,Z}^{r_0,\textnormal{naive}}(\cI)$ (resp. $\sigma_{L,Z}^{r_0,\textnormal{naive}}(\cI)$), where $r_0=n-\dim_{Z_K}(X_K)$.
We will write $\moi_{K}^{\textnormal{naive}}(\cI)$ instead of $\moi_{K,\AA_{\cO_K}^n}^{\textnormal{naive}}(\cI)$  and call it the naive motivic  oscillation index of $\cI$ over $K$. 
\end{defn}
Now, we introduce a local version of Lemma-Definition \ref{moidef}. Let $\cI$ be a non-zero ideal of $\cO_K[x_1,\dots,x_n]$ such that $\cI_K\neq (1)$.  Let $Z$ be an $\cO_K$-scheme of finite type such that  $Z\subset X=\spec(\cO_K[x_1,\dots,x_n]/\cI)$ and $Z_K\neq \emptyset$. Let $U$ be an affine open subset of $\AA_K^n$ such that $Z_K\subset U$. Then we can find $f_1,\dots,f_r\in \cO_K[x_1,\dots,x_n]$ for some positive integer $r$ such that  $\cI\cO_U=(f_1,\dots,f_r)\cO_U$. Let $\cA(Z_K,X_K)$ be the set consisting of all tuples $(U,r,f_1,\dots,f_r)$ such that $U$ is an affine open subset of $\AA_K^n$,  $f_1,\dots,f_r$ are elements in $\cO_K[x_1,\dots,x_n]$, $Z_K\subset U$ and $\cI_K\cO_U=(f_1,\dots,f_r)\cO_U$.
\begin{defn}\label{moiloc1}
Let $\cI$ be a non-zero ideal of $\cO_K[x_1,\dots,x_n]$ such that $\cI_K\neq (1)$.  Let $Z$ be an $\cO_K$-scheme of finite type such that $Z\subset X=\spec(\cO_K[x_1,\dots,x_n]/\cI)$. The \index{$\moi_{K,Z}^{\textnormal{loc}}(\cI)$, local motivic oscillation index of an ideal $\cI$ over a number field $K$ at an $\cO_K$-scheme $Z$}local motivic oscillation index of $\cI$ at $Z$ over $K$, denoted by $\moi_{K,Z}^{\textnormal{loc}}(\cI)$ or $\moi_{K,Z}^{\textnormal{loc}}(X_K)$ is defined by
\begin{equation*}
\moi_{K,Z}^{\textnormal{loc}}(\cI)=\moi_{K,Z}^{\textnormal{loc}}(X_K):=\sup_{(U,r,f_1,\dots,f_r)\in\cA(Z_K,X_K)}\moi_{K,Z \mathfrak{D}_{\cO_K}^r}\left(\sum_{i=1}^r y_if_i(x)\right).
\end{equation*}

\end{defn}
\begin{remark}\label{well}
Let $Y$ be the closed subscheme of $\AA_K^n$ associated to an ideal $\cI\neq (1)$ of $K[x_1,...,x_n]$. For each subscheme $W$ of $Y$, we denote by $\cA(W,Y)$ the set consisting of all tuples $(U,r,f_1,\dots,f_r)$ such that $U$ is an affine open subset of $\AA_K^n$,  $f_1,\dots,f_r$ are elements in $\cO_K[x_1,\dots,x_n]$, $W\subset U$ and $\cI\cO_U=(f_1,\dots,f_r)\cO_U$. 

Let $Z$ be an $\cO_K$-subscheme of $\AA_{\cO_K}^n$ of finite type such that $Z_K\subset Y$.  We can use the idea of Definition \ref{moiloc1} to define the local motivic oscillation index $\moi_{K,Z}^{\textnormal{loc}}(Y)$ of $Y$ at $Z$ over $K$. 
More precisely, we only need to  set $$\moi_{K,Z}^{\textnormal{loc}}(Y):=\sup_{(U,r,f_1,\dots,f_r)\in\cA(Z_K,Y)}\moi_{K,Z \mathfrak{D}_{\cO_K}^r}\left(\sum_{i=1}^r y_if_i(x)\right).$$

In addition, if $\tilde{Z}$ is a subscheme of $Y$, then we can also define $\moi_{K,\tilde{Z}}^{\textnormal{loc}}(Y)$. Indeed, we observe that there is an $\cO_K$-scheme  $\mathsf{Z}$ of finite type such that $\tilde{Z}=\mathsf{Z}_K$. Moreover, suppose that $\mathsf{Z}_1,\mathsf{Z}_2\subset \AA_{\cO_K}^n$ are $\cO_K$-schemes of finite type such that $\mathsf{Z}_1\otimes K=\tilde{Z}=\mathsf{Z}_2\otimes K$, then there exists a non-zero integer $M$ such that $\mathsf{Z}_1\otimes \cO_K[1/M]=\mathsf{Z}_2\otimes \cO_K[1/M]$. Thus, $\phi_{L,\mathsf{Z}_1\mathfrak{D}_{\cO_K}^r}=\phi_{L,\mathsf{Z}_2\mathfrak{D}_{\cO_K}^r}$ if $L\in\cL_{K,M}$. Hence, by Definitions \ref{L-os} and \ref{dmoif}, if $f_1,\dots,f_r\in \cO_K[x_1,...,x_n]$ then $$\moi_{K,\mathsf{Z}_1\mathfrak{D}_{\cO_K}^r}\left(\sum_{i=1}^r y_if_i(x)\right)=\moi_{K,\mathsf{Z}_2\mathfrak{D}_{\cO_K}^r}\left(\sum_{i=1}^r y_if_i(x)\right).$$
Therefore, we only need to define $\moi_{K,\tilde{Z}}^{\textnormal{loc}}(Y)$ to be $\moi_{K,\mathsf{Z}}^{\textnormal{loc}}(Y)$ for an arbitrary $\cO_K$-scheme  $\mathsf{Z}\subset\AA_{\cO_K}^n$ of finite type such that $\tilde{Z}=\mathsf{Z}_K$.
\end{remark}

More generally, if $X$ is an $\cO_K$-scheme of finite type and $Z$ is a subscheme of $X_K$ such that $Z\neq \emptyset$ and there is an affine open subset $V$ of $X_K$ satisfying $Z\subset V$, then we may define the local motivic oscillation index $\moi_{K,Z}^{\textnormal{loc}}(X)$ of $X$ at $Z$ over $K$ by using a closed embedding of $V$ to a $K$-affine space. However, by this way,  $\moi_{K,Z}^{\textnormal{loc}}(X)$ depends on the choice of such a closed embedding. Therefore, we will need to avoid this dependence. In order to do this, we denote by $\cB(X,Z)$ the set consisting of all triples $(\iota,V,n)$, where $V$ is an affine open subset of $X_K$ containing $Z$ and $\iota:V\to  \AA_{K}^n$ is a closed embedding. 

\begin{defn}\label{moiloc2}Let $X$ be an $\cO_K$-scheme of finite type and $Z$ be a subscheme of $X_K$ such that $Z\neq \emptyset$ and there is an affine open subset $V$ of $X_K$ with $Z\subset V$. The \index{$\moi_{K,Z}^{\textnormal{aloc}}(X)$, absolutely local motivic oscillation index of a scheme $X$ over a number field $K$ at a subscheme $Z$ of $X$}absolutely local motivic oscillation index $\moi_{K,Z}^{\textnormal{aloc}}(X)$ of $X$ at $Z$ over $K$ is defined by
$$\moi_{K,Z}^{\textnormal{aloc}}(X):=\sup_{(\iota,V,n)\in \cB(X,Z)}\left(\moi_{K,\iota(Z)}^{\textnormal{loc}}(\iota(V))-n\right),$$
where $\moi_{K,\iota(Z)}^{\textnormal{loc}}(\iota(V))$ was defined in Definition \ref{moiloc1} and Remark \ref{well}. Note that we need the existence of $V$ to make sure that $\cB(X,Z)\neq\emptyset$. If $U$ is an open subscheme of $X$ of finite type over $\cO_K$ and $Z\subset U_K$, then Definition \ref{moiloc1} implies that $\moi_{K,Z}^{\textnormal{aloc}}(X)=\moi_{K,Z}^{\textnormal{aloc}}(U)$.

When $Z\subset X$ is an $\cO_K$-scheme of finite type such that $\emptyset\neq Z_K$ contained in an affine open subscheme of $X_K$, we define $$\moi_{K,Z}^{\textnormal{aloc}}(X):=\moi_{K,Z_K}^{\textnormal{aloc}}(X)$$
in an obvious way. 
\end{defn}
\begin{remark}\label{variety}In Definition \ref{moiloc2}, we can play with $K$-schemes of finite type instead of $\cO_K$-schemes of finite type. Indeed, if $X$ is a $K$-scheme  of finite type, we only need to choose $\tilde{X}$ to be an $\cO_K$-scheme of finite type  such that $X=\tilde{X}_K$ and define the absolutely motivic oscillation index of $X$ at $Z$ to be the absolutely motivic oscillation index of $\tilde{X}$ at $Z$. The discussion in Remark \ref{well} guarantees that our definition is well-defined.
\end{remark}
On the other hand,  Equalities (\ref{exdm}) and (\ref{exd1}) in the proof of Lemma-Definition \ref{moidef} suggest defining exponential sums of an ideal or a scheme as follows.

\begin{defn}\label{expodef}\index{exponential sum modulo $p^m$ of ideals}
Let $\cI$ be a non-zero ideal of $\cO_K[x_1,\dots,x_n]$ such that $\cI_K\neq (1)$. We set $X=\spec(\cO_K[x_1,\dots,x_n]/\cI)$. Let $Z\subset \AA_{\cO_K}^n$ be an $\cO_K$-scheme of finite type and $r$ be a positive integer. Let $L$ be a local field over $\cO_K$ and $m$ be a positive integer, then the $r^{\textnormal{th}}$ exponential sum modulo $\varpi_L^m$ of $\cI$ at $Z$ over $L$ is defined by 
\begin{align*}
&E_{L,Z,\cI}^{(r)}(m):=q_L^{-mn}\left(\#\pi_{m1}^{-1}(X\cap Z)(k_L)-q_L^{n-r}\#\pi_{(m-1)1}^{-1}(X\cap Z)(k_L)\right)
\end{align*}
if $m\geq 2$ and 
$$E_{L,Z,\cI}^{(r)}(m):=q_L^{-n}\left(\#(X\cap Z)(k_L)-q_L^{-r}\#Z(k_L)\right)$$
if $m=1$,  where if $i\geq j\geq 0$ then the map $\pi_{ij}:X(\cO_L/(\varpi_L)^i)\to X(\cO_L/(\varpi_L)^j)$ was defined in Section \ref{nuloc}. If $Z=\AA_{\cO_K}^n$, we will write $E_{L,\cI}^{(r)}(m)$ instead of $E_{L,Z,\cI}^{(r)}(m)$.


\end{defn}
\begin{remark}\label{absexp}
Definition \ref{expodef} is an abstract definition since the usual form of exponential sums was absent. However, we can find an exponential sum form of $E_{L,Z,\cI}^{(r)}(m)$ whenever there is  an affine neighbourhood $U$ of $Z_K$ in $\AA_K^n$ such that $\cI\cO_U$ can be generated by polynomials $f_1,\dots,f_r\in \cO_K[x_1,...,x_n]$ (e.g.,    $r$ is large enough) by observing that  $$E_{L,Z,\cI}^{(r)}(m)=E_{L,Z\mathfrak{D}_{\cO_K}^r,\sum_{1\leq i\leq r}y_if_i(x)}(\psi)$$ for all additive characters $\psi$ of $L$ of conductor $m$ provided that $L\in\tilde{\cL}_{K,M}$ for a large enough integer $M$ independent of $Z$  as seen in the proof of Lemma-Definition \ref{moidef}. Note that we can take $M=1$ if $\cI$ can be generated by $r$ elements.
\end{remark}


Let us relate $\moi_{K,Z}(\cI), \moi_{K,Z}^{\textnormal{naive}}(\cI)$ and $\moi_{K,Z}^{\textnormal{loc}}(\cI)$.
\begin{prop}\label{compa}Let $\cI$ be a non-zero ideal of $\cO_K[x_1,\dots,x_n]$ such that $\cI_K\neq (1)$. Let $X=\spec(\cO_K[x_1,\dots,x_n]/\cI)$ and $Z\subset X$ be an $\cO_K$-scheme of finite type.  We have
$$\moi_{K,Z}^{\textnormal{naive}}(\cI)\geq \moi_{K,Z}^{\textnormal{loc}}(\cI)\geq \moi_{K,Z}(\cI).$$
Moreover, the following claims hold:
\begin{itemize}
\item[(\textit{i}),] $\moi_{K,Z}^{(r)}(\cI)=0$ if $r<n-\dim(X_K)$.
\item[(\textit{ii}),] $\moi_{K,Z}^{\textnormal{naive}}(\cI)=\moi_{K,Z}^{\textnormal{loc}}(\cI)$ if and only if either there is an affine open neighbourhood $V$ of $Z_K$ in $\AA_K^n$ such that $\cI\cO_{\AA_K^n}(V)$ can be generated by $n-\dim_{Z_K}(X_K)$ elements or $\moi_{K,Z}^{\textnormal{naive}}(\cI)=\moi_{K,Z}^{\textnormal{loc}}(\cI)=\lct_Z(\cI)$.
\item[(\textit{iii}),]$\moi_{K,X}^{\textnormal{naive}}(\cI)=\moi_{K,X}(\cI)$ if and only if either  $\cI_K$ can be generated by $n-\dim(X_K)$ elements or $\moi_{K,X}^{\textnormal{naive}}(\cI)=\moi_{K,X}(\cI)=\lct(\cI)$. 
\end{itemize}
\end{prop}
In order to prove Proposition \ref{compa}, we need the following lemma.
\begin{lem}\label{eqzeta}Let $\cI$ be a non-zero ideal of $\cO_K[x_1,\dots,x_n]$ such that $\cI_K\neq (1)$. Let $X=\spec(\cO_K[x_1,\dots,x_n]/\cI)$ and $Z\subset X$ be an $\cO_K$-scheme of finite type. With the notation of Definitions \ref{naiv} and \ref{expodef}, we have 
 \begin{equation}
 F_{r,L,Z,\cI}(s)=q_L^{-(n+s)}\left(1-q_L^{-r}\right)\#Z(k_L)+\left(1-q_L^s\right)\sum_{m\geq 2}E_{L,Z,\cI}^{(r)}(m)q_L^{-ms}
 \end{equation}
 for all $L\in\cL_{K,1}$.
\end{lem}
\begin{proof}Let $L\in\cL_{K,1}$. We recall the map $\pi_{ij}:X(\cO_L/(\varpi_L)^i)\to X(\cO_L/(\varpi_L)^j)$ if $i\geq j\geq 0$ in Section \ref{nuloc}. For each $L\in \cL_{K,1}$ and $m\geq 0$, we set 
$$A_{mL}=\{x\in\cO_L^n\mid\overline{x}\in Z(k_L), \ord_L(\cI_{\cO_L}(x))\geq m\}.$$
Since $Z\subset X$ we have $A_{0L}=A_{1L}$ and
$$\int_{A_{mL}}|dx|=q_L^{-mn}\#\left(\pi_{m1}^{-1}(Z(k_L))\right)$$
for all $L\in\cL_{K,1}$.
These together with the definitions of  $F_{r,L,Z,\cI}(s)$ and $E_{L,Z,\cI}^{(r)}(m)$ yield
\begin{align*}
&F_{r,L,Z,\cI}(s)\\
=&\left(1-q_L^{-(s+r)}\right)\int_{\cO_L^n}\phi_{L,Z}(x)q_L^{-s\ord_L\left(\cI_{\cO_L}(x)\right)}\left|dx\right|\\
=&\left(1-q_L^{-(s+r)}\right)\int_{\{x\in\cO_L^n\mid\overline{x}\in Z(k_L)\}}q_L^{-s\ord_L\left(\cI_{\cO_L}(x)\right)}\left|dx\right|\\
=&\left(1-q_L^{-(s+r)}\right)\Biggl(\sum_{m\geq 1}\int_{x\in A_{mL}\setminus A_{(m+1)L}}q_L^{-sm}\left|dx\right|\Biggr)\\
=&\left(1-q_L^{-(s+r)}\right)\sum_{m\geq 1}\Biggl(\int_{x\in A_{mL}}q_L^{-sm}\left|dx\right|-\int_{x\in A_{(m+1)L}}q_L^{-sm}\left|dx\right|\Biggr)\\
=&\left(1-q_L^{-(s+r)}\right)\Biggl(\int_{x\in A_{1L}}q_L^{-s}\left|dx\right|+\sum_{m\geq 2}\left(\int_{x\in A_{mL}}\left(q_L^{-sm}-q_L^{-s(m-1)}\right)\left|dx\right|\right)\Biggr)\\
=&\left(1-q_L^{-(s+r)}\right)\Bigl(q_L^{-(n+s)}\#Z(k_L)+\sum_{m\geq 2}\#\pi_{m1}^{-1}(Z(k_L))q_L^{-mn}\left(q_L^{-sm}-q_L^{-s(m-1)}\right)\Bigr)\\
=&\left(1-q_L^{-(s+r)}\right)q_L^{-(n+s)}\#Z(k_L)+\left(1-q_L^s\right)\sum_{m\geq 2}\#\pi_{m1}^{-1}(Z(k_L))\left(1-q_L^{-(s+r)}\right)q_L^{-mn}q_L^{-sm}\\
=&\left(1-q_L^{-(s+r)}\right)q_L^{-(n+s)}\#Z(k_L)+\left(1-q_L^s\right)\sum_{m\geq 2}\#\pi_{m1}^{-1}(Z(k_L))\left(q_L^{-m(s+r)}-q_L^{-(m+1)(s+r)}\right)q_L^{-m(n-r)}\\
=&\left(1-q_L^{-(s+r)}\right)q_L^{-(n+s)}\#Z(k_L)+ \left(1-q_L^s\right)q_L^{-(2s+n+r)}\#Z(k_L)\\
&+\left(1-q_L^s\right)\sum_{m\geq 2}\Bigl(q_L^{-m(n-r)}\#\pi_{m1}^{-1}(Z(k_L))-q_L^{-(m-1)(n-r)}\#\pi_{(m-1)1}^{-1}(Z(k_L))\Bigr)q_L^{-m(s+r)}\\
=&q_L^{-(n+s)}\left(1-q_L^{-r}\right)\#Z(k_L)+\left(1-q_L^s\right)\sum_{m\geq 2}q_L^{-mn}\left(\#\pi_{m1}^{-1}\left(Z(k_L)\right)-q_L^{n-r}\#\pi_{(m-1)1}^{-1}\left(Z(k_L)\right)\right)q_L^{-ms}\\
=&q_L^{-(n+s)}\left(1-q_L^{-r}\right)\#Z(k_L)+\left(1-q_L^s\right)\sum_{m\geq 2}E_{L,Z,\cI}^{(r)}(m)q_L^{-ms}.
\end{align*}
\end{proof}
\begin{cor}\label{naig}Let $\cI$ be a non-zero ideal of $\cO_K[x_1,\dots,x_n]$ such that $\cI_K\neq (1)$. Let $X=\spec(\cO_K[x_1,\dots,x_n]/\cI)$ and $Z\subset X$ be an $\cO_K$-scheme of finite type. Suppose that $\cI_K$ can be generated by $r$ elements $f_1,...,f_r\in \cO_K[x_1,\dots,x_n]$. Then we have 
$$\sigma_{L,Z}^{r,\textnormal{naive}}(\cI)=\sigma_{L, Z\mathfrak{D}_{\cO_K}^r}\left(\sum_{i=1}^r y_if_i(x)\right)$$
provided that $L\in\cL_{K,M}$ for a large enough integer $M$. Consequently, we have $\moi_{K,Z}^{r,\textnormal{naive}}(\cI)=\moi_{K,Z\mathfrak{D}_{\cO_K}^r}\left(\sum_{i=1}^r y_if_i(x)\right)=\moi_{K,Z}^{(r)}(\cI).$
Moreover, the supremums  in Lemma-Definition \ref{moidef}, Definitions \ref{moiloc1}, \ref{moiloc2} and Remark \ref{well} are attained. 
\end{cor}
\begin{proof}
By Remark \ref{absexp} we have  
$$E_{L,Z,\cI}^{(r)}(m)=E_{L,Z\mathfrak{D}_{\cO_K}^r,g}(\psi)$$
where $g=\sum_{i=1}^r y_if_i(x)$ and $\psi$ is an additive character of conductor $m$ of $L\in\tilde{\cL}_{K,M}$ for a large enough integer $M$. By \cite[Proposition 1.4.4, Corollary 1.4.5]{DenefBour}, if $L\in\cL_{K,M}$ then $$G_{r,L,Z,\cI}(s):=\sum_{m\geq 2}E_{L,Z,\cI}^{(r)}(m)q_L^{-ms}$$ is a rational function of $q_L^{-s}$ and  $\sigma_{L,Z\mathfrak{D}_{\cO_K}^r}(g)>0$ is the minimum taken over all $\sigma\in \RR\cup\{+\infty\}$ such that either $\sigma=+\infty$ or  $-\sigma$ is the real part of a pole of $G_{r,L,Z,\cI}(s)$. Thus, it follows from Lemma \ref{eqzeta} that 
\begin{equation}\label{equa20}
\sigma_{L,Z}^{r,\textnormal{naive}}(\cI)=\sigma_{L,Z\mathfrak{D}_{\cO_K}^r}(g)
\end{equation}
for all $L\in\cL_{K,M}$. The second claim follows from Lemma-Definition \ref{moidef},  Definitions \ref{dmoif} and \ref{naiv}, and (\ref{equa20}). 

We can conclude immediately that the supremums in Lemma-Definition \ref{moidef}, Definition \ref{moiloc1} and Remark \ref{well} are attained  by using (\ref{equa20}), Proposition \ref{VZ} and Definition \ref{naiv}. For the claim that the supremum in Definition \ref{moiloc2} is attained, we only need to observe moreover that if $n'\geq n$ and $\iota:X\to \AA_{\cO_K}^{n'}$ is a closed embedding whose image is the closed subscheme $Y$ of $\AA_{\cO_K}^{n'}$ associated to an ideal $\cI'$ of $\cO_K[y_1,...,y_{n'}]$ then Definition \ref{expodef} implies 
$$E_{L,Z,\cI}^{(r)}(m)=q_L^{m(n'-n)}E_{L,\iota(Z),\cI'}^{(r+n'-n)}(m)$$
and 
$$G_{r,L,Z,\cI}(s)=G_{r+n'-n,L,\iota(Z),\cI'}(s-n'+n).$$

\end{proof}
\begin{cor}\label{coboun}
Let $\cI$ be a non-zero ideal of $\cO_K[x_1,\dots,x_n]$ such that $\cI_K\neq (1)$. Let $X=\spec(\cO_K[x_1,\dots,x_n]/\cI)$. There is a rational constant $c$ and an integer $M$ such that $$\left|E_{L,Z,\cI_{\cO_{K'}}}^{(r)}(m)\right|\leq q_L^cm^{n-1}q_L^{-\sigma_{L,Z}^{r,\textnormal{naive}}(\cI)m}$$
 for all finite extensions $K'$ of $K$, all $\cO_{K'}$-schemes $Z\subset X_{\cO_{K'}}$ of finite type, all $L\in \cL_{K',M}$ and all $m\geq 2$. 
 
 In particular, for each $\cO_K$-scheme $Z\subset X$ of finite type, there is a rational constant $c$ and an integer $M$ such that
 $$\left|E_{L,Z,\cI}^{(r)}(m)\right|\leq q_L^cm^{n-1}q_L^{-\moi_{K,Z}^{r,\textnormal{naive}}(\cI)m}$$
 for all $L\in \cL_{K,M}$ and all $m\geq 2$.
 
 Moreover, for each $\cO_K$-scheme $Z\subset X$ of finite type, there is a strictly increasing sequence $(m_i)_{i\geq  1}$ of natural numbers,  a set $\cA_{i}\subset \cL_{K,1}$  for each $i\geq 1$ and a rational constant $c$ such that $\cA_i\cap \cL_{K',M}\neq \emptyset$ for all $i,M\geq 1$ and all $K'\in\mathfrak{A}_{Z}$, and we have
 $$\left|E_{L,Z,\cI}^{(r)}(m_i)\right|\geq q_{L}^cq_{L}^{-\moi_{K,Z}^{r,\textnormal{naive}}(\cI)m_i}$$
 for all $i\geq 1$ and all $L\in\cA_i$, where $\mathfrak{A}_{Z}$ is the set consisting of all finite extensions $K'$ of $K$ such that  $\moi_{K,Z}^{r,\textnormal{naive}}(\cI)=\moi_{K',Z_{\cO_{K'}}}^{r,\textnormal{naive}}(\cI_{\cO_{K'}})$.

\end{cor}
\begin{proof}Let $h$ be a log resolution of $\cI_K$. Let $((\nu_i,N_i))_{i\in I}$ be the numerical data of $h$.  Let $K'$ be a finite extension of $K$ and $Z\subset X_{\cO_{K'}}$ be an $\cO_{K'}$-scheme of finite type. By using \cite[Theorem 2.10]{VeysZ}, we can give an explicit formula of $F_{r,L,Z,\cI_{\cO_{K'}}}(s)$ associated to $h$ uniformly in  $L\in\cL_{K',M}$ for an integer $M$ depending only on $h$. This and Lemma \ref{eqzeta} yield a rational function $G_{L,Z,\cI}(T)$ whose power series expansion must be $\sum_{m\geq 2}E_{L,Z,\cI_{\cO_{K'}}}^{(r)}(m)T^{m}$. Moreover, the denominator of $G_{L,Z,\cI}(T)$ can be taken to be the product of factors of the form $1-q_L^{-a}T^b$ for $(a,b)\in ((\nu_i,N_i))_{i\in I}$ such that $\frac{-a}{b}$ is the real part of a pole of  $F_{r,L,Z,\cI_{\cO_{K'}}}(s)$ (see also Proposition \ref{VZ}). By writing  $G_{L,Z,\cI}(T)$ in partial fractions, we can expand it as a power series in $T$ and compute its coefficients $C_{L,Z}(m)$ uniformly in $L, Z$ and $m$. The result is quite similar to \cite[Corollary 1.4.5]{DenefBour} (see also \cite[Corollary 2.2.4]{Saskia-Kien}).  More precisely, let $\textnormal{Pol}_Z(L)\subset \{-\frac{\nu_i}{N_i}\mid i\in I\}$ be the set consisting of the real part of poles of $F_{r,L,Z,\cI_{\cO_{K'}}}(s)$. Then for each $0\leq i\leq n-1$ and $\lambda\in \textnormal{Pol}_Z(L)$, there is a positive integer $\ell_{i\lambda}$, a Presburger subset $A_{i\lambda}$ of $\NN$, rational numbers $c_{i\lambda\ell},b_{i\lambda\ell}$ together with $\AA_{\cO_K}^n$-schemes  $X_{i\lambda\ell}$ of finite type and non-zero polynomials $H_{i\lambda\ell}(y)\in \QQ[y]$ for $1\leq \ell\leq \ell_{i\lambda}$, all of which are independent of $Z, K',L$ such that
\begin{equation}\label{expressione}
C_{L,Z}(m)=\sum_{\lambda\in \textnormal{Pol}_Z(L)}\sum_{0\leq i\leq n-1}\11_{A_{i\lambda}}(m)\sum_{1\leq \ell\leq \ell_{i\lambda}}\frac{c_{i\lambda\ell}}{H_{i\lambda\ell}(q_L)}\#X_{i\lambda\ell Z_{k_L}}(k_L)m^{i}q_L^{b_{i\lambda\ell}+\lambda m}
\end{equation}for all $L\in \cL_{K',M}$ and all $m\geq 1$,  where $X_{i\lambda\ell Z_{k_L}}= X_{i\lambda\ell}\times_{\AA_{\cO_{K}}^n}Z_{k_L}$ and a Presburger subset of $\NN$ is a set which can be written as a finite Boolean combination of sets of the form $\{ m\in \NN\mid m_1\leq m, m\equiv m_2 \mod m_3\}$ for integers $m_1,m_2,m_3$. Moreover, each $\lambda\in \textnormal{Pol}_Z(L)$ appears non-trivially in (\ref{expressione}) when $m\to+\infty$, i.e., for each integer $N>0$, $\lambda$ appears non-trivially in the restriction of $C_{L,Z}$ to the interval $[N,+\infty)$.  It follows from (\ref{expressione}) and Proposition  \ref{Lang-Weil} below  that there is a rational constant $c$ satisfying
 $$\left|C_{L,Z}(m)\right|\leq q_L^cm^{n-1}q_L^{-\sigma_{L,Z}^{r,\textnormal{naive}}(\cI)m}$$
 for all finite extensions $K'$ of $K$, all $\cO_{K'}$-schemes $Z\subset X_{\cO_{K'}}$ of finite type, all $L\in \cL_{K',M}$ and all $m\geq 1$. On the other hand, $C_{L,Z}(m)=E_{L,Z,\cI_{\cO_{K'}}}^{(r)}(m)$ for $m\geq 2$ as mentioned above. This implies the first assertion. 
 
The second assertion follows from the first assertion if we can show that there is an integer $M_0$ such that 
\begin{equation}\label{Meq}
\moi_{K,Z}^{r,\textnormal{naive}}(\cI)=\min_{L\in\cL_{K,M'}}\sigma_{L,Z}^{r,\textnormal{naive}}(\cI),
\end{equation}
for all $M'>M_0$. By Proposition \ref{VZ} and Definition \ref{naiv}, the set $\{\sigma_{L,Z}^{r,\textnormal{naive}}\mid L\in \cL_{K,1}\}$ is contained in the finite set $\{\frac{\nu_i}{N_i}| i\in I\}$. This and the definition of $\moi_{K,Z}^{r,\textnormal{naive}}(\cI)$ imply $(\ref{Meq})$. 

In order to prove the last assertion, we use (\ref{expressione}) again. Let $Z$ be an $\cO_K$-subscheme of $\AA_{\cO_K}^n$ of finite type, we set $\lambda_0=-\moi_{K,Z}^{r,\textnormal{naive}}(\cI)$, $A_{\lambda_0}=\cup_{i=0}^{n-1}A_{i\lambda_0}$ and $\cL_{K,M'}^{(Z)}=\{ L\in \cL_{K,M'}| \sigma_{L,Z}^{r,\textnormal{naive}}(\cI)=\moi_{K,Z}^{r,\textnormal{naive}}(\cI)\}$ for each integer $M'$. By (\ref{Meq}), we have $\cL_{K,M'}^{(Z)}\neq \emptyset$ for all $M'$. Since  $\lambda_0$ appears non-trivially in (\ref{expressione}) whenever $L\in \cL_{K,M}^{(Z)}\neq\emptyset$, $A_{\lambda_0}$ must be an infinite set. It follows from (\ref{expressione}) that there is a rational constant $c$ such that if $m\in A_{\lambda_0}$, $m$ is large enough and $L\in\cL_{K,M}^{(Z)}$ then either $$f_{Z,L}(m):=\sum_{0\leq i\leq n-1}\11_{A_{i\lambda_0}}(m)\sum_{1\leq \ell\leq \ell_{i\lambda_0}}\frac{c_{i\lambda_0\ell}}{H_{i\lambda_0\ell}(q_L)}\#X_{i\lambda_0\ell Z_{k_L}}(k_L)m^{i}q_L^{b_{i\lambda_0\ell}}=0$$
or 
$$\left|C_{L,Z}(m)\right|\geq q_L^cq_L^{-\moi_{K,Z}^{r,\textnormal{naive}}(\cI) m}.$$
We set $$B=\{m\in A_{\lambda_0}\mid\exists M'>M \textnormal{ such that } f_{Z,L}(m)=0 \textnormal{ for all } L\in \cL_{K,M'}^{(Z)}\}.$$
Since $A_{\lambda_0}$ is infinite,  it is clear that the last assertion follows if we can show that $\#B<n$. Suppose that $\#B\geq n$. Let $m_1,...,m_n$ be distinct elements of $B$. By the definition of $B$, for each $1\leq i\leq n$, there is $M_i$ such that $f_{Z,L}(m_i)=0$ for all $L\in \cL_{K,M_i}^{(Z)}$. Let us put $M_0=\max_{1\leq i\leq n}M_{i}$. Then  $f_{Z,L}(m_i)=0$ for all $1\leq i\leq n$ and all $L\in \cL_{K,M_0}^{(Z)}$. Since $f_{Z,L}(m)$ is a polynomial in $m$ of degree at most $n-1$, we must have $f_{Z,L}(m)=0$ for all $m$ and all $L\in\cL_{K,M_0}^{(Z)}$.  Thus, $\lambda_0$ appears trivially in (\ref{expressione}) whenever $L\in \cL_{K,M_0}^{(Z)}$. It implies that $\cL_{K,M_0}^{(Z)}=\emptyset$, a contradiction.  Thus, $\#B<n$ as desired.
\end{proof}
\begin{proof}[Proof of Proposition \ref{compa}]
The inequality $\moi_{K,Z}^{\textnormal{loc}}(\cI)\geq \moi_{K,Z}(\cI)$ follows from Lemma-Definition \ref{moidef} and Definition \ref{moiloc1}. Similarly, the claim $(i)$ is obvious by Lemma-Definition \ref{moidef}.

Let $r_0=n-\dim_{Z_K}(X_K)$. Let $\sigma_{0L}$ be the minimum taken over all real numbers $\sigma$ such that $-\sigma$ is the real part of a pole of the Igusa local zeta function
$$\cZ_{L,Z,\cI}(s)=\int_{\cO_L^n}\phi_{L,Z}(x)q_L^{-s\ord_L\left(\cI_{\cO_L}(x)\right)}\left|dx\right|.$$
By Proposition \ref{VZ}, we have $\sigma_{0L}\geq \lct_{Z}(\cI)$ for all $L\in\cL_{K,1}$ and $\sigma_{0L}=\lct_{Z}(\cI)$ if $L\in\cL_{K',M}$ for a finite extension $K'$ of $K$ and a large enough integer $M$.  Thus, $\sigma_{0L}=\lct_{Z}(\cI)\leq r_0$ if $L\in \cL_{K',M}$ as mentioned in Section \ref{SJ}. 

Let $L\in\cL_{K,1}$ and $r$ be a positive integer, we use the notation $\sigma_{L,Z}^{r,\textnormal{naive}}(\cI)$ in Definition \ref{naiv}. Obviously, we have $\sigma_{L,Z}^{r,\textnormal{naive}}(\cI)\geq\sigma_{0L}$ for all $r$ and all $L\in\cL_{K,1}$. This and Definition \ref{naiv} imply 
\begin{equation}\label{nalo}
\sigma_{L,Z}^{r_0,\textnormal{naive}}(\cI)\geq \sigma_{0L}\geq \lct_{Z}(\cI)
\end{equation}
and 
\begin{equation}\label{nalom}
\moi_{K,Z}^{\textnormal{naive}}(\cI)\geq \lct_{Z}(\cI).
\end{equation}

 If $\sigma_{0L}\neq r$, then it is clear that $\sigma_{L,Z}^{r,\textnormal{naive}}(\cI)=\sigma_{0L}$. If $L\in \cL_{K',M}$, then $\sigma_{0L}=\lct_{Z}(\cI)\leq r_0$. Thus, we must have 
$$\sigma_{L,Z}^{r,\textnormal{naive}}(\cI)=\lct_{Z}(\cI)$$
for all $r>r_0$ and all $L\in \cL_{K',M}$. Therefore, if $r>r_0$ then we have  
\begin{equation}\label{equa1}
\liminf_{L\in\cL_{K,M}, M\to +\infty}\sigma_{L,Z}^{r,\textnormal{naive}}(\cI)= \lct_{Z}(\cI).
\end{equation}

Suppose that there exists a tuple $(U,f_1,\dots,f_r)\in\cA(X,Z)$. Then it is easy to see that $r\geq r_0$ (see, e.g., \cite[Chapter 1, Proposition 7.1]{Harts}). It follows from Remark \ref{absexp} and Corollary \ref{naig} that 
\begin{equation}\label{equa2}
\sigma_{L,Z}^{r,\textnormal{naive}}(\cI)=\sigma_{L,Z\mathfrak{D}_{\cO_K}^r}(g)
\end{equation}
for all $L\in\cL_{K,M}$ provided that $M$ is large enough.
If $r>r_0$ then Definition \ref{dmoif}, (\ref{equa1}) and (\ref{equa2}) imply 
\begin{equation}\label{equa3}
\moi_{K,Z\mathfrak{D}_{\cO_K}^r}(g)=\lct_{Z}(\cI).
\end{equation}
If $r=r_0$ then Definitions \ref{dmoif} and \ref{naiv}, (\ref{nalom}) and (\ref{equa2}) imply 
\begin{equation}\label{equa4}
\moi_{K,Z\mathfrak{D}_{\cO_K}^r}(g)=\moi_{K,Z}^{\textnormal{naive}}(\cI)\geq \lct_{Z}(\cI).
\end{equation}
By (\ref{nalom}), (\ref{equa3}) and (\ref{equa4}), we always have 
\begin{equation}\label{eq100}
\lct_Z(\cI)\leq \moi_{K,Z\mathfrak{D}_{\cO_K}^r}(g)\leq \moi_{K,Z}^{\textnormal{naive}}(\cI).
\end{equation}
The inequalities \begin{equation}\label{eq101}
\moi_{K,Z}^{\textnormal{naive}}(\cI)\geq \moi_{K,Z}^{\textnormal{loc}}(\cI)\geq\lct_Z(\cI)
\end{equation}
follow from Definition \ref{moiloc1} and (\ref{eq100}).

If $\moi_{K,Z}^{\textnormal{naive}}(\cI)=\lct_{Z}(\cI)$, then (\ref{eq101}) implies that $$\moi_{K,Z}^{\textnormal{naive}}(\cI)=\moi_{K,Z}^{\textnormal{loc}}(\cI)=\lct_{Z}(\cI).$$
If $\moi_{K,Z}^{\textnormal{naive}}(\cI)>\lct_{Z}(\cI)$, then  by using Definition \ref{moiloc1}, (\ref{equa3}) and (\ref{equa4}), $$\moi_{K,Z}^{\textnormal{naive}}(\cI)=\moi_{K,Z}^{\textnormal{loc}}(\cI)$$
if and only if we can find an element $(U,f_1,\dots,f_r)\in\cA(X,Z)$ such that $r=r_0$. The claim $(ii)$ is proved. 

Similarly, we can prove $(iii)$ by repeating the above argument  for $\moi_{K,X}(\cI)$ instead of $\moi_{K,Z}^{\textnormal{loc}}(\cI)$.
\end{proof}


\begin{cor}\label{llct}Let $\cI$ be a non-zero ideal of $\cO_K[x_1,\dots,x_n]$ such that $\cI_K\neq (1)$. Let $X=\spec(\cO_K[x_1,\dots,x_n]/\cI)$ and $Z\subset X$ be an $\cO_K$-scheme of finite type. The following claims hold:
\begin{itemize}
\item[(\textit{i}),] $\lct_{Z}(\cI)\leq \moi_{K,Z}^{(r)}(\cI)$ if $\moi_{K,Z}^{(r)}(\cI)\neq 0$.
\item[(\textit{ii}),]If $\moi_{K,Z}^{(r)}(\cI)>0$, then $\moi_{K,Z}^{(\ell)}(\cI)=\lct_{Z}(\cI)$ for all $\ell>r$.
\item[(\textit{iii}),]If $\moi_{K,Z}^{(r)}(\cI)>\lct_{Z}(\cI)$, then $\moi_{K,Z}^{(\ell)}(\cI)=0$ for all $\ell<r$ and $\moi_{K,Z}^{(r)}(\cI)=\moi_{K,Z}(\cI)$.
\item[(\textit{iv}),]If $n-\dim_{Z_K}(X_K)>\lct_{Z}(\cI)$, then  $\lct_{Z}(\cI)=\moi_{K,Z}(\cI)$.
\item[(\textit{v}),]If $r>\lct_{Z}(\cI)$  and $\moi_{K,Z}^{(r)}(\cI)\neq 0$, then  $\lct_{Z}(\cI)=\moi_{K,Z}^{(r)}(\cI)$.
\item[(\textit{vi}),]$\sigma_{L,Z}^{r,\textnormal{naive}}(\cI)\geq\lct_Z(\cI)$ for all $r\geq 1$ and all $L\in\cL_{K,1}$.
\item[(\textit{vii}),] Let $h$ be a log resolution of $\cI_K$ associated to irreducible divisors $(E_i)_{i\in I}$ and numerical data $((\nu_i,N_i))_{i\in I}$. Then the motivic oscillation indexes $\moi_{K,Z}^{(r)}(\cI), \moi_{K,Z}(\cI),\moi_{K,Z}^{\textnormal{naive}}(\cI)$, $\moi_{K,Z}^{r,\textnormal{naive}}(\cI)$ and $\moi_{K,Z}^{\textnormal{loc}}(\cI)$  belong to $\{\frac{\nu_i}{N_i}\mid i\in I\}$.
\end{itemize} 
\end{cor}
\begin{proof}
By Corollary \ref{naig}, Lemma-Definition \ref{moidef} and Definition \ref{naiv}, if $r\geq 1$ then either $\moi_{K,Z}^{(r)}(\cI)=0$ or $\moi_{K,Z}^{(r)}(\cI)=\moi_{K,Z}^{r,\textnormal{naive}}(\cI)$. By Lemma-Definition \ref{moidef} and \cite[Chapter 1, Proposition 7.1]{Harts}, there exists $r_0\geq n-\dim_{Z_K}(X_K)$ such that $\moi_{K,Z}^{(r)}(\cI)=0$ if $r<r_0$ and $\moi_{K,Z}^{(r)}(\cI)>0$ if $r\geq r_0$.  On the other hand, by Definition \ref{naiv} and Proposition \ref{VZ}, $\moi_{K,Z}^{r,\textnormal{naive}}(\cI)\geq \lct_Z(\cI)$ for all $r\geq 1$ and if the equality does not hold then $r=\lct_Z(\cI)$. As mentioned in Section \ref{SJ}, we have $\lct_Z(\cI)\leq n-\dim_{Z_K}(X_K)$. These facts together with Lemma-Definition \ref{moidef} imply the items $(i),(ii),(iii),(iv),(v)$. 
 
The item $(vi)$ follows immediately from Definition \ref{naiv} and Proposition \ref{VZ}. 

Lastly,  we can deduce the item $(vii)$ by using Proposition \ref{VZ}, Corollary \ref{naig}, Lemma-Definition \ref{moidef}, Definitions \ref{dmoif}, \ref{naiv} and \ref{moiloc1}.
\end{proof}
We can understand  the (local) motivic oscillation indexes of ideals better by the following propositions.
\begin{prop}\label{limitp}Let $f_1,\dots,f_r\in \cO_K[x_1,\dots,x_n]$, $g=\sum_{i=1}^ry_if_i(x)$ and $Z\subset\AA_{\cO_K}^n$ be an $\cO_K$-scheme of finite type. Then we have
$$\moi_{K,Z\mathfrak{D}_{\cO_K}^r}(g)=\liminf_{\mathfrak{p}\in\Specm(\cO_K), p_{K_\mathfrak{p}}\to\infty} \sigma_{K_\mathfrak{p},Z\mathfrak{D}_{\cO_K}^r}(g).
$$

\end{prop}
\begin{proof} By Corollary \ref{naig}, it is sufficient to show that 
$$\moi_{K,Z}^{r,\textnormal{naive}}(\cI)=\liminf_{\mathfrak{p}\in\Specm(\cO_K), p_{K_\mathfrak{p}}\to\infty} \moi_{K_\mathfrak{p},Z}^{r,\textnormal{naive}}(\cI)$$
for all $r\geq 1$.
Let $\pi: Y\to \AA^n_K$ be a log resolution of $\cI_K$ and $\sigma$ be a real number. By \cite[Theorem 2.10]{VeysZ} and using the partial fraction decomposition of rational functions over the field $\QQ(y)$ of the form 
$$(1-Ty^{-r})\prod_{1\leq i\leq e}\frac{(y-1)y^{-\nu_i}T^{N_i}}{1-T^{N_i}y^{-\nu_i}}$$
for positive integers $\nu_i$ and $N_i$, there is an integer $M$, a non-zero polynomial $H(y)\in \QQ[y]$, polynomials $Q_j(y,T)\in \QQ[y,T]$ with $\deg_T(Q_j)=d_j\geq 1$ and $\cO_K$-schemes of finite type $U_{0,0}, V_{0,0}, U_{0,1}, V_{0,1},$ $U_{j,0},V_{j,0},\dots,U_{j,d_j-1},V_{j,d_j-1}$  for $1\leq j\leq \tilde{\ell}$, all of them depend only on $\pi$ and $\cI$ such that the following conditions hold:
\begin{itemize}
\item $H(q_L)\neq 0$ if $L\in\cL_{K,M}$.
\item $Q_j(y,T)$ is a power of an irreducible polynomial in $T$ over the field $\QQ(y)$ for all $1\leq j\leq \tilde{\ell}$.
\item For each $1\leq j\leq\tilde{\ell}$, $\QQ_j(y,T)\mid (T^{N_j}-y^{\nu_j})^{e_j}$ for some positive integers $\nu_j,N_j,e_j$.
\item $Q_j(y,T)$ and $Q_{j'}(y,T)$ are coprime in the ring $\QQ(y)[T]$ if $1\leq j\neq j'\leq \tilde{\ell}$.
\item For all local fields $L\in\cL_{K,M}$, we have 
$$F_{r,L,Z,\cI}(s)=\frac{1}{H(q_L)}\sum_{0\leq j\leq \tilde{\ell}}\frac{\sum_{i=0}^{d_j-1}(\#(U_{j,i}(k_L))-\#(V_{j,i}(k_L)))q_L^{-is}}{Q_j(q_L,q_L^{-s})},$$
where $$F_{r,L,Z,\cI}(s)=\left(1-q_L^{-(s+r)}\right)\int_{\cO_L^n}\phi_{L,Z}(x)q_L^{-s\ord_L\left(\cI_{\cO_L}(x)\right)}\left|dx\right|$$ as in Definition \ref{naiv} and we use the convention that $Q_0(y,T)=1, d_0=2$.
\end{itemize}
If $1\leq j\neq j'\leq \tilde{\ell}$, by using the resultant $res_T(Q_j(y,T),Q_{j'}(y,T))$ and the fact that $Q_j(y,T)$ and $Q_{j'}(y,T)$ are coprime in the ring $\QQ(y)[T]$, we can enlarge $M$ if needed to have that $Q_j(q_L,T)$ and $Q_{j'}(q_L,T)$ are coprime in the ring $\CC[T]$ whenever $L\in \cL_{K,M}$. Therefore, for each real number $\sigma$, there are $\cO_K$-schemes of finite type $X_1, Y_1, \dots, X_\ell, Y_\ell$ depending only on $\sigma,\cI$ and $\pi$ such that for all local fields $L\in\cL_{K,M}$, the following conditions are equivalent:
\begin{itemize}
\item[(\textit{i}),] There exists a pole $s_0$ of $F_{r,L,Z,\cI}(s)$ such that $\Re(s_0)=-\sigma$.
\item[(\textit{ii}),] There exists $1\leq j\leq \ell$ such that $\#X_j(k_L)\neq \#Y_j(k_L)$.
\end{itemize}
Now, let $\sigma=\moi_{K,Z}^{r,\textnormal{naive}}(\cI)$, then by its definition, we have 
$$\sigma\leq \liminf_{\mathfrak{p}\in\Specm(\cO_K), p_{K_\mathfrak{p}}\to\infty} \moi_{K_\mathfrak{p},Z}^{r,\textnormal{naive}}(\cI).$$
Suppose that 
$$\sigma<\liminf_{\mathfrak{p}\in\Specm(\cO_K), p_{K_\mathfrak{p}}\to\infty} \moi_{K_\mathfrak{p},Z}^{r,\textnormal{naive}}(\cI),
$$
then the condition $(i)$ does not hold for $\sigma$ and $K_{\mathfrak{p}}$ provided that $p_{K_\mathfrak{p}}$ is large enough. Thus, we must have that  $\#X_j(k_{K_\mathfrak{p}})=\#Y_j(k_{K_\mathfrak{p}})$
for all $1\leq j\leq \ell$ and all but finitely many primes $\mathfrak{p}\in\Specm(\cO_K)$. By \cite[Theorem 6.6]{NXp}, there exists an integer $M_1$ such that $\#X_j(k_{L})=\#Y_j(k_{L})$ for all $1\leq i\leq \ell$ and all $L\in \cL_{K,M_1}$. Then the condition $(i)$ does not hold for $\sigma$ and $L\in \cL_{K,M_1}$. By Definition \ref{naiv}, we have $\sigma\neq \moi_{K,Z}^{r,\textnormal{naive}}(\cI)$. This is a contradiction.
\end{proof}
\begin{prop}\label{gl-lc}
Let $\cI=(f_1,\dots,f_r)$ be a non-zero ideal of $\cO_K[x_1,\dots,x_n]$ such that $\cI_K\neq (1)$. Let $X=\Spec(\cO_K[x_1,\dots,x_n]/\cI)$ and $Z\subset X$ be an $\cO_K$-scheme of finite type. Then we have 
$$\moi^{(r)}_{K,Z}(\cI)\geq \min_{P\in Z(\overline{K})} \moi^{(r)}_{K(P),\tilde{P}}(\cI_{\cO_{K(P)}}),$$
where $K(P)$ is the extension of $K$ by the coordinates of $P=(P_1,...,P_n)$ and if $\ell$ is the smallest positive integer such that $\ell P\in \cO_{K(P)}^n$ then $\tilde{P}=\Spec(\cO_{K(P)}[x_1,...,x_n]/(\ell x_1-\ell P_1,\dots,\ell x_n-\ell P_n))$.

As a consequence, we have 
$$\moi^{\loc}_{K,Z}(\cI)\geq \min_{P\in Z(\overline{K})} \moi^{(r)}_{K(P),\tilde{P}}(\cI_{\cO_{K(P)}}).$$
\end{prop}
\begin{proof}We can write $Z_K=\sqcup_{i=1}^{\ell}Z_i$ such that $(Z_i)_{red}$ is a smooth $K$-variety for all $i$. For each $1\leq i\leq \ell$, we take an $\cO_K$-scheme $\tilde{Z}_i\subset\AA_{\cO_K}^n$ of finite type such that $(Z_i)_{red}=\tilde{Z}_i\otimes K$. Since $Z_i\cap (\cup_{j\neq i}Z_{j})=\emptyset$, we have $(\tilde{Z_i}\setminus \cup_{j\neq i}\tilde{Z}_j)\otimes K=(Z_i)_{red}$. Replacing $\tilde{Z_i}$ by $\tilde{Z_i}\setminus \cup_{j\neq i}\tilde{Z}_j$ if needed, we can suppose that $\tilde{Z}_i\cap \tilde{Z}_j=\emptyset$ if $i\neq j$.  We put $\tilde{Z}=\cup_{1\leq i\leq \ell}\tilde{Z}_i$. By Lemma-Definition \ref{moidef}, we have $\moi^{(r)}_{K,Z}(\cI)=\moi^{(r)}_{K,\tilde{Z}}(\cI)$. Moreover, it is easy to show that $Z(\overline{K})=\cup_{1\leq i\leq\ell}\tilde{Z}_i(\overline{K})$. Thus, it is sufficient to prove the following facts:
\begin{itemize}
\item[\textit{i},]$\moi^{(r)}_{K,\tilde{Z}}(\cI)\geq\min_{1\leq i\leq\ell}\moi^{(r)}_{K,\tilde{Z}_i}(\cI)$
\item[\textit{ii},]$\moi^{(r)}_{K,\tilde{Z}_i}(\cI)\geq \min_{P\in \tilde{Z}_i(\overline{K})} \moi^{(r)}_{K(P),\tilde{P}}(\cI_{\cO_{K(P)}})$
for all $i$.
\end{itemize}
We set $g=\sum_{1\leq i\leq r}y_if_i(x)$ as usual. We observe that the conditions $\tilde{Z}_i\cap\tilde{Z}_j=\emptyset$ if $i\neq j$ and $\tilde{Z}=\cup_{1\leq i\leq \ell}\tilde{Z}_i$ imply $\phi_{L,\tilde{Z}}=\sum_{1\leq i\leq\ell}\phi_{L,\tilde{Z}_i}$ for all $L\in\cL_{K,1}$. Thus, we have 
$$E^{(r)}_{L,\tilde{Z},\cI}(m)=\sum_{1\leq i\leq \ell}E^{(r)}_{L,\tilde{Z}_i,\cI}(m)$$
for all $L\in\cL_{K,1}$ and all $m\geq 1$. 

By Corollary \ref{coboun}, there is a rational constant $c$ and an integer $M$ depending only on $\cI$ satisfying
 $$\left|E_{L,Y,\cI_{\cO_{K'}}}^{(r)}(m)\right|\leq q_L^cm^{n-1}q_L^{-\sigma_{L,Y}^{r,\textnormal{naive}}(\cI)m}$$
 for  all finite extensions $K'$ of $K$, all $\cO_{K'}$-schemes $Y\subset X_{\cO_{K'}}$ of finite type, all $L\in \cL_{K',M}$ and all $m\geq 2$. By this inequality, Remark \ref{absexp} and Corollary \ref{naig}, we have
\begin{equation}\label{rough}
\left|E_{L,Y\mathfrak{D}_{\cO_{K'}}^r,g}(\psi)\right|=\left|E^{(r)}_{L,Y,\cI_{\cO_{K'}}}(m)\right|\leq q_L^cm^{n-1}q_L^{-\sigma_{L,Y\mathfrak{D}_{\cO_K}^r}(g)m}
\end{equation}
for all finite extensions $K'$ of $K$, all $\cO_{K'}$-schemes $Y\subset X_{\cO_{K'}}$ of finite type, all $m\geq 2$, all $L\in \cL_{K',M}$ and all additive characters $\psi$ of $L$ of conductor $m$. Therefore, we have 
\begin{equation}\label{rough1}
\left|E^{(r)}_{L,\tilde{Z},\cI}(m)\right|=\bigg|\sum_{1\leq i\leq \ell}E^{(r)}_{L,\tilde{Z}_i,\cI}(m)\bigg|\leq \ell q_L^cm^{n-1}q_L^{-\min_{1\leq i\leq \ell}\sigma_{L,\tilde{Z}_i\mathfrak{D}_{\cO_K}^r}(g)m}
\end{equation}
for all $m\geq 2$ and all $L\in\cL_{K,M}$. Let $((\nu_i,N_i))_{i\in I}$ be the numerical data of $h$. By Proposition \ref{VZ}, Definition \ref{naiv} and Corollary \ref{naig}, we have 
$$\{\sigma_{L,\tilde{Z}_i\mathfrak{D}_{\cO_K}^r}(g)\mid 1\leq i\leq\ell, L\in\cL_{K,M}\}\subset \{\frac{\nu_i}{N_i}\mid i\in I\}.$$
 This fact together with the finiteness of $I$,  Definition \ref{dmoif},  Lemma-Definition \ref{moidef}, (\ref{rough}) and (\ref{rough1}) imply that 
$$\moi^{(r)}_{K,\tilde{Z}}(\cI)\geq \liminf_{L\in\cL_{K,M'},M'\to+\infty}\min_{1\leq i\leq \ell}\sigma_{L,\tilde{Z}_i\mathfrak{D}_{\cO_K}^r}(g)=\min_{1\leq i\leq\ell}\moi^{(r)}_{K,\tilde{Z}_i}(\cI).$$
Thus, $(\textit{i})$ is proved.

Now, we will prove $(\textit{ii})$.  Note that $\tilde{Z}_i\otimes K$ is smooth. Thus, the logical compactness in \cite[Corollary 2.2.10]{Marker} implies that there exists $M_1>M$ such that $\tilde{Z}_i\otimes k_L$ is smooth for all $L\in\cL_{K,M_1}$. Therefore, Hensel's lemma and the fact that $\overline{K}\cap L$ is dense in $L$ for all $L\in\cL_{K,1}$ imply that if $L\in \cL_{K,M_1}$, then every point $P\in \tilde{Z}_i(k_L)$ can be lifted  to a point $\underline{P}\in \tilde{Z}_i(\overline{K})\cap \tilde{Z}_i(\cO_L)$. So we have 
$$E^{(r)}_{L,\tilde{Z}_i,\cI}(m)=\sum_{P\in \tilde{Z}_i(k_L)}E^{(r)}_{L,\tilde{\underline{P}},\cI_{\cO_{K(\underline{P})}}}(m)$$
for all $L\in\cL_{K,M_1}$ and all $m\geq 1$. By combining this with $(\ref{rough})$, we have 
\begin{align*}
\left|E^{(r)}_{L,\tilde{Z}_i,\cI}(m)\right|&\leq \#\tilde{Z}_i(k_L)q_L^cm^{n-1}q_L^{-\min_{P\in\tilde{Z}_i(\overline{K})}\sigma_{L,\tilde{P}\mathfrak{D}_{\cO_{K(P)}}^r}(g)m}\\
&\leq q_L^{c+n}m^{n-1}q_L^{-\min_{P\in\tilde{Z}_i(\overline{K})}\sigma_{L,\tilde{P}\mathfrak{D}_{\cO_{K(P)}}^r}(g)m}
\end{align*}
for all $L\in \cL_{K,M_1}$ and all $m\geq 1$. Hence, we can repeat the above argument for $(\textit{i})$ to obtain $(ii)$.

By Lemma-Definition \ref{moidef} and  Definition \ref{moiloc1}, we have $\moi^{\loc}_{K,Z}(\cI)\geq \moi^{(r)}_{K,Z}(\cI)$.  Thus, the consequence follows immediately from the above assertion.
\end{proof}
\subsection{Conjectures for exponential sums}\label{Conjexp}
Let $\cI$ be a non-zero ideal of $\cO_K[x_1,\dots,x_n]$ such that $\cI_K\neq (1)$ and $Z\subset \AA_{\cO_K}^n$ be an $\cO_K$-scheme of finite type. In this paper, we will formulate a conjecture for the exponential sums of $\cI$ at $Z$ by specializing Conjecture \ref{Iguconj} to the case where $f=\sum_{i=1}^ry_if_i$ and $\tilde{Z}=Z\mathfrak{D}_{\cO_K}^r$ if $(f_1,\dots,f_r)=\cI_K$. Definition \ref{expodef} and Remark \ref{absexp} suggest that we can state our conjecture without mentioning $f_1,\dots,f_r$.
\begin{conj}[General form of averaged Igusa conjecture for exponential sums]\label{avelocIgu}Let $\cI$ be a non-zero ideal of $\cO_K[x_1,\dots,x_n]$ such that $\cI_K\neq (1)$ and $Z\subset\AA_{\cO_K}^n$ be an $\cO_K$-scheme of finite type. For each positive integer $r$, there is an integer $M$ depending only on $\cI$ and a positive constant $c_r$ depending only on $\cI, r$  such that for all local fields $L\in\tilde{\cL}_{K,M}$ and all $m\geq 2$, we have 
$$\left|E_{L,Z,\cI}^{(r)}(m)\right|\leq c_rm^{n+r-1}q_L^{-m\moi_{K,Z}^{(r)}(\cI)}.$$
\end{conj}
\begin{remark}
The word ''averaged '' comes from the fact that if $r$ is large enough, then for a suitable choice of polynomials $f_1,\dots,f_r$ and an additive character $\psi$ of $L\in\tilde{\cL}_{K,1}$ of conductor $m$, $E_{L,Z,\cI}^{(r)}(m)$ is the ''average'' of the exponential sums $E_{L,Z,\sum_{i=1}^na_if_i(x)}(\psi)$ for $(a_1,\dots,a_r)\in \cO_L^r\setminus \varpi_L\cO_L^r$.

Let $r$ such that $\moi_{K,Z}^{(r)}(\cI)=0$, then Conjecture \ref{avelocIgu} holds trivially, since it is easy to verify that 
$$\left|E_{L,Z,\cI}^{(r)}(m)\right|\leq q_L^{-mn} \left(q_L^{mn}+q_L^{n-r} q_L^{(m-1)n}\right)
\leq 1+q_L^{-r}
\leq 2$$
in any case. 

When $\moi_{K,Z}^{(r)}(\cI)>0$, Lemma-Definition \ref{moidef} says that $\cI_K$ can be generated by $r$ polynomials $f_1,...,f_r\in\cO_K[x_1,...,x_n]$. Thus, Remark \ref{absexp} yields  $E_{L,Z,\cI}^{(r)}(m)= E_{L,Z\mathfrak{D}_{\cO_K}^r,\sum_{i=1}^na_if_i(x)}(\psi)$ for all additive characters $\psi$ of $L$ of conductor $m$ provided that $L\in\tilde{\cL}_{K,M}$ for a large enough integer $M$. On the other hand, by the definition,  $\moi_{K,Z}^{(r)}(\cI)=\moi_{K,Z\mathfrak{D}_{\cO_K}^r}\left(\sum_{i=1}^ny_if_i(x)\right)$.  Therefore, Conjecture \ref{avelocIgu} can be regarded as a special case of Conjecture \ref{Iguconj} with respect to $f:=\sum_{i=1}^na_if_i(x)$ and $\tilde{Z}:=Z\mathfrak{D}_{\cO_K}^r$.

On the other hand, it is clear that Conjecture \ref{avelocIgu} is a stronger version of Conjecture \ref{aveIgu}.
\end{remark}

Conjecture \ref{avelocIgu} is quite abstract and it is not easy to work with $\moi_{K,Z}^{(r)}(\cI)$ for its applications. Thus, it is better to consider some good lower bound $\sigma$ of $\moi_{K,Z}^{(r)}(\cI)$ and study Conjecture \ref{avelocIgu} for such $\sigma$ instead of $\moi_{K,Z}^{(r)}(\cI)$. 
Let $\cI\subset \cO_K[x_1,\dots,x_n]$ be a non-zero ideal which does not contain any non-zero constant. As in \cite{CluckerNguyen}, we can use the strong monodromy conjecture and the recent work of Musta\c{t}\v{a} and Popa on the minimal exponents in \cite{MustPopa} to predict a beautiful lower bound of $\moi_K(\cI)$. Recall that if $f_1,\dots,f_r$ are non-constant polynomials in $\CC[x_1,\dots,x_n]$, then Birch's singular locus $\BSing(f_1,\dots,f_r)$ is defined by
\[\BSing\left(f_1,\dots,f_r\right)=\{x\in \CC^n|\rank\left(\left(\frac{\partial f_i}{\partial x_j}(x)\right)_{1\leq i\leq r, 1\leq j\leq n}\right)<r\}.\]
We set $s(f_1,\dots,f_r)=\dim(\BSing(f_1,\dots,f_r))$ together with using the convention that $\dim(\emptyset)=-1$. If $f_1,\dots,f_r$ are homogeneous polynomials of degree $d$, then Birch's work in \cite{Birch} implies that 
\begin{equation}\label{B-lower}
\moi_\QQ(\cI)\geq \frac{n-s(f_1,\dots,f_r)}{r(d-1)2^{d-1}},
\end{equation}
where we use the convention that $\frac{n}{0}=+\infty$ and $\frac{0}{0}=0$. In \cite{Browning-HB}, Browning and Heath-Brown extended Birch's result to the situation for which we do not require that the polynomials $f_1,\dots,f_r$ are of the same degree. More precisely, they renumbered the homogeneous polynomials $f_1,\dots,f_r$ by $F_{ij}$ for $i\in J$ and $1\leq j\leq r_i$, where  $J$ is a non-empty finite subset of $\NN_{\geq 1}$ and $r_i>0$ for $i\in J$, such that $\deg(F_{ij})=i$ for all $i,j$. After that, they worked with Birch's singular locus of polynomials of the same degree. In particular, their work implies that 
$$\moi_K(\cI)\geq \tau_0:=\frac{1-\sum_{i\in J}t_ir_i}{t_{i_0}}+\sum_{i\in J}r_i,$$
where 
$$t_i=\sum_{\ell\in J, \ell\geq i} \frac{2^{\ell-1}(\ell-1)r_\ell}{n-s(F_{\ell 1},\dots,F_{\ell r_\ell})}$$
for each $i\in J$ and $i_0=\min\{i |i\in J\}$. Moreover, Conjecture \ref{aveIgu} holds for $\sigma=\tau_0$.  As mentioned in \cite{Browning-HB}, the estimates of Browning and Heath-Brown can be adapted to the general situation for which the homogeneity of polynomials is not assumed. In fact, the estimates of Browning and Heath-Brown may hold for any ideal $\cI$ generated by polynomials $g_{ij}$ for $i\in J$ and $1\leq j\leq r_i$, such that $\deg(g_{ij})=i$  and $F_{ij}$ is the homogeneous part of highest degree of $g_{ij}$ for all $i,j$. 

In this paper, we also work with the general situation where the homogeneity of polynomials is not necessary.  Let $J$ be a non-empty finite subset of $\NN_{\geq 1}$. For each $i\in J$, let $r_i$ be a positive integer. Let $\cI\subset\cO_K[x_1,\dots,x_n]$ be the ideal generated by non-constant polynomials $f_{ij}(x_1,\dots,x_n)$ for $i\in J$  and $1\leq j\leq r_i$, such that $\cI_K\neq (1)$ and $\deg(f_{ij})=i$ for all $i,j$. We will denote by $F_{ij}$ the homogeneous part of highest degree of $f_{ij}$. We put $r=\sum_{i\in J}r_i$ and $g=\sum_{i\in J,1\leq j\leq r_i}y_{ij}f_{ij}(x_1,\dots,x_n)$. Let us set
\begin{equation}\label{lbound}
\sigma_0\big( \left(f_{ij}\right)_{i\in J, 1\leq j\leq r_i}\big):=\min_{\ell\in J}\frac{n-s(F_{\ell 1},\dots,F_{\ell r_\ell})}{\ell}.
\end{equation}
Note that we have no further assumptions regarding our setting. If $z\in \CC^n\times (\CC^r\setminus \{0\})=U$, we will show that $\sigma_0\big( \left(f_{ij}\right)_{i\in J, 1\leq j\leq r_i}\big)$ is a lower bound of the minimal exponent $\tilde{\alpha}_{z}(g)$ of $g$ at $z$ (see Section \ref{BM}) 
as follows.
\begin{prop}\label{minimalexp}With the above setting, we have 
$$\tilde{\alpha}_{g\mid_U}=\min_{z\in (\CC^r\setminus \{0\})\times \CC^n}\tilde{\alpha}_{z}(g)\geq \sigma_0\big( \left(f_{ij}\right)_{i\in J, 1\leq j\leq r_i}\big).$$
\end{prop}
\begin{proof}
Let $(\textbf{a},x)\in U=\CC^n\times (\CC^r\setminus\{0\})$, where $\textbf{a}=(a_{ij})_{i\in J, 1\leq j\leq r_i}$ and $x=(x_1,\dots,x_n)$. Since $\tilde{\alpha}_{(\textbf{a},x)}(g)=+\infty$ if $g(\textbf{a},x)\neq 0$, we can suppose that $g(\textbf{a},x)=0$.  By \cite[Theorem E]{MustPopa} we have $$\tilde{\alpha}_{(\textbf{a},x)}(g)\geq\tilde{\alpha}_{x}(g_{\textbf{a}}),$$
where $g_\textbf{a}(x)=\sum_{i\in J, 1\leq j\leq r_i} a_{ij}f_{ij}(x).$

If there is $i\in J$ such that  $F_{i1},\dots,F_{ir_i}$ are linearly dependent then $\BSing(F_{i1},\dots,F_{ir_i})=\CC^n$ thus $s(F_{i1},\dots,F_{ir_i})=n$ and $$\sigma_0\big( \left(f_{ij}\right)_{i\in J, 1\leq j\leq r_i}\big)=\min_{\ell\in J}\frac{n-s(F_{\ell 1},\dots,F_{\ell r_\ell})}{\ell}=0.$$ Therefore, our claim is trivial since $\tilde{\alpha}_{g\mid_U}>0$.

If $F_{i1},\dots,F_{ir_i}$ are linearly independent for all $i\in J$, then we can suppose that $\deg(g_\textbf{a}(x))=\ell$ for some $\ell\in J$ and the homogeneous part of highest degree of $g_\textbf{a}(x)$ is $\sum_{j\in I}a_{\ell j}F_{\ell j}$ for a non-empty subset $I$ of $\{1,\dots,r_\ell\}$ such that $a_{\ell j}\neq 0$ for all $j\in I$. It is clear that $\Sing(\sum_{j\in I}a_{\ell j}F_{\ell j})\subset\BSing(F_{\ell 1},\dots,F_{\ell r_\ell})$. It follows from these facts, Section \ref{BM} and \cite[Proposition 2.4]{CluckerNguyen} that
$$\tilde{\alpha}_{x}(g_\textbf{a})\geq \tilde{\alpha}_{g_\textbf{a}}\geq\frac{n-\dim\big(\Sing(\sum_{j\in I}a_{\ell j}F_{\ell j})\big)}{\ell}\geq \frac{n-s\left(F_{\ell 1},\dots,F_{\ell r_\ell}\right)}{\ell}.$$
Therefore, we have 
\begin{align*}
\tilde{\alpha}_{g\mid_U}&=\min_{(\textbf{a},x)\in U, g(\textbf{a},x)=0} \tilde{\alpha}_{(\textbf{a},x)}(g)
\geq \min_{\textbf{a}\in \CC^r\setminus\{0\}}\min_{x\in\CC^n, g(\textbf{a},x)=0} \tilde{\alpha}_{x}(g_\textbf{a})\\
&\geq \min_{\ell\in J}\frac{n-s(F_{\ell 1},\dots,F_{\ell r_\ell})}{\ell}=\sigma_0\big( \left(f_{ij}\right)_{i\in J, 1\leq j\leq r_i}\big).
\end{align*}
\end{proof}
Let $Z\subset \AA^n_{\cO_K}$ be an $\cO_K$-scheme of finite type. By Lemma-Definition \ref{moidef}, $\moi_{K,Z}^{(r)}(\cI)=\moi_{K,Z\mathfrak{D}_{\cO_K}^r}(g)$. By an alternative definition of $\moi_{K,Z\mathfrak{D}_{\cO_K}^r}(g)$ in \cite[Section 2.4]{NguyenVeys}, the strong monodromy conjecture (see \cite{Igubf,DenefBour}) and the fact that $g$ has $0$ as its only critical value (since $\cI_K=\sum_{i\in J,1\leq j\leq r_i}(f_{ij})\neq (1)$ and $g=\sum_{i\in J,1\leq j\leq r_i}y_{ij}f_{ij}(x_1,\dots,x_n)$), we expect that $$\moi_{K,Z}^{(r)}(\cI)=\moi_{K,Z\mathfrak{D}_{\cO_K}^r}(g)\geq \tilde{\alpha}_{g\mid_{V}}\geq \tilde{\alpha}_{g\mid_{U}} ,$$
where $V$ is a small enough neighbourhood of $Z(\CC)\times (\CC^r\setminus \{0\})$ in $U=\CC^n\times (\CC^r\setminus \{0\}) $. On the other hand, since the strong monodromy conjecture may also hold for all $p$-adic fields,  we can predict a variant of Conjecture \ref{avelocIgu} as follows.
\begin{conj}\label{conjeff}Let $J$ be a non-empty finite subset of $\NN_{\geq 1}$. For each $i\in J$, let  $r_i$ be a positive integer. For each $1\leq j\leq r_i$, let $f_{ij}(x_1,\dots,x_n)\in \cO_K[x_1,\dots,x_n]$ be a non-constant polynomial. We set $r=\sum_{i\in J} r_i$ and $\cI=\sum_{i\in J, 1\leq j\leq r_i}(f_{ij})\subset \cO_K[x_1,\dots,x_n]$. Suppose that $\cI_K\neq (1)$ and $\deg(f_{ij})=i$ for all $i$. Let $Z\subset\AA^n_{\cO_K}$ be an $\cO_K$-scheme of finite type.  Then there is an integer $M$ depending only on $\cI$, a positive constant $c$ depending only on $\cI$ and $r$, and a positive constant $c_{L,Z,\cI,r}$ for each $L\in\tilde{\cL}_{K,M}\cup\cL_{K,1}$  such that $c_{L,Z,\cI,r}=c$ if $L\in \tilde{\cL}_{K,M}$ and
$$\left|E_{L,Z,\cI}^{(r)}(m)\right|\leq c_{L,Z,\cI,r}m^{n+r-1}q_L^{-m\sigma_0\big( \left(f_{ij}\right)_{i\in J, 1\leq j\leq r_i}\big)}$$
for all $L\in\tilde{\cL}_{K,M}\cup\cL_{K,1}$ and all $m\geq 2.$
\end{conj}
Note that the statement of Conjecture \ref{conjeff} also holds for all $m\geq 1$ if $Z=\AA_{\cO_K}^n$. Indeed, we will prove the following proposition in Section \ref{exsums}.
\begin{prop}\label{m=1}With the notation and the assumption of Conjecture \ref{conjeff}, there is an integer $M_\cI$, a positive constant $c_{\cI,r}$ and a positive constant $c_{L,\cI,r}$ for each $L\in\tilde{\cL}_{K,1}$  such that $c_{L,\cI,r}=c_{\cI,r}$ if $L\in \tilde{\cL}_{K,M_\cI}$ and
$$\left|E_{L,\cI}^{(r)}(1)\right|\leq c_{L,\cI,r}q_L^{-\sigma_0\big( \left(f_{ij}\right)_{i\in J, 1\leq j\leq r_i}\big)}$$
for all $L\in\tilde{\cL}_{K,1}$.
\end{prop}
\begin{example}

It is easy to show that if $r=1$ and $f=f_{11}$ is a homogeneous polynomial of degree $d>1$  such that the projective scheme associated to $f$ is smooth then $\moi_{\QQ}^{(1)}\big((f)\big)=\sigma_0\big(\left(f_{11}\right)\big)$ (see \cite{CluckerNguyen}). 

\end{example}

\begin{remark}
Conjecture \ref{avelocIgu} does not imply Conjecture \ref{conjeff} even if one can show that $\moi_{K,Z}^{(r)}(\cI)\geq \sigma_0\big(\left(f_{ij}\right)_{i\in J, 1\leq j\leq r_i}\big)$.

Proposition \ref{m=1} implies the case $w=(1,...,1)\wedge m=1$ of \ref{wnsd}.

On the other hand, the case  $w=(1,...,1)\wedge m\geq 2$ of \ref{wnsd} follows from  the weaker version of 
Conjecture \ref{conjeff} for $Z=\AA_\ZZ^n$ with the exponent $\tilde{\sigma}_{0w}\big(\left(f_{ij}\right)_{i\in J, 1\leq j\leq r_i}\big)$ instead of 
$\sigma_{0}\big(\left(f_{ij}\right)_{i\in J, 1\leq j\leq r_i}\big)$. In Section \ref{exsums}, we will show that this weaker version is true not only for $w=(1,...,1)$ but also for every $w\in \NN_{\geq 1}^n$ to imply \ref{wnsd} for $m\geq 2$.

\end{remark}
Now we will see how Conjecture \ref{conjeff} helps to study the singular series mentioned in the introduction. Let $I\neq (1)$ be a non-zero ideal of $\cO_K$, then we can write $I=\prod_{j}\mathfrak{p}_j^{m_j}$ for maximal ideals $\mathfrak{p}_j$ of $\cO_K$ and positive integers $m_j$. For each $j$, let $\varpi_j$ be a fixed uniformizing parameter of $\cO_{K_{\mathfrak{p}_j}}$. The Chinese remainder theorem yields
$$\cO_K/I\simeq\bigoplus_{j}\cO_{K_{\mathfrak{p}_j}}/(\varpi_j^{m_j}).$$
Let $\psi:\cO_K/I\to\CC^*$ be a primitive character, i.e., $\psi$ is of order $\cN_I:=\left|\cO_K/I\right|=N_{K/\QQ}(I)$. Let $i_j:\cO_{K_{\mathfrak{p}_j}}/(\varpi_j^{m_j})\to \cO_K/I$ be the canonical injection and $\psi_j=\psi\circ i_j$, then there is an additive character $\tilde{\psi}_j$ of  $K_{\mathfrak{p}_j}$ of conductor $m_j$ such that $\psi_j\pi_{m_j}=\tilde{\psi}_j|_{\cO_{K_{\mathfrak{p}_j}}}$, where $\pi_{m_j}:\cO_{K_{\mathfrak{p}_j}}\to \cO_{K_{\mathfrak{p}_j}}/(\varpi_j^{m_j})$ is the canonical projection.
\begin{lemdefn} Let $I$ be an ideal of $\cO_K$. Let $\cI=(f_1,\dots,f_r)$ be a non-zero ideal of $\cO_K[x_1,\dots,x_n]$ such that $\cI_K\neq (1)$. If $I=(1)$, we set $E_{\cI}^{(r)}(I)=1$. If $I\neq (1)$, suppose that $I=\prod_{1\leq j\leq \ell}\mathfrak{p}_j^{m_j}$ for maximal ideals $\mathfrak{p}_j$ of $\cO_K$ and positive integers $m_j$. Let $\psi$ be a primitive character of  $\cO_K/I$. The exponential sum 
\begin{equation*}
E_{\cI}^{(r)}(I):=\cN_I^{-(n+r)}\sum_{(\overline{y}_1,...,\overline{y}_r)\in (\cO_K/I)^r, I+\sum_{i}y_i\cO_K=(1)}\sum_{x\in (\cO_K/I)^n}\psi\bigg(\sum_{1\leq i\leq r}\overline{y}_{i}f_{i}(x)\bigg)
\end{equation*}
depends only on $I,\cI,r$. We call $E_{\cI}^{(r)}(I)$ the $r^{\textnormal{th}}$-exponential sum modulo $I$ of $\cI$.

\end{lemdefn}
\begin{proof}Suppose that $I=\prod_{1\leq j\leq \ell}\mathfrak{p}_j^{m_j}$  for maximal ideals $\mathfrak{p}_j$ of $\cO_K$ and positive integers $m_j$. Let $\psi$ be a primitive character of $\cO_K/I$. For each $1\leq j\leq \ell$, let $\tilde{\psi}_j$ be an additive character of $K_{\mathfrak{p}_j}$ of conductor $m_j$ such that $\psi\circ i_j\circ \pi_{m_j}=\tilde{\psi}_j|_{\cO_{K_{\mathfrak{p}_j}}}$ as mentioned above. We set $g=\sum_{1\leq i\leq r}y_if_i(x_1,...,x_n)$. It follows from the Chinese remainder theorem and Remark \ref{absexp} that  
$$E_{\cI}^{(r)}(I)=\prod_{1\leq j\leq \ell}E_{K_{\mathfrak{p}_j},\mathfrak{D}_{\cO_K}^{n,r},g}(\tilde{\psi}_j)=\prod_{1\leq j\leq \ell}E_{K_{\mathfrak{p}_j},\cI}^{(r)}(m_j),$$
where $\mathfrak{D}_{\cO_K}^{n,r}$ was defined in Section \ref{SV}. Thus, $E_{\cI}^{(r)}(I)$ depends only on $I,r,\cI$.
\end{proof}
\begin{defn}\label{defsing}
Let $\cI$ be a non-zero ideal of $\cO_K[x_1,\dots,x_n]$ such that $\cI_K\neq (1)$. For each $r$, we consider the singular series 
$$\mathfrak{S}_{\cI}^{(r)}=\sum_{(0)\neq I\subset \cO_K}E_{\cI}^{(r)}(I)\cN_I^r$$
associated to $r$ and $\cI$.
\end{defn}
\begin{remark} 
If we can find polynomials $f_1,\dots,f_r$ such that $\cI=\sum_{i=1}^r(f_i)$ then  $\mathfrak{S}_{\cI}^{(r)}$ agrees with the singular series $\mathfrak{S}(f_1,\dots,f_r)$ associated to $f_1,\dots,f_r$ (see \cite{Birch}). In particular, if $\cI=(f)$ then $\mathfrak{S}_{\cI}^{(1)}$ is the singular series $\mathfrak{S}(f)$ in the introduction. 

If $\cI=\sum_{i\in J, 1\leq j\leq r_i}(f_{ij}), r=\sum_{i\in J}r_i$ are as in the statement of Conjecture \ref{conjeff}, then Conjecture \ref{conjeff} implies that there exists a constant $c_{\epsilon}>0$ for each $\epsilon>0$ such that
$$\left|E_{\cI}^{(r)}(I)\right|\leq c_{\epsilon} \cN_I^{-\sigma_0\left( \left(f_{ij}\right)_{i\in J, 1\leq j\leq r_i}\right)+\epsilon}$$
for all non-zero ideals $I$ of $\cO_K$. Here, the error term $\epsilon$ appears when we majorize the factor $m^{n+r-1}$ in Conjecture \ref{conjeff}. 

If Conjecture \ref{conjeff} holds  for $\cI$ then $\mathfrak{S}_{\cI}^{(r)}$ converges absolutely if $\sigma_0\left( \left(f_{ij}\right)_{i\in J, 1\leq j\leq r_i}\right)>r+1$ and in this case we have 
\begin{align*}
\mathfrak{S}_{\cI}^{(r)}&=\prod_{\mathfrak{p}\in \Specm(\cO_K)}\bigg(1+\sum_{m\geq 1}E_{K_{\mathfrak{p}},\cI}^{(r)}(m)\cN_{\mathfrak{p}}^{rm}\bigg)\\
&=\prod_{\mathfrak{p}\in \Specm(\cO_K)}\lim_{m\to +\infty}\cN_{\mathfrak{p}}^{-m(n-r)}\#\spec\left(\cO_K[x_1,\dots,x_n]/\cI\right)\left(\cO_K/\mathfrak{p}^m\right).
\end{align*}
\end{remark}

\section{Some properties of the motivic oscillation indexes}\label{property}
 In this section we will first give some important properties of the (local) motivic oscillation indexes of ideals. We will use the notation in Sections \ref{nota} and \ref{exponential sum}. If there is no further assumption, then  $K$ is a number field and $X$ is the closed subscheme of $\AA_{\cO_K}^n$ associated to an ideal $\cI\subset \cO_K[x_1,\dots,x_n]$. We start this section by recalling a version of the Lang-Weil estimate from \cite{LWeil}.
\begin{prop}[Lang-Weil estimate] \label{Lang-Weil}
Let $k=\FF_q$ be a finite field and $Y\subset\PP^n_k$ be a closed subvariety of dimension $r$.  For each positive integer $\ell$, we denote by $a_Y(\ell)$ the number of geometrically irreducible components of $Y_{\FF_{q^\ell}}$ of dimension $r$. Then there is a constant $c_Y$ depending only on $n, r$, and the number and degree of the equations defining $Y$ such that for every $\ell\geq 1$, we have
$$\left|\#Y(\FF_{q^\ell})-a_Y(\ell)q^{\ell r}\right|\leq c_Yq^{\ell(r-\frac{1}{2})}.$$
\end{prop}
 Let $\delta_K^{(r)}(\cI)\geq -\infty$ be the infimum of all real numbers $\delta$ satisfying
$$\left|E^{(r)}_{\cI}(\mathfrak{p})\right|=\cN_{\mathfrak{p}}^{-n}\left|\#X(k_{K_{\mathfrak{p}}})-\cN_{\mathfrak{p}}^{n-r}\right|\leq c_\delta \cN_{\mathfrak{p}}^{\delta}$$
for all $\mathfrak{p}\in\Specm(\cO_K)$ and some positive constant $c_\delta$ independent of $\mathfrak{p}$.
The following proposition allows us to obtain  some geometric information about $E_{\cI}^{(r)}(\mathfrak{p})$ for $\mathfrak{p}\in\Specm(\cO_K)$.

\begin{prop}\label{delta}Let $X$ be the closed subscheme of $\AA_{\cO_K}^n$ associated to an ideal $\cI$ of $\cO_K[x_1,...,x_n]$. Let $r$ be a positive integer.  Suppose that $\cI$ can be generated by $r$ elements.  If one of the following conditions holds
\begin{itemize}
\item[(\textit{i}),] $\dim(X_K)>n-r$,
\item[(\textit{ii}),] $\dim(X_K)=n-r$ and $X_\CC$ is not irreducible,
\end{itemize}
then $\delta_K^{(r)}(\cI)=\dim(X_K)-n$. 

In particular, there is $\delta<-r$ and a constant $c$ such that 
$$\left|E_{\cI}^{(r)}(\mathfrak{p})\right|\leq c\cN_\mathfrak{p}^\delta$$
for all $\mathfrak{p}\in\Specm(\cO_K)$ if and only if $X_K$ is geometrically irreducible of dimension $n-r$.
\end{prop}
\begin{proof}Suppose that $X_\CC\neq\emptyset$. By using the Lang-Weil estimate in Proposition \ref{Lang-Weil}, it is easy to verify that $\delta_K^{(r)}(\cI)\leq d=\dim(X_\CC)-n=\dim(X_K)-n.$ Let $X_1,\dots,X_\ell$ be the irreducible components of $X_\CC$ of dimension $d+n$. Let $K'$ be a finite extension of $K$ such that $X_1,\dots,X_\ell$ can be defined over $K'$. By Chebotarev's density theorem, there are infinitely many primes $\mathfrak{p}\in\Specm(\cO_K)
$ such that $\mathfrak{p}$ splits completely in $\cO_{K'}$. Let $\mathfrak{p}_0\in\Specm(\cO_K)$ such that $\mathfrak{p}_0$ splits completely in $\cO_{K'}$ and $\cN_{\mathfrak{p}_0}$ is large enough then $X_1,\dots,X_\ell$ can be defined over $\cO_{K_{\mathfrak{p}_0}}$ and $X_{i}\otimes k_{K_{\mathfrak{p}_0}}$ is geometrically irreducible for all $i$. We can use the Lang-Weil estimate in Proposition \ref{Lang-Weil} again to see that there is a positive constant $c_X$ depending only on $X$ such that
$$\left|\#X(k_{K_{\mathfrak{p}_0}})-\ell\cN_{\mathfrak{p}_0}^{d+n}\right|\leq c_{X}\cN_{\mathfrak{p}_0}^{d+n-1/2}.$$
 Thus, if either $d=-r$ and $\ell>1$  or $d>-r$, then $\delta_K^{(r)}(\cI)\geq d$. Therefore, $\delta_K^{(r)}(\cI)= d$
as desired.

In order to prove the second claim, we observe that if $X_\CC=\emptyset$ then it is easy to check that $\delta_K^{(r)}(\cI)=-r$. Thus, if $\delta_K^{(r)}(\cI)<-r$ then $X_\CC\neq\emptyset$. Moreover, if $\delta_K^{(r)}(\cI)<-r$ then neither $(i)$ nor $(ii)$ holds since otherwise $-r\leq d=\delta_K^{(r)}(\cI)<-r$. Thus, the condition $\delta_K^{(r)}(\cI)<-r$ implies that $\dim(X_K)=n-r$ and $X_\CC$ is irreducible. Conversely, if $\dim(X_K)=n-r$ and $X_\CC$ is irreducible, then we can use the Lang-Weil estimate in Proposition \ref{Lang-Weil} again to have $\delta_K^{(r)}(\cI)<-r$. 
\end{proof}

 Suppose that $\cI$ can be generated by non-zero elements $f_1,\dots,f_r$ and $\cI_K\neq (1)$. We set $g=\sum_{1\leq i\leq r}y_if_i(x)$. If $L\in\cL_{K,1}$, we denote by 
$$\Oi_L^{(r)}(\cI) :=\sigma_{L,\AA_{\cO_{K}}^n}^{r,\textnormal{naive}}(\cI)=\sigma_{L,\mathfrak{D}_{\cO_K}^{n,r}}\bigg(\sum_{1\leq i\leq r}y_if_i(x)\bigg)$$  the $L$-oscillation index of $g$ at $\mathfrak{D}_{\cO_K}^{n,r}$ (see Section \ref{IE}, Definition \ref{naiv} and Corollary \ref{naig}). We set $$\tilde{\sigma}_K^{(r)}(\cI)=\min_{L\in \cL_{K,1}}\Oi_L^{(r)}(\cI)$$
and $$\sigma_K^{(r)}(\cI)=\min_{\mathfrak{p}\in \Specm(\cO_K)}\Oi_{K_\mathfrak{p}}^{(r)}(\cI).$$
By Lemma-Definition \ref{moidef}, Proposition \ref{limitp} and the item $(vi)$ of Corollary  \ref{llct}, it is clear that 
\begin{equation}\label{low}
\moi_{K}^{(r)}(\cI)\geq \sigma_{K}^{(r)}(\cI)\geq \tilde{\sigma}_K^{(r)}(\cI)\geq \lct(\cI).
\end{equation} 

Now, we also obtain some geometric and arithmetic information about $\moi_K(\cI)$ as follows.
\begin{prop}\label{equcondi} Let $X$ be the closed subscheme of $\AA_{\cO_K}^n$ associated to an ideal $\cI$ of $\cO_K[x_1,...,x_n]$ generated by $r$ elements such that $\cI_K\neq (1)$. The following conditions are equivalent:
\begin{itemize}
\item[(\textit{i}),] $\tilde{\sigma}_K^{(r)}(\cI)>r$.
\item[(\textit{ii}),] If $L\in \cL_{K,1}$, then the Igusa local zeta function $\cZ_{L,\AA_{\cO_K}^n,\cI}(s)$ has a pole at $s_0=-r$ of multiplicity at most $1$. Moreover, if $s_1$ is a pole of $\cZ_{L,\AA_{\cO_K}^n,\cI}(s)$ such that $\Re(s_1)\neq -r$, then $\Re(s_1)<-r$.
\item[(\textit{iii}),] If $L\in \cL_{K,1}$, then $\sigma_{L,\AA_{\cO_K}^n}^{\textnormal{naive}}(\cI)=\Oi_L^{(r)}(\cI)>r$.
\item[(\textit{iv}),]$\moi_K(\cI)=\moi_K^{(r)}(\cI)=\moi_K^{\textnormal{naive}}(\cI)>r$.
\item[(\textit{v}),]$\moi_K^{(r)}(\cI)>r$.
\item[(\textit{vi}),] $X_{K}$ is a complete intersection variety in $\AA^n_K$ of dimension $n-r$ and has only rational singularities.
\end{itemize}
\end{prop}
\begin{proof}Suppose that $(i)$ holds, by the definition, $r<\tilde{\sigma}_K^{(r)}(\cI)\leq\Oi_L^{(r)}(\cI)$. By Definition \ref{naiv}, if $s_1$ is a pole of   $(1-q_L^{-(s+r)})\cZ_{L,\AA_{\cO_K}^n,\cI}(s)$, then $\Re(s_1)\leq -\Oi_L^{(r)}(\cI)< -r$. So $(ii)$ follows from $(i)$.

Suppose $(ii)$ holds. By the definition of $\Oi_L^{(r)}(\cI)$, one has $\Oi_L^{(r)}(\cI)>r$ for all $L\in \cL_{K,1}$.  If $\dim(X_{K})>n-r$, then it is known that $\lct(\cI)<r$ (see Section \ref{SJ}). By Proposition \ref{VZ}, for each integer $M$, there is a local field $L\in\cL_{K,M}$ such that $s_1=-\lct(\cI)>-r$ is a pole of $\cZ_{L,\AA_{\cO_K}^n,\cI}(s)$. We have a contradiction with $(ii)$. Thus, $\dim(X_{K})=n-r$. This together with the definitions of $\sigma_{L,\AA_{\cO_K}^n}^{\textnormal{naive}}(\cI)$ and $\Oi_L^{(r)}(\cI)$ imply $(iii)$.

Suppose $(iii)$ holds. The fact that $\moi_K^{(r)}(\cI)=\moi_K^{\textnormal{naive}}(\cI)>r$ follows from Lemma-Definition \ref{moidef} and Definition \ref{naiv}. We have $r\geq n-\dim(X_K)\geq \lct(\cI)$ as mentioned in Section \ref{SJ}. Thus, $\moi_K^{(r)}(\cI)>\lct(\cI)$. We use Corollary \ref{llct} to see that $\moi_K(\cI)=\moi_K^{(r)}(\cI)=\moi_K^{\textnormal{naive}}(\cI)>r$. 

The implication $(iv)\Rightarrow (v)$ is trivial.

Let us  prove that $(v)$ implies $(vi)$. If $\dim(X_{K})>n-r$, then $\lct(\cI)\leq n-\dim(X_{K})<r$ (see Section \ref{SJ}). But it implies that $\moi_K^{(r)}(\cI)=\lct(\cI)<r$ by using Corollary \ref{llct}. We get a contradiction. Therefore, $\dim(X_{K})=n-r$ and $X_K$ is equi-dimensional. So $X_K$ is a complete intersection in $\AA_K^n$. In order to prove that $X_K$ has only rational singularities, we firstly show that $X_K$ is a variety. It is sufficient to show that $X_\CC$ is reduced. Suppose that $X_{\CC}$ is non-reduced. Let $X'$ be an irreducible component of $X_{\CC}$ such that $X'$ is non-reduced. Let $h_1:B=B_{X_{\CC,red}}(\AA_\CC^n)\to \AA_\CC^n$ be the blowing up of $\AA_\CC^n$ along $X_{\CC,red}$.  Let $h_2: Y \to B$ be a proper morphism which is isomorphic over the complement of a closed subset of $h_1^{-1}(X_{\CC})$ such that $Y$ is smooth and $(h_2^{-1}h_1^{-1}(X_{\CC}))_{red}$ is a divisor with normal crossings. Let $E'$ be an irreducible component of $(h_1^{-1}(X'))_{red}$. Then $E'$ is an integral divisor of $B$. Let $E$ be the proper transform of $E'$ by $h_2$. Then the coefficient of $E$ in $K_{Y/\AA_\CC^n}$ is $r-1$. However, since $X'$ is non-reduced, the coefficient of $E$ in $(h_1h_2)^{-1}(X_\CC)$ is at least $2$. Thus, $\lct(\cI)\leq \frac{r}{2}<r$. We use Corollary \ref{llct} again to see that $\moi_K^{(r)}(\cI)= \lct(\cI)<r$. It is a contradiction. Thus, $X_\CC$ is reduced, so is $X_K$. 
By Corollary \ref{jetsi},  it suffices to show that if $h:Y\to \AA_K^n$ is a log resolution of $\cI_K$ with $h^{-1}(X_K)=\sum_{i\in I}N_iE_i$ and $K_{Y/\AA_K^n}=\sum_{i\in I}(\nu_i-1)E_i$ then $\nu_i>rN_i$ if $(\nu_i,N_i)\neq (r,1)$, and $E_i\cap E_j=\emptyset$ whenever $(\nu_i,N_i)=(\nu_j,N_j)=(r,1)$ and $i\neq j$. Suppose that for some $i$ we have $\nu_i<rN_i$ then $\lct(\cI)<r$ thus $\moi_L^{(r)}(\cI)= \lct(\cI)<r$ as mentioned above. We get a contradiction. Now, suppose that $\nu_i\geq rN_i$ for all $i$ and either $\nu_{i_0}=rN_{i_0}>r$ for some $i_0$ or there exist distinct elements $j_1, j_2$ of  $I$ satisfying $E_{j_1}\cap E_{j_2}\neq\emptyset$ and $(\nu_{j_1},N_{j_1})=(\nu_{j_2},N_{j_2})=(r,1)$. For each $m$, let $X_m$ be the $m^{\textnormal{th}}$ jet scheme over $\cO_K$ associated to $X$ (see Section \ref{SJ}). We can use the result in \cite{ELM} to see that $\dim(X_m\otimes K)=\dim(X_K)+\dim(X_{m-1}\otimes K)=(m+1)(n-r)$ for all $m\geq 0$. Moreover, by using \cite{ELM} again,  if we denote by $a_m$ the number of irreducible components of $X_{m}\otimes \overline{K}$ of dimension $(m+1)(n-r)$, then the existence of either $i_0$ or $(j_1,j_2)$ implies that  there are infinitely many $m$ such that $a_m\neq a_{m-1}$. 
By Corollaries \ref{naig} and \ref{coboun}, there is a rational constant $b$ and an integer $M$ such that 
 \begin{equation}\label{weakb}
\left|E_{L,\cI}^{(r)}(m)\right|\leq q_L^bm^{n-1}q_L^{-m\moi_K^{(r)}(\cI)},
\end{equation}
 for all $m\geq 2$ and all $L\in\cL_{K,M}$. Let $m>2$ such that $a_{m-2}\neq a_{m-1}$ and $m(\moi_K^{(r)}(\cI)-r)>b+1$. 
 
By \cite[Proposition 3.0.2]{AizenAvni} and enlarging $M$ if needed, we can suppose that 
$$\#X(\cO_L/\varpi_L^m\cO_L)=\# X_{m-1}(k_L) \textnormal{ and } \#X(\cO_L/\varpi_L^{m-1}\cO_L)=\# X_{m-2}(k_L)$$
for all $L\in\cL_{K,M}$. Let $K'$ be a finite extension of $K$ such that all irreducible components of $X_{m-1}\otimes \overline{K}$ and $X_{m-2}\otimes \overline{K}$ are defined over $K'$. By using \cite[Proposition 9.7.8]{Gro} and enlarging $M$ again, we can suppose that any irreducible component of $(X_{m-1}\sqcup X_{m-2})\otimes k_L$ is geometrically irreducible if $L\in\cL_{K',M}$. By Proposition \ref{Lang-Weil}, we have 
\begin{align*}
\left|E_{L,\cI}^{(r)}(m)\right|&=q_L^{-mn}\left|\#X(\cO_L/\varpi_L^m\cO_L)-q_L^{n-r}\#X(\cO_L/\varpi_L^{m-1}\cO_L)\right|\\
&=q_L^{-mn}\left|\# X_{m-1}(k_L)-q_L^{n-r}\# X_{m-2}(k_L)\right|\\
&=j\left|a_{m-2}-a_{m-1}\right|q_L^{-mr}(1+O_{X,m}(q_L^{-1/2}))\\
&>\left|a_{m-2}-a_{m-1}\right|q_L^{-m\moi_K^{(r)}(\cI)+b+1}(1+O_{X,m}(q_L^{-1/2}))
\end{align*}
if $L\in\cL_{K',M}$. This contradicts with (\ref{weakb}) if $q_L$ is large enough. Thus, $(v)$ implies $(vi)$.

To finish our proof, suppose that $X_{K}$ is  a complete intersection and has only rational singularities. We use Corollary \ref{jetsi} again to produce a log resolution $h:Y\to \AA_K^n$ with $K_{Y/\AA_K^n}=\sum_{i\in I}(\nu_i-1)E_i$ and $h^{-1}(X_K)=\sum_{i\in I}N_iE_i$ such that $\nu_i>rN_i$ if $(\nu_i,N_i)\neq (r,1)$, moreover if  $i\neq j$ and  $(\nu_i,N_i)=(\nu_j,N_j)=(r,1)$ then $E_i\cap E_j=\emptyset$. We can use the log resolution $h$ and Proposition \ref{VZ} to see that if $L\in\cL_{K,1}$ and $s_0$ is a pole of $(1-q_L^{-(s+r)})\cZ_{L,\AA_{\cO_K}^n,\cI}(s)$ then $\Re(s_0)<-r$. By Definition  \ref{naiv}, we have  $\Oi_{L}^{(r)}(\cI)=\sigma_{L, \AA_{\cO_K}^n}^{r,\textnormal{naive}}(\cI)>r$ for all $L\in\cL_{K,1}$. Thus, $\tilde{\sigma}_K^{(r)}(\cI)>r$.

\end{proof}
Similarly, we have an arithmetic characterization of the property of a scheme being a complete intersection having only (semi) log canonical singularities as follows.
\begin{prop}\label{equcondilog} Let $X$ be the closed subscheme of $\AA_{\cO_K}^n$ associated to an ideal $\cI$ of $\cO_K[x_1,...,x_n]$ generated by $r$ elements such that $\cI_K\neq (1)$. If $X_K$ is normal then the following conditions are equivalent:
\begin{itemize}
\item[(\textit{i}),] $\tilde{\sigma}_K^{(r)}(\cI)\geq r$.
\item[(\textit{ii}),] If $L\in \cL_{K,1}$ and $s_0$ is a pole of the Igusa local zeta function $\cZ_{L,\AA_{\cO_K}^n,\cI}(s)$, then $\Re(s_0)\leq -r$.
\item[(\textit{iii}),]If $L\in \cL_{K,1}$, then $\sigma_{L,\AA_{\cO_K}^n}^{\textnormal{naive}}(\cI)=\Oi_L^{(r)}(\cI)\geq r$.
\item[(\textit{iv}),] $\moi_K(\cI)=\moi_K^{(r)}(\cI)=\moi_K^{\textnormal{naive}}\geq r$.
\item[(\textit{v}),]$\moi_K^{(r)}(\cI)\geq r$.
\item[(\textit{vi}),] $X_{K}$ is a complete intersection variety in $\AA^n_K$ of dimension $n-r$ and has only log canonical singularities.
\end{itemize}
On the other hand, if $X_K$ is not assumed to be normal then the above assertion still holds when we replace log canonical singularities by semi-log canonical singularities in the item $(vi)$.
\end{prop}
\begin{proof}
We can adapt the proof of Proposition \ref{equcondi}. The main points are as follows:
\begin{itemize}
\item If $h:Y\to \AA_K^n$ is a log resolution of $\cI_K$ with $h^{-1}(X_K)=\sum_{i\in I}N_iE_i$ and $K_{Y/\AA_K^n}=\sum_{i\in I}(\nu_i-1)E_i$, then $\lct(\cI)=\min_{i\in I} \nu_i/N_i$. Thus, if $\lct(\cI)<r$, then Corollary \ref{llct} and (\ref{low}) implies that $\tilde{\sigma}_K^{(r)}(\cI)=\moi_K^{(r)}(\cI)=\lct(\cI)<r$. Moreover, for each integer $M$, there is a local field $L\in\cL_{K,M}$ such that $s_1=-\lct(\cI)>-r$ is a pole of $\cZ_{L,\AA_{\cO_K}^n,\cI}(s)$ (see Proposition \ref{VZ}).
\item If $\lct(\cI)=r$ and $\cI_K$ can be generated by $r$ elements, then the proof of Proposition \ref{equcondi} deduces that $X_K$ is geometrically reduced and a complete intersection of dimension $n-r$. Let $X_m$ be the $m^{th}$ jet scheme of $X$ over $\cO_K$ (see Section \ref{SJ}). We can use \cite{ELM} to see that every irreducible component of $X_{m}\otimes K$ is of dimension at most $(m+1)\dim(X_K)$. 
By Proposition \ref{singjet},  this holds if only if $X_K$ has only semi-log canonical singularities. On the other hand, by  Proposition \ref{singjet}, when $X_K$ is normal then $X_K$ has only log canonical singularities if and only if $X_K$ has only semi-log canonical singularities.   
\end{itemize}
\end{proof}

\begin{cor}\label{co1}Let $X$ be the closed subscheme of $\AA_{\cO_K}^n$ associated to an ideal $\cI$ of $\cO_K[x_1,...,x_n]$ generated by $r$ elements such that $\cI_K\neq (1)$. The following conditions are equivalent:
\begin{itemize}
\item[(\textit{a}),] $\moi_K(\cI)>n-\dim(X_K)$.
\item[(\textit{b}),]$X_K$ is a complete intersection in $\AA^n_K$ and has only rational singularities.
\end{itemize}
Moreover, the following conditions are equivalent:
\begin{itemize}
\item[(\textit{c}),] $\moi_K(\cI)\geq n-\dim(X_K)$.
\item[(\textit{d}),]$X_K$ is a complete intersection in $\AA^n_K$ and has only semi-log canonical singularities.
\end{itemize}
\end{cor}
\begin{proof}  

Suppose that $(a)$ holds. Let $r'\in\NN_{\geq 1}$ such that $\moi_K(\cI)=\moi_K^{(r')}(\cI)$. Then it is clear that the ideal $\cI_K$ can be generated by $r'$ elements $g_1,...,g_{r'}\in\cO_K[x_1,...,x_n]$. It follows from \cite[Chapter 1, Section 7]{Harts} and $\cI_K\neq (1)$ that $\dim(X_K)\geq n-r'$. By Lemma-Definition \ref{moidef},  replacing $\cI$ by $(g_1,...,g_{r'})$ does not harm our proof, thus we can suppose that $\cI=(g_1,...,g_{r'})$.  If $\dim(X_K)>n-r'$, then we have $\lct(\cI)\leq n-\dim(X_K)<r'$ (see Section \ref{SJ}). Thus, Corollary \ref{llct} implies that  $\moi_K^{(r')}(\cI)=\lct(\cI)\leq n-\dim(X_K)$. We get a contradiction. If $\dim(X_K)=n-r'$, then we have $r'<\moi_K(\cI)=\moi_K^{(r')}(\cI)$. Thus, we can use Proposition \ref{equcondi} to obtain $(b)$.

Suppose that $(b)$ holds. Let $r'=n-\dim(X_K)$. Then $\cI_K$ can be generated by elements $g_1,...,g_{r'}\in\cO_K[x_1,...,x_n]$. We can use  Proposition \ref{equcondi} for $\cI'=(g_1,...,g_{r'})$ and Lemma-Definition \ref{moidef} to conclude that $\moi_K(\cI)=\moi_K(\cI')=\moi_K^{(r')}(\cI')>r'$. We get $(a)$.

By a similar argument together with using Proposition \ref{equcondilog}, we have $(c) \Leftrightarrow (d)$.
\end{proof}
\begin{cor}\label{co2}Let $X$ be the closed subscheme of $\AA_{\cO_K}^n$ associated to an ideal $\cI$ of $\cO_K[x_1,...,x_n]$ generated by $r$ elements such that $\cI_K\neq (1)$. Suppose that $\moi_K^{(r)}(\cI)>r$  and $X_{\overline{K}}$ has one irreducible component which can be defined over $K$, then there exists an integer $M$ such that $\cZ_{L,\AA_{\cO_K}^n,\cI}(s)$ has a pole at $s_0=-r$ of multiplicity $1$ for all $L\in\cL_{K,M}$.
\end{cor}
\begin{proof}Since $\moi_K^{(r)}(\cI)>r$, it follows from Proposition \ref{equcondi} that $X_K$ is a complete intersection in $\AA_K^n$ of dimension $n-r$ and has only rational singularities. In particular, $X_K$ is equi-dimensional. Moreover, Corollary \ref{jetsi} helps us to produce a log resolution $h:Y\to \AA_K^n$ of $\cI_K$ associated to irreducible divisors $(E_i)_{i\in I}$ and numerical data $((\nu_i,N_i))_{i\in I}$ with the following properties:
\begin{itemize}
\item[(\textit{i}),]$\nu_i\geq rN_i$ for all $i\in I$.
\item[(\textit{ii}),]If $\nu_i=rN_i$, then $(\nu_i,N_i)=(r,1)$.
\item[(\textit{iii}),]If $i\neq j$ and $(\nu_i,N_i)=(\nu_j,N_j)=(r,1)$, then $E_i\cap E_j=\emptyset$.
\end{itemize} 
Since $X_K$ has a geometrically irreducible component, there is $i\in I$ such that $E_i$ is geometrically irreducible and $(\nu_i,N_i)=(r,1)$.  
These facts and Proposition \ref{VZ} imply that $s_0=-r$ is a pole of multiplicity $1$ of $\cZ_{L,\AA_{\cO_K}^n,\cI}(s)$ if $L$ is of large enough residue field characteristic.
\end{proof}

\begin{cor}\label{root}Let $\cI$ be an ideal of $\cO_K[x_1,...,x_n]$ generated by $r$ elements such that $\cI_K\neq (1)$. Recall the Bernstein-Sato polynomial $b_{\cI_K}(s)$  of $\cI_K$ in Section \ref{BM}.  Suppose that $\moi_K^{(r)}(\cI)>r$, then $-r$ is the largest root of $b_{\cI_K}(s)$ and the multiplicity of $b_{\cI_K}(s)$ at $-r$ is $1$.
\end{cor}
\begin{proof}
The assertion follows  immediately from Proposition  \ref{equcondi} and \cite[Theorem 1.1, Corollary 1.3]{Mustideal}.
\end{proof}
\begin{proof}[Proof of \ref{pole1} and Corollary \ref{corpol}]Note that $\dim(X_K)\geq n-r$. If $\moi_K(\cI)>r\geq n-\dim(X_K)$, then Corollary \ref{co1} and Proposition \ref{equcondi} imply that $\moi_K(\cI)=\moi_K^{(r)}(\cI)$. Thus, Corollary \ref{corpol} and many parts of \ref{pole1} were proved in Proposition \ref{equcondi}. It remains to consider the case that $\moi_K(\cI)\leq r$ in \ref{pole1}.

We show that $\moi_K(\cI)\leq r$ if and only if $\moi_K^{(r)}(\cI)=\lct(\cI)$. Indeed, suppose $\moi_K^{(r)}(\cI)=\lct(\cI)$ and $\moi_K(\cI)>r$. As mentioned in Section \ref{SJ}, $\lct(\cI)\leq n-\dim(X_K)\leq r$, thus $\moi_K^{(r)}(\cI)\leq r$. However, since $\moi_K(\cI)>r$ we have $\moi_K(\cI)=\moi_K^{(r)}(\cI)>r$ as shown above. It is a contradiction. Thus, $\moi_K^{(r)}(\cI)=\lct(\cI)$ implies $\moi_K(\cI)\leq r$. Conversely, suppose that $\moi_K(\cI)\leq r$ then Corollary \ref{llct} implies $\lct(\cI)\leq \moi_K^{(r)}(\cI)\leq \moi_K(\cI)\leq r$. If $\moi_K^{(r)}(\cI)<\moi_K(\cI)$, then Corollary \ref{llct} implies that  $\lct(\cI)=\moi_K^{(r)}(\cI)$. By Corollary \ref{llct}, if $r\geq\moi_K^{(r)}(\cI)=\moi_K(\cI)>\lct(\cI)$, then $\moi_K^{(r)}(\cI)=\lct(\cI)$, thus we have a contradiction. Therefore, the condition $r\geq\moi_K^{(r)}(\cI)=\moi_K(\cI)$ implies $\lct(\cI)=\moi_K^{(r)}(\cI)$. In any case, we have $\lct(\cI)=\moi_K^{(r)}(\cI)$.


If $\moi_K(\cI)=r$, it follows from Lemma-Definition \ref{moidef} that $\moi_K^{(r)}(\cI)=\lct(\cI)\leq r$. It is sufficient to show that if $r=\moi_K(\cI)>\moi_K^{(r)}(\cI)$, then $X_\QQ$ is a complete intersection having only rational singularities. By Corollary \ref{llct}, we have $\moi_K(\cI)=\moi_K^{(\ell)}(\cI)$ for some $\ell<r$. This and Proposition \ref{compa} imply $\moi_K(\cI)=r>\ell\geq n-\dim(X_K)$. Therefore, our claim follows from Corollary \ref{co1}.
\end{proof}
We can state a local version of Propositions \ref{equcondi}, \ref{equcondilog} and \ref{pole1} as follows.
\begin{named}{Theorem F}\label{locpole}Let $K$ be a number field and $X$ be the subscheme of $\AA_{\cO_K}^n$ associated to a non-zero ideal $\cI\subset \cO_K[x_1,\dots,x_n]$ such that $\cI_K\neq (1)$. Let $Z\subset X$ be an $\cO_K$-scheme of finite type such that $Z_K\neq\emptyset$. Suppose that $X_K$ is equi-dimensional 
then $\moi_{K,Z}^{\textnormal{loc}}(\cI)>n-\dim(X_K)$ (resp. $\moi_{K,Z}^{\textnormal{loc}}(\cI)\geq n-\dim(X_K)$) if and only if there is an affine open subset $U$ of $X_K$ containing $Z_K$ such that $U$ is a complete intersection in an affine open subset of $\AA_K^n$ and has only rational singularities (resp. semi-log canonical singularities). Moreover, the following conditions are equivalent:
\begin{itemize}
\item[\textit{(i)},]There is an integer $M$ such that if $L\in\cL_{K,M}$ and $s_0$ is a pole of $\cZ_{L,Z,\cI}(s)$, then either $\Re(s_0)<\dim(X_K)-n$ or $\Re(s_0)=\dim(X_K)-n$ and $\cZ_{L,Z,\cI}(s)$ has a pole of multiplicity at most $1$ at $s_0$.
\item[\textit{(ii)},]$X_{\overline{K}}$ is a locally complete intersection having only rational singularities at any point $P\in Z(\overline{K})$.
\item[\textit{(iii)},]For all finite extensions $K'$ of $K$ and all points $P\in Z(K')$ we have $\moi_{K',P}^{\textnormal{loc}}(\cI)>n-\dim(X_K)$.
\end{itemize} 
\end{named}
\begin{proof}
The proof is similar to the proof of Propositions \ref{equcondi}, \ref{equcondilog} and Corollary \ref{co1}, but instead of working with  $X_K$ and generators of $\cI_K$, we will proceed with a suitable open affine neigbourhood $U\subset \AA_K^n$ of $Z_K$ and polynomials $f_1,\dots,f_r\in\cO_K[x_1,\dots,x_n]$ which generate $\cI_K\cO_U$ as in Definition \ref{moiloc1}.
\end{proof}
\ref{locpole} implies the following consequences.
\begin{cor}\label{rati1}Let $K$ be a number field and $X$ be the subscheme of $\AA_{\cO_K}^n$ associated to a non-zero ideal $\cI\subset \cO_K[x_1,\dots,x_n]$ such that $\cI_K\neq (1)$. Suppose that $X_K$ is equi-dimensional. Then $X_K$ is  a locally complete intersection having only rational singularities if and only if $\moi_{K',P}^{\textnormal{loc}}(\cI)>n-\dim(X_K)$ for all finite extensions $K'$ of $K$ and all points $P\in X(K')$.
\end{cor}
\begin{proof}
The assertion follows immediately from \ref{locpole}.
\end{proof}
\begin{cor}\label{rati2}Let $K$ be a number field and $X$ be a $K$-scheme of finite type. Suppose that $X$ is equi-dimensional  then $X$ is a locally complete intersection having only rational singularities  (resp. semi-log canonical singularities) if and only if $\moi_{K',P}^{\textnormal{aloc}}(X)>-\dim(X)$ (resp. $\moi_{K',P}^{\textnormal{aloc}}(X_{K'})\geq -\dim(X)$) for all finite extensions $K'$ of $K$ and all points $P\in X(K')$, where we used Remark \ref{variety} to define $\moi_{K',P}^{\textnormal{aloc}}(X)$.
\end{cor}
\begin{proof}Suppose that $\moi_{K',P}^{\textnormal{aloc}}(X_{K'})>-\dim(X)$ for all finite extensions $K'$ of $K$ and all points $P\in X(K')$. Let $P\in X(K')$, we can use Definition \ref{moiloc2}, Remark \ref{variety} and Corollary \ref{naig} to obtain a suitable closed embedding of a small enough affine neighbourhood of $P$ in $X_{K'}$ to an affine space of finite dimension over $K'$ such that this closed embedding computes $\moi_{K',P}^{\textnormal{aloc}}(X_{K'})$. Then  \ref{locpole} implies that $X_{K'}$ is  a locally  complete intersection having only rational singularities in a neighbourhood of $P'$ in $X_{K'}$. This fact holds for all $K'$ and all $P\in X(K')$. Thus, $X_{\overline{K}}$ is a locally complete intersection having only rational singularities, so is $X$ (see Sections \ref{SV} and \ref{SJ}).

If $X$ is  a locally  complete intersection having only rational singularities, so is $X_{K'}$ for all finite extensions $K'$ of $K$. Let $P\in X(K')$, we can choose a suitable affine neighbourhood $U$ of $P$ in  $X_{K'}$ and a suitable closed embedding $\varphi$ from $U$ to an affine space of finite dimension over $K'$ such that $\varphi(U)$ is a complete intersection. Then $\varphi(U)$ is a complete intersection having only rational singularities. We can use  Definition \ref{moiloc2}, Remark \ref{variety} and \ref{locpole} to deduce that $\moi_{K',P}^{\textnormal{aloc}}(X_{K'})>-\dim(X)$. 

The claim for semi-log canonical singularities is proved similarly.
\end{proof}

\section{$FRS/FTS$ morphisms, uniform counting points of schemes over finite rings and some Waring type problems} \label{uniadd}
We will use the notation of Sections \ref{exponential sum} and \ref{property}. We also need other definitions. In this section, let $K$ be a number field as usual.
\begin{defn}\label{acomD}Suppose that $X$ is an affine $\cO_K$-scheme of finite type and $N,R,D\in \ZZ_{\geq 0}$.
\begin{itemize}
\item[(\textit{i}),]\index{complexity}One says that $X$ is of complexity at most $(N,R,D)$ if we have a closed embedding $X_K\to\AA_{K}^n$ such that $$X_K=\spec\big(K[x_1,\dots,x_n]/(f_1,\dots,f_r)\big),$$ where $n\leq N, r\leq R$ and $\max_{1\leq i\leq r}(\deg(f_i))\leq D$.
\item[(\textit{ii}),]\index{degree complexity}Weakly, one says that $X$ is of degree complexity at most $(R,D)$ if we only require that $\max_{1\leq i\leq r}(\deg(f_i))\leq D$ and $r\leq R$ for a closed embedding $X_K\to \AA_{K}^n$ as above.
 \end{itemize}
\end{defn}
\begin{defn}\label{comD}Let $X$ be a scheme of finite type over $\cO_K$ and $N,R,D\in\ZZ_{\geq 0}$. One says that $X$ is of complexity at most $(N,R,D)$ if we can find an open cover $X=\cup_{1\leq i\leq \ell} U_i$ of $X$ where $\ell\leq N$ and $U_1,\dots,U_\ell$ are open affine subschemes of $X$ of complexity at most $(N,R,D)$. Similarly, one can say about the degree complexity $(R,D)$ of $X$ by forgetting $N$.
\end{defn}
\begin{defn}\label{mcomD}
Let $X$ and $Y$ be schemes of finite type over $\cO_K$ of complexity at most $(N,R,D)$. Let $\varphi:X\to Y$ be an $\cO_K$-morphism.  One says that $\varphi$ is of  complexity at most $(N,R,D)$ if we can find open covers $Y=\cup_{1\leq i\leq \ell} U_i$ and $X=\cup_{1\leq i\leq \ell}\cup_{1\leq j\leq \ell_i}V_{ij}$ where $\ell,\ell_i\leq N$ and $U_i, V_{ij}$ are $\cO_K$-affine schemes of complexity at most $(N,R,D)$ such that the following conditions hold:
\begin{itemize}
\item[(\textit{i}),]$\varphi(V_{ij})\subset U_i$ for all $i,j$.
\item[(\textit{ii}),]For each $1\leq i\leq \ell$ and each $1\leq j\leq \ell_i$, there are closed embeddings $\gamma_i:U_i\otimes K\to \AA_K^{n_i}, \gamma_{ij}:V_{ij}\otimes K\to \AA_K^{n_{ij}}$ as in Definition \ref{acomD} and a morphism $\varphi_{ij}=(g_1,\dots,g_{n_{i}}): \AA_K^{n_{ij}}\to \AA_K^{n_{i}}$ such that $\gamma_i^{-1}\varphi_{ij}\gamma_{ij}=\varphi_K|_{V_{ij}\otimes K}$, $n_i, n_{ij}\leq N$ and $\max_{1\leq m\leq n_{i}}\deg(g_{m})\leq D$.
\end{itemize}

\end{defn}
\begin{defn}\label{dcomD}Let $X$ and $Y$ be schemes of finite type over $\cO_K$ of degree complexity at most $(R,D)$. Let $\varphi:X\to Y$ be an $\cO_K$-morphism.  One says that $\varphi$ is of  degree complexity at most $(R,D)$ if in Definition \ref{mcomD}, we only require that $\deg(g_{m})\leq D$ for all $m$ and $U_i, V_{ij}$ are of degree complexity at most $(R,D)$ for all $i,j$.
\end{defn}
\begin{remark}
If $R=N=D$, Definition \ref{mcomD} is weaker than \cite[Definition 7.7]{G-H1}. In particular, we do not require any conditions on the transition morphisms of affine covers.
\end{remark}
\begin{defn}\label{defrs}
Let $\mathsf{F}$ be a field of characteristic zero. Let $X$ and $Y$ be $\mathsf{F}$-schemes of finite type. Let $\varphi:X\to Y$ be an $\mathsf{F}$-morphism then $\varphi$ is \index{\textit{FRS} morphism}$FRS$ if $\varphi$ is flat and every non-empty fiber of $\varphi$ has only rational singularities. One says that $\varphi$ is \index{\textit{FTS} morphism}$FTS$ if it is flat and its non-empty fibers are normal and have only terminal singularities. One also says that $\varphi$ is \index{\textit{FGI} morphism}$FGI$ if it is flat with geometrically irreducible fibers. 
\end{defn}
We also need a slight extension of the notion of $E$-smooth morphisms in \cite{C-G-H} to the situation of morphisms between varieties having only semi-log canonical singularities.
\begin{defn}\label{esm}Let $\mathsf{F}$ be a field of  characteristic zero and $X,Y$ be $\mathsf{F}$-varieties. Let $\varphi: X\to Y$ be an $\mathsf{F}$-morphism. Suppose that $X$ and $Y$ are equi-dimensional and locally complete intersections  having only semi-log canonical singularities. For each $0\leq m$, we denote by $X_m, Y_m$ the $m^{\textnormal{th}}$ jet schemes associated to $X, Y$ over $\mathsf{F}$ as in Section \ref{SJ}.  We also denote by $\varphi_m: X_m\to Y_m$ the morphism induced by $\varphi$. Let $E$ be a positive integer. The morphism \index{$E$-smooth morphism}$\varphi$ is $E$-smooth if the following conditions hold:
\begin{itemize}
\item[(\textit{D}),] For all $m\geq 0$, every non-empty fiber of $\varphi_m$ is of dimension $\dim(X_m)-\dim(Y_m)$.
\item[(\textit{S1}),] For all $m\geq 0$, every non-empty fiber of $\varphi_m$ has singular locus of codimension at least $E$.
 \end{itemize}
\end{defn}
In this paper, we also need a variant version of Definition \ref{esm}.
\begin{defn}\label{sesm}Let $\mathsf{F}$ be a field of  characteristic zero and $X,Y$ be $\mathsf{F}$-varieties. Let $\varphi: X\to Y$ be an $\mathsf{F}$-morphism. Suppose that $X$ and $Y$ are equi-dimensional and locally complete intersections  having only semi-log canonical singularities. Let $E$ be a positive integer. For each $m\geq 0$, let $X_m,Y_m,\varphi_m$ be as in Definition \ref{esm}. Let $y\in Y_m$, we write $X_{m,\varphi,y}$ (or simple $X_{m,y}$) for the fiber of $\varphi_m$ over $y$ (if $\varphi$ is understood).  We also denote by $\pi_{X}^{ij}:X_i\to X_j$ the truncation morphism if $i\geq j$ (see Section \ref{SJ}).  The morphism $\varphi$ is \index{strongly $E$-smooth morphism}strongly $E$-smooth if the following conditions hold:
\begin{itemize}
\item[(\textit{D}),]For all $m\geq 0$, every non-empty fiber of $\varphi_m$ is of dimension $\dim(X_m)-\dim(Y_m)$.  
\item[(\textit{S2}),]For all $m\geq 0$ and all $y\in Y_m$, $X_{m,y}\cap \left(\pi_{X}^{m0}\right)^{-1}\left(\Sing(\varphi)\right)$ is of codimension at least $E$ in $X_{m,y}$, where $\Sing(\varphi)$ is the singular locus of $\varphi$.
 \end{itemize}
\end{defn}
The following proposition provides more information about strongly $E$-smooth morphisms.
\begin{prop}\label{strong}With the notation of Definition \ref{sesm}, the following conditions are equivalent:
\begin{itemize}
\item[(\textit{i}),]$\varphi$ is strongly $E$-smooth.
\item[(\textit{ii}),]If $x$ is a geometric point of $\Sing(\varphi)$ and $m\geq 0$, then  $$\dim\big(X_{m,s_{Y,m}(\varphi(x))}\big)= \dim(X_m)-\dim(Y_m)$$ and  $X_{m,s_{Y,m}(\varphi(x))}\cap \left(\pi_{X}^{m0}\right)^{-1}\left(\Sing(\varphi)\right)$ is of dimension at most $$\dim(X_m)-\dim(Y_m)-E,$$ where $s_{Y,m}:Y\to Y_m$ is the zero section of $\pi_Y^{m0}$ mentioned in Section \ref{SJ}.
\end{itemize}
\end{prop}
\begin{proof}
The implication $(\textit{i})\Rightarrow (\textit{ii})$ is trivial from Definition \ref{sesm}.

Now, we suppose that condition $(ii)$ holds and prove that $(i)$ holds. Since $X$ is equi-dimensional and a locally complete intersection having 
only semi-log canonical singularities, we deduce from Proposition \ref{singjet} that $\dim(X_m)=(m+1)\dim(X)$ and $X_m$ is equi-dimensional for all $m\geq 0$. Therefore, \cite[Lemma 13.1.1]{EGA4} implies that any non-empty fiber of $\varphi_m$ is of dimension at least $\dim(X_m)-\dim(Y_m)$. Thus, it is sufficient to show that for each geometric point $P$ of $X$ and each geometric point $Q\in Y_m$ with $\pi_Y^{m0}(Q)=\varphi(P)$, we have 
\begin{equation}\label{ieqdimen}
\dim\big(X_{m,s_{Y,m}\left(\varphi(P)\right)}\cap \left(\pi_{X}^{m0}\right)^{-1}\left(\Sing(\varphi)\right)\big)\geq \dim\big(X_{m,Q}\cap \left(\pi_{X}^{m0}\right)^{-1}\left(\Sing\left(\varphi\right)\right)\big).
\end{equation}
and
\begin{equation}\label{eqdimen}
\dim\big(X_{m,s_{Y,m}(\varphi(P))}\big)\geq \dim\left(X_{m,Q}\right).
\end{equation}

We will prove (\ref{ieqdimen}). The proof of (\ref{eqdimen}) proceeds similarly. In order to prove (\ref{ieqdimen}), we only need to verify that 
\begin{equation}\label{uppdi}
\dim\big(X_{m,s_{Y,m}\left(\varphi(P)\right)}\cap \left(\pi_{X}^{m0}\right)^{-1}(P)\big)\geq \dim\big(X_{m,Q}\cap \left(\pi_{X}^{m0}\right)^{-1}(P)\big)
\end{equation}
for each geometric point $P$ of $\Sing(\varphi)$ and each geometric point $Q$ of $Y_m$ such that $\pi_Y^{m0}(Q)=\varphi(P)$. The statement is local, thus we can suppose that $$P=O_n:=(0,\dots,0)\in X=\spec(\mathsf{F}[x_1,\dots,x_n]/(f_1,\dots,f_r)),$$ 
$$\varphi(P)=O_e:=(0,\dots,0)\in Y=\spec(\mathsf{F}[y_1,\dots,y_e]/(g_1,\dots,g_\ell)),$$
and $\varphi=(h_1,\dots,h_e)$ for polynomials $f_1,\dots,f_r$, $h_1,\dots,h_e\in \mathsf{F}[x_1,\dots,x_n]$ and polynomials  $g_1,\dots,g_\ell$ in $\mathsf{F}[y_1,\dots,y_e]$. 

We use the inclusion $X\subset \tilde{X}=\AA_\mathsf{F}^n$ to regard $X_m^{(O_n)}:=X_m\cap \left(\pi_{X}^{m0}\right)^{-1}(O_n)$ as a closed subscheme of $\tilde{X}_m\cap \big(\pi_{\tilde{X}}^{m0}\big)^{-1}(O_n)=\AA_\mathsf{F}^{nm}$ associated  to an ideal $\cJ_m$ of $\mathsf{F}[x_{ij}]_{1\leq i\leq m,1\leq j\leq n}$. Similarly, we can view $Y_{m}^{(O_e)}:=Y_m\cap \left(\pi_{Y}^{m0}\right)^{-1}(O_e)$ as a closed subscheme of $\AA_\mathsf{F}^{em}=\tilde{Y}_m\cap \big(\pi_{\tilde{Y}}^{m0}\big)^{-1}(O_e)$, where $\tilde{Y}=\AA_\mathsf{F}^{e}$. Let $\tilde{\varphi}:\tilde{X}=\AA_\mathsf{F}^n\to \tilde{Y}=\AA_\mathsf{F}^{e}$ be the morphism given by  $x\mapsto \left(h_1(x),...,h_e(x)\right)$. Then $\varphi_m$ sends $X_m^{(O_n)}$ to $Y_{m}^{(O_e)}$ and $\tilde{\varphi}_m$ sends $\AA_\mathsf{F}^{nm}=\tilde{X}_m\cap \big(\pi_{\tilde{X}}^{m0}\big)^{-1}(O_n)$ to $\tilde{Y}_m\cap \big(\pi_{\tilde{Y}}^{m0}\big)^{-1}(O_e)=\AA_\mathsf{F}^{em}$. Moreover, the restriction of $\varphi_m$ to $X_m^{(O_n)}$ equals to the restriction of $\tilde{\varphi}_m$ to $X_m^{(O_n)}$. Let us write $\tilde{\varphi}_m=\left(G_{ij}\right)_{1\leq i\leq e, 1\leq j\leq m}$ with polynomials $G_{ij}\in\mathsf{F}[x_{ij}]_{1\leq i\leq m,1\leq j\leq n}$. By a simple calculation, we can show that the following conditions hold:
\begin{itemize}
\item[(\textit{a}),]$G_{ij}$ is $w$-weighted homogeneous of $w$-degree $d_{ij}>0$ for all $1\leq i\leq e$ and $1\leq j\leq m$, where $w=\left(w_{ij}\right)_{1\leq i\leq m,1\leq j\leq n}$ with $w_{ij}=i$ is the weight for the variable $x_{ij}$.
\item[(\textit{b}),]$\cJ_m$ can be generated by non-constant $w$-weighted homogeneous polynomials $H_{\ell}$ for $1\leq \ell\leq r$.
\item[(\textit{c}),] For each geometric point $Q=(y_{ij})_{1\leq i\leq e, 1\leq j\leq m}$ of $Y_m^{(O_e)}$, $X_{m,Q}\cap \left(\pi_{X}^{m0}\right)^{-1}(O_n)=X_m^{(O_n)}\cap\{ G_{ij}-y_{ij}=0 \hspace{0.1cm}\forall 1\leq i\leq e, 1\leq j\leq m\}$.
\item[(\textit{d}),]$s_{Y,m}(O_e)=(0,...,0)\in Y_m^{(O_e)}\subset\AA_\mathsf{F}^{em}.$
\end{itemize}
Let $\PP_\mathsf{F}(w)$ be the weighted homogeneous space over $\mathsf{F}$ associated to $w=\left(w_{ij}\right)_{1\leq i\leq m,1\leq j\leq n}$ and $V$ be the closed scheme of $\PP_{\mathsf{F}}(w)\times \AA_\mathsf{F}^{em}$ associated to the equations $H_\ell=0$ for $1\leq \ell\leq r$ and $G_{ij}-y_{ij}=0$ for $1\leq i\leq e, 1\leq j\leq m$.
Since the projection $\rho:V\subset\PP_{\mathsf{F}}(w)\times \AA_\mathsf{F}^{em}\to \AA_\mathsf{F}^{em}$ is projective and  $\rho^{-1}\big(\left(y_{ij}\right)_{1\leq i\leq e, 1\leq j\leq m}\big)\to \rho^{-1}\big(\left(c^{d_{ij}}y_{ij}\right)_{1\leq i\leq e, 1\leq j\leq m}\big), \left(x_{ij}\right)_{1\leq i\leq n,1\leq j\leq m}\mapsto \left(c^{w_{ij}}x_{ij}\right)_{1\leq i\leq n,1\leq j\leq m}$ is an $\overline{\mathsf{F}}$-isomorphism for all $\left(y_{ij}\right)_{1\leq i\leq e, 1\leq j\leq m}\in \AA^{em}(\overline{\mathsf{F}})$ and all $c\in\overline{\mathsf{F}}^{\times}$, we can use Chevalley's upper semi-continuous theorem \cite[Theorem 13.1.3]{EGA4} to obtain (\ref{uppdi}).
\end{proof}
We have a relation between Definition \ref{esm} and Definition \ref{sesm} as follows.
\begin{prop}\label{se}Let $X$ and $Y$ be as in Definition \ref{sesm}. Let $\varphi:X\to Y$ be a morphism and $E$ be a positive integer. If  $\varphi$ is strongly $E$-smooth,  then $\varphi$ is $E$-smooth. Conversely, if $Y$ is smooth and $\varphi$ is $E$-smooth, then $\varphi$ is strongly $E$-smooth.
\end{prop}
\begin{proof}
Suppose that $\varphi$ is strongly $E$-smooth. Let $m\geq 0$. Let $y\in Y_m$ such that $X_{m,y}\neq \emptyset$ then  we have $\Sing(X_{m,y})\subset \left(\pi_{X}^{m0}\right)^{-1}\left(\Sing(\varphi)\right)$ as a consequence of \cite[Proposition 1.1]{Ish}. Thus, $\Sing(X_{m,y})$ is of codimension at least $E$ in  $X_{m,y}$ for every $m\geq 0$ and $y\in Y_m$ such that $X_{m,y}\neq \emptyset$. Thus, $\varphi$ is $E$-smooth.

Suppose that $Y$ is smooth and $\varphi$ is $E$-smooth. Then $Y_m$ is smooth for all $m$. By the assumption on $X$, we can use Proposition \ref{singjet} to have that  $X_m$ is equi-dimensional and a locally complete intersection for all $m\geq 0$.  By these facts together with the condition $(\textit{D})$, the smoothness of $Y_m$, the miracle flatness theorem and \cite[Chapter III, Proposition 9.5]{Harts}, we see that $\varphi_m$ is flat and every non-empty fiber of $\varphi_m$ is a locally complete intersection of pure dimension $\dim(X_m)-\dim(Y_m)$ for all $m\geq 0$. This and the condition $(\textit{S1})$ imply that any non-empty fiber of $\varphi_m$ is reduced.
Thus, we can adapt the argument in \cite[Theorem 3.3]{EMY} to have $\Sing(X_{m,z})=
\left(\pi_{X}^{m0}\right)^{-1}\left(\Sing(\varphi)\right)\cap X_{m,z}$ if $X_{m,z}\neq\emptyset$.  Therefore, if the condition $(\textit{S1})$ holds, then so does the condition $(\textit{S2})$. Thus, $\varphi$ is strongly $E$-smooth.
\end{proof}
Note that our condition $(\textit{D})$ in Definition \ref{esm} is a bit weaker than the jet-flat condition in \cite{C-G-H}. However, when $Y$ is smooth, these two conditions agree. Moreover, we also have a version of \cite[Remark 2.8]{C-G-H} for  detecting the condition $(\textit{D})$ as follows.
\begin{lem}Using the notation of Definition \ref{sesm}, suppose that $X$ and $Y$ are equi-dimensional, locally complete intersections having only semi-log canonical singularities. Let $\varphi:X\to Y$ be a morphism. Then $\varphi$ satisfies the condition $(\textit{D})$ if one of the following conditions holds:
\begin{itemize}
\item[(\textit{a}),]$\varphi_m$ is flat for all $m\geq 0$, in other words, $\varphi$ is jet-flat in the sense of \cite{C-G-H}.
\item[(\textit{b}),]$\varphi$ is flat and every non-empty fiber of $\varphi$ is a locally complete intersection having only semi-log canonical singularities.
\end{itemize}

Additionally, suppose that $Y$ is smooth, then $(\textit{D})\Leftrightarrow(\textit{a})\Leftrightarrow(\textit{b})$.  

\end{lem}
\begin{proof}By Proposition \ref{singjet}, we conclude that $X_m$ and $Y_m$ are of pure dimension $(m+1)\dim(X)$ and $(m+1)\dim(Y)$ respectively.

If $\varphi_m$ is flat, then every non-empty fiber of $\varphi_m$ is of pure dimension $\dim(X_m)-\dim(Y_m)$ (see \cite[Chapter III, Proposition 9.5]{Harts}).  Thus,  $(\textit{a})$ implies $(\textit{D})$. 

On the other hand, suppose that $(\textit{b})$ holds. Since $\varphi$ is flat and $X,Y$ are equi-dimensional, every non-empty fiber of $\varphi$ is of pure dimension $\dim(X)-\dim(Y)$ as mentioned above. Let $y\in Y_m$ such that $\varphi_m^{-1}(y)$ is non-empty. The proof of Proposition \ref{strong} implies that 
$$\dim(X_m)-\dim(Y_m)\leq \dim\left(\varphi_m^{-1}(y)\right)\leq \dim\left(\varphi_m^{-1}\left(s_{Y,m}\left(\pi^{m0}(y)\right)\right)\right).$$
Note that $\left(\varphi^{-1}\left(\pi^{m0}(y)\right)\right)_m\simeq \varphi_m^{-1}\left(s_{Y,m}\left(\pi^{m0}(y)\right)\right)$. Since $\varphi^{-1}\left(\pi^{m0}(y)\right)$ is a locally complete intersection having only semi-log canonical singularities, we can use Proposition \ref{singjet} to have 
\begin{align*}
\dim\left(\varphi_m^{-1}\left(s_{Y,m}\left(\pi^{m0}(y)\right)\right)\right)&=\dim\left(\left(\varphi^{-1}\left(\pi^{m0}(y)\right)\right)_m\right)\\
&=(m+1)\dim\left(\varphi^{-1}(\pi^{m0}(y))\right)\\&=(m+1)(\dim(X)-\dim(Y))\\&=\dim(X_m)-\dim(Y_m).
\end{align*}
Thus,  $\dim\left(\varphi_m^{-1}(y)\right)=\dim(X_m)-\dim(Y_m)$ whenever $\varphi_m^{-1}(y)\neq\emptyset$. So $(\textit{b})$ implies $(\textit{D})$.

Suppose that $Y$ is smooth. Then $Y_m$ is smooth for all $m$. We use this, the miracle flatness theorem and the fact that  $X_m$ and $Y_m$ are equi-dimensional to see that $(\textit{D})$ implies $(\textit{a})$. It follows from this and the above argument that $(\textit{b})\Rightarrow(\textit{D})\Leftrightarrow(\textit{a})$. It remains to show that $(\textit{a})\Rightarrow(\textit{b})$.  Indeed, suppose moreover that the condition $(\textit{a})$ holds. If $m\geq 0$ then every non-empty fiber of $\varphi_m$ is of pure dimension $\dim(X_m)-\dim(Y_m)$ as explained above. The smoothness of $Y$ and the fact that $X$ is a locally complete intersection imply that every non-empty fiber of $\varphi=\varphi_0$ is a locally complete intersection. Let $y\in Y$ such that $\varphi^{-1}(y)\neq\emptyset$. We have $(\varphi^{-1}(y))_m\simeq \varphi_m^{-1}(s_{Y,m}(y))$. Therefore, $\dim\left((\varphi^{-1}(y))_m\right)=\dim\left(\varphi_m^{-1}(s_{Y,m}(y))\right)=\dim\left(X_m\right)-\dim\left(Y_m\right)=(m+1)\left(\dim(X)-\dim(Y)\right)=(m+1)\dim\left(\varphi^{-1}(y)\right)$ for all $m\geq 0$. Thus, we can use Proposition \ref{singjet} to conclude that the fiber $\varphi^{-1}(y)$ is a locally complete intersection having only semi-log canonical singularities whenever it is non-empty. This means that $(\textit{a})\Rightarrow(\textit{b})$ as desired.


\end{proof}
By a slight improvement of the proof of \cite[Theorem 4.11]{C-G-H}, we have the following result for counting points of schemes over finite rings.
\begin{prop}\label{countEsmth}Let $K$ be a number field and $X, Y$ be $\cO_K$-schemes of finite type. Suppose that $X_K, Y_K$ are equi-dimensional and locally complete intersections having only semi-log canonical singularities. Let  $E$ be a positive integer and $\varphi:X\to Y$ be an $\cO_K$-morphism such that $\varphi_{K}$ is strongly $E$-smooth. Then there is a positive constant $C$ and an integer $M$ such that 
$$\left|\frac{\#\left(_m^L\varphi\right)^{-1}(y)}{q_L^{m(\dim(X_K)-\dim(Y_K))}}-\frac{\# \left(_1^L\varphi\right)^{-1}(\pi_{m1}(y))}{q_L^{\dim(X_K)-\dim(Y_K)}}\right|\leq Cq_L^{-E}$$
for all $L\in \tilde{\cL}_{K,M}$, all $m\geq 1$ and all $y\in Y(\cO_L/(\varpi_L^{m}))$, where $_m^L\varphi:X(\cO_L/(\varpi_L^{m}))\to Y(\cO_L/(\varpi_L^{m}))$ is the canonical map induced by $\varphi$.
\end{prop}
\begin{proof}Let $X_m$ (resp. $Y_m$) be the $m^{\textnormal{th}}$ jet scheme associated to $X$ (resp.  $Y$) over $\cO_K$ as in Section \ref{SJ}. Let $\varphi_m:X_m\to Y_m$ be the $m^{\textnormal{th}}$ jet morphism of $\varphi$. For each $L\in\tilde{\cL}_{K,1}$, $y\in Y_m(k_L)=Y(k_L[t]/(t^{m+1}))$, we set $X_{m,y}=\varphi_m^{-1}(y)$. 
By the argument of \cite[Theorems 4.7 and 4.11]{C-G-H} and Remark \ref{jetloc}, it is sufficient to prove that for each  $m\in\NN$  there is a constant $C_m$ and an integer $M_m$  such that 
\begin{equation*}\left|\frac{\#\left(X_{m,y}\cap \left(\pi_{X}^{m0}\right)^{-1}\left(\Sing(\varphi_{k_L})\right)\right)(k_L)}{q_L^{(m+1)(\dim(X_K)-\dim(Y_K))}}-\frac{\# \left(\varphi^{-1}\left(\pi_Y^{m0}(y)\right)\cap \Sing\left(\varphi_{k_L}\right)\right)(k_L)}{q_L^{\dim(X_K)-\dim(Y_K)}}\right|\leq C_mq_L^{-E}
\end{equation*}for all $L\in \tilde{\cL}_{K,M_m}$ and all $y\in Y_m(k_L)$. By Proposition \ref{singjet}, we have $\dim(X_m\otimes K)=(m+1)\dim(X_K)$ and $\dim(Y_m\otimes K)=(m+1)\dim(Y_K)$. This together with the strongly $E$-smooth condition of $\varphi_K$ and the logical compactness  in \cite[Corollary 2.2.10]{Marker} imply that there is an integer $M_m$ such that $$\dim\big(X_{m,y}\cap \left(\pi_{X}^{m0}\right)^{-1}\left(\Sing(\varphi_{k_L})\right)\big)\leq (m+1)(\dim(X_K)-\dim(Y_K))-E$$ and 
$$\dim\big(\varphi^{-1}\left(\pi_Y^{m0}(y)\right)\cap \Sing\left(\varphi_{k_L}\right)\big)\leq \dim(X_K)-\dim(Y_K)-E$$ if $L\in \tilde{\cL}_{K,M_m}$ and $y\in Y_m(k_L)$. The assertion follows by these inequalities and Proposition \ref{Lang-Weil}.
\end{proof}
Now, we have enough materials to state every result of this section.


\begin{prop}\label{counting}Let $\cI$ be a non-zero ideal of $\cO_K[x_1,\dots,x_n]$ such that $\cI$ can be generated by $r$ elements and $\cI_K\neq (1)$. Suppose that $\moi_{K}^{(r)}(\cI)>r$ and Conjecture \ref{avelocIgu} holds for $\cI,r, Z=\AA_{\cO_K}^n$. Let  $0<\epsilon<\moi_{K}^{(r)}(\cI)-r$. Then there exists an integer $M$ depending only on $\epsilon, r, \cI$ such that for all local fields $L\in\tilde{\cL}_{K,M}$ and all integers $m>0$, we have
$$\left|\frac{\#X(\cO_L/\varpi_L^m\cO_L)}{q_L^{m(n-r)}}-\frac{\# X(k_L)}{q_L^{n-r}}\right|\leq \frac{q_L^{-2(\moi_{K}^{(r)}(\cI)-r-\epsilon)}}{1-q_L^{r+\epsilon-\moi_{K}^{(r)}(\cI)}}.$$
\end{prop}
\begin{proof}Our claim for $m=1$ is trivial. Suppose that $m\geq 2$. Since Conjecture \ref{avelocIgu} holds for  $\cI,r, Z=\AA_{\cO_K}^n$, there is an integer $M_0$ and a constant $c$ depending only on $r, \cI$ such that 
$$|E_{L,\cI}^{(r)}(m)|\leq cm^{n-1}q_L^{-m\moi_{K}^{(r)}(\cI)},$$
for all $m\geq 2$ and all $L\in\tilde{\cL}_{K,M_0}$. Thus, for each $\epsilon>0$, there exists an integer $M_{\epsilon}$ depending only on $\epsilon,r,\cI$ such that 
$$|E_{L,\cI}^{(r)}(m)|\leq q_L^{-m(\moi_{K}^{(r)}(\cI)-\epsilon)},$$
for all $m\geq 2$ and all $L\in\tilde{\cL}_{K,M_{\epsilon}}$. By the definition, we have 
$$E_{L,\cI}^{(r)}(m)=q_L^{-mn}(\#X(\cO_L/(\varpi_L^m))-q_L^{n-r}\#X(\cO_L/(\varpi_L^{m-1})))$$
if $m\geq 2$. Hence,
\begin{equation}\label{differ}
\left|\frac{\# X(\cO_L/\varpi_L^m\cO_L)}{q_L^{m(n-r)}}-\frac{\# X(\cO_L/\varpi_L^{m-1}\cO_L)}{q_L^{(m-1)(n-r)}}\right|\leq q_L^{-m(\moi_{K}^{(r)}(\cI)-r-\epsilon)}
\end{equation}
if $m\geq 2$ and $L\in\tilde{\cL}_{K,M_\epsilon}$.
Therefore, we have
\begin{align*}
\left|\frac{\#X(\cO_L/\varpi_L^m\cO_L)}{q_L^{m(n-r)}}-\frac{\# X(k_L)}{q_L^{n-r}}\right|&\leq \sum_{i=2}^m\left|\frac{\# X(\cO_L/\varpi_L^m\cO_L)}{q_L^{m(n-r)}}-\frac{\# X(\cO_L/\varpi_L^{m-1}\cO_L)}{q_L^{(m-1)(n-r)}}\right|\\
&\leq \sum_{i=2}^mq_L^{-m(\moi_{K}^{(r)}(\cI)-r-\epsilon)}\leq \frac{q_L^{-2(\moi_{K}^{(r)}(\cI)-r-\epsilon)}}{1-q_L^{r+\epsilon-\moi_{K}^{(r)}(\cI)}}
\end{align*}
if $0<\epsilon<\moi_{K}^{(r)}(\cI)-r$, $m\geq 2$ and $L\in\tilde{\cL}_{K,M_{\epsilon}}$.
\end{proof}

From the definition of $\Oi_L^{(r)}(\cI)$ in Section 2.2 and the proof of Proposition \ref{counting}, we obtain the following corollary.
\begin{cor}\label{countab}
With the assumption of Proposition \ref{counting}, for each $L\in \cL_{K,1}$ such that $\Oi_L^{(r)}(\cI)>r$ and $0<\epsilon<\Oi_L^{(r)}(\cI)-r$, there is a constant $c_{L,\epsilon,\cI}$ depending only on $L,\epsilon,\cI$ such that
$$\left|\frac{\#X(\cO_L/\varpi_L^m\cO_L)}{q_L^{m(n-r)}}-\frac{\# X(k_L)}{q_L^{n-r}}\right|\leq c_{L,\epsilon,\cI}\frac{q_L^{-2(\Oi_{L}^{(r)}(\cI)-r-\epsilon)}}{1-q_L^{r+\epsilon-\Oi_{L}^{(r)}(\cI)}}.$$
\end{cor}

\begin{proof}[Proof of \ref{count}]Let $\sigma_0>r$. Let $\epsilon$ be a small enough positive number such that $\sigma_0-\epsilon>r$. By the proof of Proposition \ref{counting}, there is an integer $M_{\epsilon}$ such that for all $L\in\tilde{\cL}_{\QQ,M_{\epsilon}}$ and all $m\geq 1$, we have 
$$\left|\frac{\#X(\cO_L/\varpi_L^m\cO_L)}{q_L^{m(n-r)}}-\frac{\# X(k_L)}{q_L^{n-r}}\right|\leq \frac{q_L^{-2(\sigma_0-r-\epsilon)}}{1-q_L^{r+\epsilon-\sigma_0}}.$$
Then we can improve slightly the proof of \cite[Theorem 4.11]{C-G-H} with the remark that the dimension of a non-empty scheme is always a natural integer to see that the structure morphism $X_K\to \spec(K)$ is $E_{\epsilon}$-smooth, where $E_{\epsilon}=\lceil 2(\sigma_0-r-\epsilon)\rceil$. By taking $\epsilon$ small enough, we have $\lceil 2(\sigma_0-r-\epsilon)\rceil=\lceil 2(\sigma_0-r)\rceil$. Thus, the structure morphism $X_K\to \spec(K)$ is $E$-smooth if $E=\lceil 2(\sigma_0-r)\rceil$. Since $\spec(K)$ is smooth, Proposition \ref{se} implies that $\varphi$ is strongly $E$-smooth. Thus, we can use Proposition \ref{countEsmth} to see that there is an integer $M$ and a positive constant $C$ such that for all $L\in\tilde{\cL}_{K,M}$ and all $m\geq 1$, we have 
$$\left|\frac{\#X(\cO_L/\varpi_L^m\cO_L)}{q_L^{m(n-r)}}-\frac{\# X(k_L)}{q_L^{n-r}}\right|\leq Cq_L^{-\lceil 2(\sigma_0-r)\rceil}.$$
\end{proof}
As mentioned in Section \ref{intro},  \ref{count} can be generalized as follows.
\begin{prop}\label{acount}Suppose that Conjecture \ref{avelocIgu} holds.  Let $X$ be an $\cO_K$-scheme of finite type. Suppose that $X_K$ is equi-dimensional.  Let $\cU=(U_i)_{1\leq i\leq N}$ be an open affine cover of $X$ such that $U_i$ is of finite type over $\cO_K$ for all $i$. We set 
$Z_I:=(\cap_{i\in I} U_i)\setminus (\cup_{j\notin I} U_j)$ for each non-empty subset $I$ of $\{1,\dots,N\}$.
 Suppose that $\sigma(\cU,X):=\min_{I}\moi_{K,Z_I}^{\textnormal{aloc}}(X_K)>-\dim(X_K)$. Then there is a constant $C$ and an integer $M$ such that 
$$\left|\frac{\#X(\cO_L/\varpi_L^m\cO_L)}{q_L^{m\dim(X_K)}}-\frac{\# X(k_L)}{q_L^{\dim(X_K)}}\right|\leq Cq_L^{-\lceil 2(\sigma(\cU,X)+\dim(X_K))\rceil}$$
for all $m\geq 1$ and all $L\in\tilde{\cL}_{K,M}$.
\end{prop}
\begin{proof}Let us use the map $\pi_{ij}^Y: Y(\cO_L/(\varpi_L^i))\to Y(\cO_L/(\varpi_L^j))$ in Section \ref{nuloc} if $i\geq j\geq 0$ and $Y$ is an $\cO_K$-scheme.
It is easily seen that $X=\sqcup_{\emptyset\neq I\subset \{1,\dots,N\}}Z_I$ and $Z_I\cap Z_J=\emptyset$ if $I\neq J$. Thus, we have 
$$\frac{\#X(\cO_L/\varpi_L^m\cO_L)}{q_L^{m\dim(X_K)}}-\frac{\# X(k_L)}{q_L^{\dim(X_K)}}=\sum_{\emptyset\neq I\subset \{1,\dots,N\}}\left(\frac{\#\left(\left(\pi_{m1}^X\right)^{-1}\left(Z_I(k_L)\right)\right)}{q_L^{m\dim(X_K)}}-\frac{\# Z_I(k_L)}{q_L^{\dim(X_K)}}\right).$$
So it is sufficient to show that for each $\emptyset\neq I\subset\{1,\dots,N\}$, there is a constant $C_I$ and an integer $M_I$ such that 
\begin{equation}\label{equal1}
\left|\frac{\#\left(\left(\pi_{m1}^X\right)^{-1}\left(Z_I(k_L)\right)\right)}{q_L^{m\dim(X_K)}}-\frac{\# Z_I(k_L)}{q_L^{\dim(X_K)}}\right|\leq C_Iq_L^{-\lceil 2(\sigma(\cU,X)+\dim(X_K))\rceil}
\end{equation}
for all $m\geq 1$ and all $L\in\tilde{\cL}_{K,M_I}$. Indeed, by Definition \ref{moiloc2} and Corollary \ref{naig}, we can take an affine neighbourhood $U$ of $Z_I\otimes K$ in $X_K$ and a closed embedding  $\iota: U\to \AA_{K}^n$ such that $(U,\iota)$ computes $\moi_{K,Z_I}^{\textnormal{aloc}}(X_K)$. By Definition \ref{moiloc2} and the fact that $X_K$ is equi-dimensional, one has
$$\moi_{K,\iota(Z_I\otimes K)}^{\textnormal{loc}}(\iota(U))=n+\moi_{K,Z_I}^{\textnormal{aloc}}(X_K)>n-\dim(X_K)=n-\dim(\iota(U)).$$
Thus, Definition \ref{moiloc1}, Remark \ref{well} and \ref{locpole} imply that there is an affine neighbourhood  $V$ of $\iota(Z_I\otimes K)$ in $\AA_K^n$ together with polynomials $f_1,\dots,f_r\in \cO_K[x_1,\dots,x_n]$ such that $r=n-\dim(\iota(U))$, $\iota(U)\cap V=\spec(\cO_V(V)/(f_1,\dots,f_r))$ and $\moi_{K,\iota(Z_I\otimes K)}^{\textnormal{loc}}(\iota(U))=\moi_{K,\tilde{Z}_I}^{(r)}(\cI)$, where $\cI=(f_1,\dots,f_r)$ and $\tilde{Z}_I$ is any $\cO_K$-subscheme of $\tilde{U}=\spec(\cO_K[x_1,...,x_n]/\cI)$ of finite type satisfying $\tilde{Z}_I\otimes K=\iota(Z_I\otimes K)$. Since $\tilde{Z}_I\otimes K=\iota(Z_I\otimes K)\simeq Z_I\otimes K$ and $Z_I,\tilde{Z}_I$ are of finite type over $\cO_K$, there is an integer $M>0$ such that $Z_I\otimes \cO_K[1/M]\simeq \tilde{Z}_I\otimes \cO_K[1/M]$. On the other hand, by the definitions of $\iota$ and $\tilde{U}$, and enlarging $M$ if needed, there is an open subscheme $X_0$ containing  $Z_I\otimes \cO_K[1/M]$ of $X\otimes\cO_K[1/M]$ and an  open subscheme $\tilde{U}_0$ containing $\tilde{Z}_I\otimes \cO_K[1/M]$ of $\tilde{U}\otimes\cO_K[1/M]$ together with an isomorphism $\tau: X_0\to \tilde{U}_0$ of $\cO_K[1/M]$-schemes such that $\tau(Z_I\otimes \cO_K[1/M])=\tilde{Z}_I\otimes \cO_K[1/M]$. 
Thus, we have 
\begin{align*}
 E_{L,\tilde{Z}_I,\cI}^{(r)}(m)=&q_L^{-mn}\left(\#\left(\left(\pi_{m1}^{\tilde{U}}\right)^{-1}\left(\tilde{Z}_I\left(k_L\right)\right)\right)-q_L^{n-r}\#\left(\left(\pi_{(m-1)1}^{\tilde{U}}\right)^{-1}\left(\tilde{Z}_I(k_L)\right)\right)\right)\\
 =&q_L^{-mn}\left(\#\left(\left(\pi_{m1}^{X}\right)^{-1}\left(Z_I(k_L)\right)\right)-q_L^{n-r}\#\left(\left(\pi_{(m-1)1}^{X}\right)^{-1}\left(Z_I(k_L)\right)\right)\right)\\
 \end{align*}
for all $m\geq 2$ and all $L\in\tilde{\cL}_{K,M}$. Now, we can proceed as in the proof of Proposition \ref{counting} and the proof of \ref{count} to obtain a positive constant $C_I$ and an integer $M_I$ such that (\ref{equal1}) holds.

\end{proof}

\begin{named}{Theorem G}\label{locCounting}Suppose that Conjecture \ref{avelocIgu} holds. Let $X$ be an $\cO_K$-scheme of finite type such that $X_K$ is equi-dimensional and a locally complete intersection having only rational singularities. Let $\cU=\{U_1,\dots,U_N\}$ be an open affine cover of $X$ such that for each $i$, $U_i$ is an $\cO_K$-scheme of finite type and there is a closed embedding $\iota_i:U_i\otimes K\to\AA_K^{n_i}$ whose image is a complete intersection in $\AA_K^{n_i}$. Since $X_K$ is a locally complete intersection, such a cover $\cU$ exists. Let $\sigma(\cU,X)=\min_{I}\moi_{K,Z_I}^{\textnormal{aloc}}(X_K)$, where $Z_I:=(\cap_{i\in I} U_i)\setminus (\cup_{j\notin I} U_j)$ for $\emptyset\neq I\subset \{1,...,N\}$.  Then $\sigma(\cU,X)>-\dim(X_K)$. Moreover, there is a constant $C$ and an integer $M$ such that $$\left|\frac{\#X(\cO_L/\varpi_L^m\cO_L)}{q_L^{m\dim(X_K)}}-\frac{\# X(k_L)}{q_L^{\dim(X_K)}}\right|\leq Cq_L^{-\lceil 2(\sigma(\cU,X)+\dim(X_K))\rceil}$$
for all $m\geq 1$ and all $L\in\tilde{\cL}_{K,M}$.
\end{named}
\begin{proof}Let $\emptyset\neq I\subset \{1,...,N\}$ and $i\in I$. We deduce from the assumptions on $\iota_i$ and $X_K$ that $\iota_i(U_i\otimes K)$ is equi-dimensional and a complete intersection having only rational singularities. Therefore, we can use Definition \ref{moiloc2} and \ref{locpole} to have 
$$
\moi_{K,Z_I}^{\textnormal{aloc}}(X_K)\geq \moi_{K,\iota(Z_I\otimes K)}^{\textnormal{loc}}(\iota_i(U_i\otimes K))-n_i> -\dim(\iota_i(U_i\otimes K)).$$
It is clear that $\dim(\iota_i(U_i\otimes K))=\dim(U_i\otimes K)=\dim(X_K)$. Hence
$$\sigma(\cU,X)=\min_{I}\moi_{K,Z_I}^{\textnormal{aloc}}(X_K)>-\dim(X_K).$$
This inequality and  Proposition \ref{acount} imply the second claim.
\end{proof}
Now, we will provide a version of \ref{locCounting} for flat families of varieties. In order to deal with this, we need some new notation.  Suppose that $X$ is an $\cO_K$-scheme of finite type. Let $\cE(X)$ be the set of all finite open affine covers $\cU=(U_i)_{i\in I}$ of $X$ such that $U_i$ is of finite type over $\cO_K$ for all $i\in I$. For each $\cU\in\cE(X)$, let $\sigma(U,X)$ be as in Proposition \ref{acount} and \ref{locCounting}.  By Definitions \ref{moiloc1}, \ref{moiloc2}, Proposition \ref{VZ}, Lemma \ref{eqzeta} and Corollary \ref{naig}, the set $\sigma(X)=\{\sigma(\cU,X)\mid \cU\in \cE(X)\}$ is finite since it is contained in the set $\Lambda((V_i,\iota_i,h_i)_{i\in I})=\{\frac{\nu_{ij}}{N_{ij}}-n_i\mid i\in I, j\in J_{h_i}\}$ for a fixed cover $\cV_0=(V_i)_{i\in I}$ of $X_K$ by open affine subschemes, a fixed closed embedding $\iota_i$ of $V_i$ to $\AA_K^{n_i}$ and a fixed log resolution $h_i:Y_i\to \AA_K^{n_i}$ of $\iota_i(V_i)$ whose numerical data is $((\nu_{ij},N_{ij}))_{j\in J_{h_i}}$ for each $i\in I$.  We set
$$\rho(K,X)=\max_{\cU\in\cE(X)}\lceil 2(\sigma(\cU,X)+\dim(X_K))\rceil.$$
When $X$ is a $K$-scheme of finite type, we take an $\cO_K$-scheme $\tilde{X}$ of finite type such that $X=\tilde{X}_K$ as in Remark \ref{variety} and set $\sigma(X):=\sigma(\tilde{X})$, $\rho(K,X):=\rho(K,\tilde{X})$. As mentioned  in Remark \ref{variety}, $\sigma(X)$ and $\rho(K,X)$ do not depend on the choice of $\tilde{X}$. We use the convention that $\rho(K,\emptyset)=+\infty$.

Let $X$ and $Y$ be $\cO_K$-schemes of finite type. Let $\varphi:X\to Y$ be an $\cO_K$-morphism then $\varphi$ is of finite type since $\cO_K$ is a Noetherian ring. Thus, we can cover $X_K$ (resp. $Y_K$) by finitely many open affine subschemes $W_i$ (resp. $V_j$) of finite type over $K$ for $i\in I$ (resp. $j\in J$) such that for each $i\in I$, there is $j\in J$ such that $\varphi(W_i)\subset V_j$. Let $y\in Y(\overline{K})$ then  $\{W_{iy}=W_{i}\cap\varphi_K^{-1}(y)\mid i\in I\}$ is a cover of $X_{Ky}=\varphi_K^{-1}(y)$ by open affine subschemes of finite type over $K[y]$. For each $i\in I$ and $j\in J$, we fix a closed embedding $\iota_{ij}:W_i\to V_j\times\AA_K^{n_{ij}}$ such that $\varphi_K|_{W_i}=\textnormal{pr}\circ\iota_{ij}$, where $\textnormal{pr}$ is the projection $V_j\times\AA_K^{n_{ij}}\to V_j$. We can use Definitions \ref{moiloc1}, \ref{moiloc2}, Proposition \ref{VZ}, Lemma \ref{eqzeta} and Corollary \ref{naig} again together with the argument in \cite[Property 1.23]{Mustata2} on constructing log resolution of $\iota_{ij}(W_{iy})$ in $\AA_{K(y)}^{n_{ij}}$ (piecewise) uniformly in $y\in V_j(\overline{K})$  to show that the sets $\cup_{y\in Y(\overline{K})}\sigma(X_{Ky})$ and $\{\rho(K(y),X_{Ky})\mid y\in Y(\overline{K})\}$ are finite. Now we set
$$E(K,\varphi)=\min_{y\in Y(\overline{K})} \rho(K(y),X_{Ky}).$$

\begin{prop}\label{FRST}Suppose that $X$ and $Y$ are $\cO_K$-schemes of finite type such that $X_K,Y_K$ are equi-dimensional. Let $\varphi:X\to Y$ be an $\cO_K$-morphism such that $\varphi_K:X_K\to Y_K$ is flat. With the notion of $FRS$ and $FTS$ morphisms in Definition \ref{defrs}, the following assertions hold:
\begin{itemize}
\item[(\textit{a}),]$E(K,\varphi)>0$ if and only if $\varphi_K$ is an $FRS$ morphism whose non-empty fibers are locally complete intersections. If this is the case, we have 
$$\dim(\Sing(X_{Ky}))\leq \dim(X_{Ky})-E(K,\varphi)-1$$ and 
$$\dim\left(\big(\pi_{X_{Ky}}^{(m-1)0}\big)^{-1}\left(\Sing(X_{Ky})\right)\right)\leq m\dim(X_{Ky})-E(K,\varphi)$$ for all $m\geq 2$ and all $y\in Y(\overline{K})$ such that $X_{Ky}\neq\emptyset$.
\item[(\textit{b}),]If Conjecture \ref{avelocIgu} holds and $E(K,\varphi)> 1$, then $\varphi_K$ is an $FTS$ morphism whose non-empty fibers are locally complete intersections.
\end{itemize}

\end{prop}
\begin{proof}
Let $y\in Y(\overline{K})$ such that $X_{Ky}\neq\emptyset$. Since $\varphi_K$ is flat and $X_K, Y_K$ are equi-dimensional, we have that $X_{Ky}$ is of pure dimension $\dim(X_K)-\dim(Y_K)$ (see \cite[Chapter III, Proposition 9.5]{Harts}). Let $\cU=\{U_1,...,U_N\}\in \cE(\tilde{X}(y))$ where $\tilde{X}(y)$ in an $\cO_{K(y)}$-scheme of finite type such that $X_{Ky}= \tilde{X}(y)\otimes K(y)$. Let $Z_I= (\cap_{i\in I} U_i)\setminus (\cup_{j\notin I} U_j)$ for each $\emptyset\neq I\subset \{1,...,N\}$. We can use Definition \ref{moiloc2} and \ref{locpole} to show that $\moi_{K(y),Z_I}^{\textnormal{aloc}}(X_{Ky})>-\dim(X_{Ky})$ if and only if there is an open neighbourhood $V$ of $Z_I\otimes K(y)$ in $X_{Ky}$ such that $V$ is a locally complete intersection having only rational singularities. Thus, $$\sigma(\cU,X_{Ky})=\min_{I}\moi_{K(y),Z_I}^{\textnormal{aloc}}(X_{Ky})>-\dim(X_{Ky})$$ if and only if $X_{Ky}$ is a locally complete intersection having only rational singularities. This together with the definitions of $\rho(K(y),X_{Ky})$ and $E(K,\varphi)$ imply the first claim of (\textit{a}). We take an $\cO_{K(y)}$-scheme $Z$ of finite type such that $Z_{K(y)}=X_{Ky}$. By \ref{locCounting} and the definition of $E(K,\varphi)$, there is a constant $C$ and an integer $M$ such that 
\begin{equation}\label{upbou}
\left|\frac{\#Z(\cO_L/\varpi_L^m\cO_L)}{q_L^{m(\dim(X_K)-\dim(Y_K))}}-\frac{\# Z(k_L)}{q_L^{\dim(X_K)-\dim(Y_K)}}\right|\leq Cq_L^{-E(K,\varphi)}
\end{equation}
for all $m\geq 1$ and all $L\in\tilde{\cL}_{K(y),M}$. 
By Hensel's lemma, Remark \ref{jetloc} and (\ref{upbou}), for each $m\geq 1$, there is an integer $M_m$ such that
\begin{equation}\label{upbou1}
\left|\frac{\#\left(\left(\pi_{Z_{k_L}}^{(m-1)0}\right)^{-1}\left(\Sing\left(Z_{k_L}\right)\right)\right)(k_L)}{q_L^{m(\dim(Z_{K(y)}))}}-\frac{\# \Sing(Z_{k_L})(k_L)}{q_L^{\dim(Z_{K(y)})}}\right|\leq Cq_L^{-E(K,\varphi)}
\end{equation}for all $L\in\tilde{\cL}_{K(y),M_m}$. Thus, by the discussion in \cite[Theorem 4.11, $(3)\Rightarrow (1)$]{C-G-H}, one has $$\dim(\Sing(Z_{K(y)}))\leq \dim(Z_{K(y)})-E(K,\varphi)-1$$ and 
$$\dim\left(\big(\pi_{Z_{K(y)}}^{(m-1)0}\big)^{-1}\left(\Sing(Z_{K(y)})\right)\right)\leq m\dim\left(Z_{K(y)}\right)-E(K,\varphi)$$ for all $m\geq 2$. This completes the proof of (\textit{a}).  

To prove $(b)$, it is sufficient to prove that if Conjecture \ref{avelocIgu} holds and $E(K,\varphi)>1$ then $Z_{K(y)}=X_{Ky}$ is normal  and has only terminal singularities. 
To simplify notation, we set $\tilde{Z}=Z_{K(y)}$.
Since  $E(K,\varphi)>1$, it follows from (\textit{a}) that $\tilde{Z}$ is a locally complete intersection having only rational singularities. This and Proposition \ref{singjet} imply that 
$\tilde{Z}_{m-1}$ is equi-dimensional and a locally complete intersection of dimension $m\dim(\tilde{Z})$, where $\tilde{Z}_{m-1}$ is the $(m-1)^{th}$ jet scheme over $K(y)$ associated to $\tilde{Z}$. On the other hand, it is clear that $\Sing(\tilde{Z}_{m-1})\subset\big(\pi_{\tilde{Z}}^{(m-1)0}\big)^{-1}(\Sing(\tilde{Z}))$. Thus, $\Sing(\tilde{Z}_{m-1})$ is of codimension at least $E(K,\varphi)\geq 2$ in $\tilde{Z}_{m-1}$ for all $m\geq 1$. 
Therefore, $\tilde{Z}_{m-1}$ is Cohen-Macaulay and regular in codimension $1$ for all $m\geq 1$. So $\tilde{Z}_{m-1}$ is normal for all $m\geq 1$. Thus, Proposition \ref{singjet} implies that $\tilde{Z}$ has only terminal singularities.

\end{proof}
In the situation of Proposition \ref{FRST}, if we suppose moreover that $X_K$ and $Y_K$ are locally complete intersections having only semi-log canonical singularities then the map $\varphi$ is strongly $E(K,\varphi)$-smooth provided that $E(K,\varphi)>0$. This is the content of the following theorem.
\begin{thm}\label{improE}
Suppose that $X$ and $Y$ are $\cO_K$-schemes of finite type such that $X_K,Y_K$ are equi-dimensional and locally complete intersections having only semi-log canonical singularities. Let $\varphi:X\to Y$ be an $\cO_K$-morphism. Suppose that $\varphi_K$ is flat and $E(K,\varphi)>0$. In addition, suppose that Conjecture \ref{avelocIgu} holds.  Then $\varphi_K$ is strongly $E(K,\varphi)$-smooth. Moreover, there is a positive constant $C$  and an integer $M$ such that 
$$\left|\frac{\#\left(_m^L\varphi\right)^{-1}(y)}{q_L^{m(\dim(X_K)-\dim(Y_K))}}-\frac{\#\left(_1^L\varphi\right)^{-1}\left(\pi_{m1}(y)\right)}{q_L^{\dim(X_K)-\dim(Y_K)}}\right|\leq Cq_L^{-E(K,\varphi)}$$
for all $L\in \tilde{\cL}_{K,M}$ and all $y\in Y(\cO_L/(\varpi_L^{m}))$, where $_m^L\varphi:X(\cO_L/(\varpi_L^{m}))\to Y(\cO_L/(\varpi_L^{m}))$ is the canonical map induced by $\varphi$.

\end{thm}
\begin{proof}Firstly, we will show that $\varphi_K$ is strongly $E(K,\varphi)$-smooth.  Let $y\in Y(\overline{K})$ such that  $\tilde{Z}=X_{Ky}\neq\emptyset$. By the flatness of $\varphi_K$  and the fact that $X_K, Y_K$ are equi-dimensional, we have that $\tilde{Z}$ is of pure  dimension $\dim(X_K)-\dim(Y_K)$. It follows from Proposition \ref{FRST} that $$\dim(\Sing(\tilde{Z}))\leq \dim(\tilde{Z})-E(K,\varphi)-1$$ and 
$$\dim\left(\big(\pi_{\tilde{Z}}^{(m-1)0}\big)^{-1}\big(\Sing(\tilde{Z})\big)\right)\leq m\dim(\tilde{Z})-E(K,\varphi)$$ for all $m\geq 2$. For each $m\geq 0$, let $X_{K,m}$ (resp. $Y_{K,m}$) be the $m^{\textnormal{th}}$ jet scheme associated to $X_K$ (resp. $Y_K$) over $K$ as in Section \ref{SJ}. Let $\varphi_{K,m}:X_{K,m}\to Y_{K,m}$ be the $m^{\textnormal{th}}$ jet morphism of $\varphi_K$ for $m\geq 0$.  By Proposition \ref{singjet} and the assumption that $X_K, Y_K$ are equi-dimensional and locally complete intersections having only semi-log canonical singularities, we have $\dim(X_{K,m-1})=m\dim(X_K)$, $\dim(Y_{K,m-1})=m\dim(Y_K)$, so
$$\dim\left(\big(\pi_{\tilde{Z}}^{(m-1)0}\big)^{-1}\big(\Sing(\tilde{Z})\big)\right)\leq \dim(X_{K,m-1})-\dim(Y_{K,m-1})-E(K,\varphi)$$
for all $m\geq 1$. Since $\varphi_K$ is flat and $X_K,Y_K$ are equi-dimensional, we have $\Sing(\varphi_{K})\cap\tilde{Z}=\Sing(\tilde{Z})$. On the other hand, the condition $\varphi_{K}^{-1}(y)=\tilde{Z}$ implies $\tilde{Z}_{m-1}=\varphi_{K,m-1}^{-1}(s_{Y_{K},m-1}(y))\subset X_{K,m-1}$ for all $m\geq 1$. Therefore, we have
$$\big(\pi_{\tilde{Z}}^{(m-1)0}\big)^{-1}\big(\Sing(\tilde{Z})\big)=\big(\pi_{X_{K}}^{(m-1)0}\big)^{-1}\left(\Sing(\varphi_{K})\right)\cap \left(X_{K,m-1}\right)_{s_{Y_{K},m-1}(y)}$$
are of dimension at most $\dim(X_{K,m-1})-\dim(Y_{K,m-1})-E(K,\varphi)$ for all $m\geq 1$. By Proposition \ref{strong}, $\varphi_K$ is strongly $E(K,\varphi)$-smooth. Thus, we can use Proposition \ref{countEsmth} to obtain a positive constant $C'$ and an integer $M'$ such that 
$$\left|\frac{\#\left(_m^L\varphi\right)^{-1}(y)}{q_L^{m(\dim(X_K)-\dim(Y_K))}}-\frac{\# \left(_1^L\varphi\right)^{-1}(\pi_{m1}(y))}{q_L^{\dim(X_K)-\dim(Y_K)}}\right|\leq C'q_L^{-E(K,\varphi)}$$
for all $L\in \tilde{\cL}_{K,M'}$ and all $y\in Y(\cO_L/(\varpi_L^{m}))$.
\end{proof}
\begin{proof}[Proof of $(i)$ of \ref{FRS}] By the idea of \cite[Section 6]{G-H1}, it is sufficient to deal with number fields $K$. The main point here is to use Proposition \ref{delta} to check when a morphism is of relative dimension $d$ and whose fibers are geometrically irreducible. The proof consists of two steps. In the first step, we will prove our statement for the affine case. For this step, the problem reduces to proving a uniform bound for some family of exponential sums over finite fields. In the second step, we will try to reduce the statement to the corresponding statement for affine varieties.

\textbf{Step 1.}  We suppose that for each $i$, $X_i$ is the affine subscheme of $\AA_K^{n_i}$ associated to the ideal $\cI_i$ generated by polynomials $f_{i1},\dots,f_{iN_i}$ in $\cO_K[x_{i1},\dots,x_{in_i}]$ with $N_i=n_i-\dim(X_i)$. Since $X_i$ is geometrically irreducible and $\varphi_{i\overline{K}}(X_i\otimes \overline{K})$ is not contained in any proper affine subspace of $\AA_{\overline{K}}^r$ for all $i$, we have $\dim(X_i)\geq 1$ for all $i$. We denote by $\tilde{X}_i$ the $\cO_K$-scheme associated to the ideal generated by $f_{i1},\dots,f_{iN_i}$. Moreover, we suppose that $\varphi_i=(g_{i1},\dots,g_{ir})$ with polynomials $g_{i1},\dots,g_{ir}\in K[x_{i1},\dots,x_{in_i}]$.  Suppose that $\ell>2r$. Then we will prove that $\Phi^{(\ell)}=\varphi_1*...*\varphi_\ell:\cX^{(\ell)}=X_1\times...\times X_\ell\to\AA_K^r$ is $FGI$. So it suffices to verify that if $Q\in \overline{K}^r$ then $(\Phi_{\overline{K}}^{(\ell)})^{-1}(Q)$ is irreducible. Replacing $K$ by $K(Q)$ and using a change of coordinates if needed, we can suppose that $Q=0$.  Thus, we need to prove that $(\Phi^{(\ell)})^{-1}(0)$ is geometrically irreducible. Replacing $\Phi^{(\ell)}$ by $H\Phi^{(\ell)}$ for a suitable non-zero integer $H$, we can suppose that the coefficients of $g_{ij}$ belong to $\cO_K$ for all $i,j$. Then $(\Phi^{(\ell)})^{-1}(0)$ is associated to the ideal $\cI_K$, where
$$\cI=\bigg(f_{11},\dots,f_{1N_1},\dots,f_{\ell 1},\dots,f_{\ell N_\ell},\sum_{1\leq i\leq \ell}g_{i1},\dots,\sum_{1\leq i\leq \ell}g_{ir}\bigg).$$
We set $n=\sum_{1\leq i\leq \ell}n_i, N=\sum_{1\leq i\leq \ell}N_i, \cR=r+N$ and  $g_e=\sum_{1\leq i\leq \ell}g_{ie}$  for $1\leq e\leq r$. By Proposition \ref{delta} and the miracle flatness theorem, it is sufficient to show that there is a constant $\delta<-\cR$ and a constant $c$ such that 
\begin{equation}\label{eq1234}
\left|E_{\cI}^{(\cR)}(\mathfrak{p})\right|=\left|E_{K_\mathfrak{p},\cI}^{(\cR)}(1)\right|\leq c\cN_\mathfrak{p}^{\delta}
\end{equation}
for all $\mathfrak{p}\in\Specm(\cO_K)$.

By Remark \ref{absexp}, if $L\in \cL_{K,1}$ then we have
\begin{align*}
E_{L,\cI}^{(\cR)}(1)&=\int_{\cO_L^{n}\times (\cO_L^{\cR}\setminus \varpi_L\cO_L^{\cR})}\psi\bigg(\sum_{1\leq e\leq r}y_eg_{e}+\sum_{i=1}^\ell \sum_{1\leq j\leq N_i} y_{ij}f_{ij}\bigg)\mu_{L^{n+\cR}}\\
&=\underbrace{\int_{\cO_L^{n}\times \cO_L^{N}\times (\cO_L^r\setminus \varpi_L\cO_L^{r})}\psi\bigg(\sum_{1\leq e\leq r}y_eg_{e}+\sum_{i=1}^\ell \sum_{1\leq j\leq N_i} y_{ij}f_{ij}\bigg)\mu_{L^{n+\cR}}}_{J_1(L,1)}
\end{align*}
\begin{align}
&\hspace{2.4cm}+\underbrace{\int_{\cO_L^{n}\times (\cO_L^{N}\setminus \varpi_L\cO_L^{N})\times \varpi_L\cO_L^{r}}\psi\bigg(\sum_{1\leq e\leq r}y_eg_{e}+\sum_{i=1}^\ell \sum_{1\leq j\leq N_i} y_{ij}f_{ij}\bigg)\mu_{L^{n+\cR}}}_{J_2(L,1)}\label{eq0101}
\end{align}
for all additive characters $\psi$ of $L$ of conductor $1$. Here, we use the variables  $(y_1,\dots,y_r)\in L^r$, $(x_{ij})_{1\leq i\leq \ell, 1\leq j\leq n_i}\in L^n$ and $(y_{ij})_{1\leq i\leq\ell, 1\leq j\leq N_i}\in L^{N}$.

Now, we will estimate $J_{1}(K_{\mathfrak{p}},1)$ and $J_{2}(K_\mathfrak{p},1)$ if $\mathfrak{p}\in \Specm(\cO_K)$. First of all, if $L\in\cL_{K,1}$, $\psi$ is an additive character of $L$ of conductor $1$,  $(y_1,\dots,y_r)\in \varpi_L\cO_L^r$ and $(x_{ij})_{1\leq i\leq \ell, 1\leq j\leq n_i}\in \cO_L^n$ then $\psi(\sum_{1\leq e\leq r}y_eg_{e})=1$, thus we have 
\begin{align*}
J_2(L,1)&=\int_{\cO_L^{n}\times (\cO_L^{N}\setminus \varpi_L\cO_L^{N})\times \varpi_L\cO_L^{r}}\psi\bigg(\sum_{1\leq e\leq r}y_eg_{e}+\sum_{i=1}^\ell \sum_{1\leq j\leq N_i} y_{ij}f_{ij}\bigg)\mu_{L^{n+\cR}}\\
&=q_L^{-r}\int_{\cO_L^{n}\times (\cO_L^{N}\setminus \varpi_L\cO_L^{N})}\psi\bigg(\sum_{i=1}^\ell \sum_{1\leq j\leq N_i} y_{ij}f_{ij}\bigg)\mu_{L^{n+N}}\\
&=q_L^{-r}q_L^{-n}\left(\#\tilde{\cX}^{(\ell)}(k_L)-q_L^{n-N}\right)
=q_L^{-r}E_{L,I(\tilde{\cX}^{(\ell)})}^{(N)}(1),
\end{align*}
where $\tilde{\cX}^{(\ell)}=\tilde{X}_{1}\times...\times\tilde{X}_\ell$. On the other hand, one has $\dim(\cX^{(\ell)})=\sum_{1\leq i\leq \ell}\dim(X_i)=n-N$ and $\cX^{(\ell)}=\tilde{\cX}^{(\ell)}_K$ is geometrically irreducible. By Proposition \ref{delta}, there is a constant $c$ and a real number $\delta_0<-N$ such that
\begin{equation}\label{eq0123}
\left|J_2(K_\mathfrak{p},1)\right|\leq c\cN_{\mathfrak{p}}^{\delta_0-r}
\end{equation} 
for all $\mathfrak{p}\in\Specm(\cO_K)$.

On the other hand, by orthogonality of characters, we have
\begin{align}
J_1(L,1)&=\int_{\cO_L^{n}\times \cO_L^{N}\times (\cO_L^r\setminus \varpi_L\cO_L^{r})}\psi\bigg(\sum_{1\leq e\leq r}y_eg_{e}+\sum_{i=1}^\ell \sum_{1\leq j\leq N_i} y_{ij}f_{ij}\bigg)\mu_{L^{n+\cR}}\nonumber\\
&=\int_{\cO_L^r\setminus \varpi_L\cO_L^{r}}\bigg(\prod_{1\leq i\leq \ell}\int_{\cO_L^{n_i}\times \cO_L^{N_i}}\psi\bigg(\sum_{1\leq e\leq r}y_eg_{ie}+ \sum_{1\leq j\leq N_i} y_{ij}f_{ij}\bigg)\mu_{L^{n_i+N_i}}\bigg)\mu_{L^r}\nonumber\\
&=\int_{\cO_L^r\setminus \varpi_L\cO_L^{r}}\bigg(\prod_{1\leq i\leq \ell}q_L^{-n_i}\sum_{x\in \tilde{X}_i(k_L)}\Psi\bigg(\sum_{1\leq e\leq r}\overline{y}_eg_{ie}(x)\bigg)\bigg)\mu_{L^r},\label{eq003}
\end{align}
where $\Psi$ is the additive character of $k_L$ induced by $\psi$. For each $1\leq i\leq \ell$, since $\varphi_{i\overline{K}}(X_{i}\otimes \overline{K})$ is not contained in any proper affine subspace of $\AA_{\overline{K}}^r$, one has that $\sum_{1\leq e\leq r}a_ig_{ie}$ is non-constant on $X_{i}\otimes\overline{K}$ for all $(a_1,...,a_r)\in\overline{K}^r\setminus\{0\}$. By logical compactness  (see \cite[Corollary 2.2.10]{Marker}), there exists an integer $M$ such that for all finite fields $k$ of characteristic at least $M$ and all tuples $(\overline{y}_1,\dots,\overline{y}_r)\in \overline{k}^r\setminus\{0\}$, we have that
$\tilde{X}_{i}\otimes\overline{k}$ is irreducible and $\sum_{1\leq e\leq r}\overline{y}_ig_{ie}$ is non-constant on $\tilde{X}_{i}\otimes\overline{k}$. Thus, $(\sum_{1\leq e\leq r}\overline{y}_ig_{ie})^{-1}(a)\cap (\tilde{X}_{i}\otimes\overline{k})$ is of dimension at most $\dim(X_i)-1$ for all $a\in \overline{k}$ and all $(\overline{y}_1,\dots,\overline{y}_r)\in \overline{k}^r\setminus\{0\}$ provided that $\Ch(k)>M$. Therefore,  by enlarging $M$ if needed, it follows immediately from \cite[Theorem 2]{Kowalski:Def} that there is a constant $c_0$ such that 
$$\bigg|\sum_{x\in \tilde{X}_i(k_L)}\Psi\bigg(\sum_{1\leq e\leq r}\overline{y}_eg_{ie}(x)\bigg)\bigg|\leq c_0q_L^{-\frac{1}{2}}\#\tilde{X}_i(k_L)$$
for all $1\leq i\leq \ell$, all $L\in\tilde{\cL}_{K,M}$ and all tuples $(\overline{y}_1,\dots,\overline{y}_r)\in k_L^r\setminus\{0\}$. 
Thus, Proposition \ref{Lang-Weil} implies that there is a constant $c_1$ such that 
\begin{equation}\label{eq023}
\bigg|\sum_{x\in \tilde{X}_i(k_L)}\Psi\bigg(\sum_{1\leq e\leq r}\overline{y}_eg_{ie}(x)\bigg)\bigg|\leq c_1q_L^{\dim(X_i)-1/2}
\end{equation}
for all $1\leq i\leq \ell$, all $L\in\tilde{\cL}_{K,M}$ and all tuples $(\overline{y}_1,\dots,\overline{y}_r)\in k_L^r\setminus\{0\}$. By combining (\ref{eq003}) and (\ref{eq023}), we have
\begin{align*}
\left|J_1(L,1)\right|&\leq c_1^\ell q_L^{-n}q_L^{\sum_{1\leq i\leq \ell}\dim(X_i)}q_L^{-\ell/2}
=c^\ell q_L^{-n}q_L^{n-N}q_L^{-\ell/2}
=c^{\ell}q_L^{-N-\ell/2}
\end{align*}
for all $L\in\cL_{K,M}$. So there exists a constant $C$ such that 
\begin{equation}\label{eq0012}
\left|J_1(K_\mathfrak{p},1)\right|\leq C\cN_\mathfrak{p}^{-N-\ell/2}
\end{equation} for all $\mathfrak{p}\in \Specm(\cO_K)$. It follows from (\ref{eq0101}), (\ref{eq0123}) and (\ref{eq0012}) that $$\left|E_{L,\cI}^{(\cR)}(1)\right|\leq \left|J_1(K_\mathfrak{p},1)\right|+\left|J_2(K_\mathfrak{p},1)\right|\leq (C+c)\cN_\mathfrak{p}^{\max\{-N-\ell/2,\delta_0-r\}}$$ for all $\mathfrak{p}\in \Specm(\cO_K)$. Since $\delta_0<-N$ and $\ell>2r$, we have  $\delta=\max\{-N-\ell/2,\delta_0-r\}<-(N+r)=-\cR$. This implies (\ref{eq1234}) as desired.

\textbf{Step 2.} Since $X_{i}\otimes\overline{K}$ is irreducible, if $Y_1,Y_2$ are two non-empty open subsets of $X_{i}\otimes\overline{K}$, then $Y_1\cap Y_2\neq \emptyset$. We take a cover of $X_{i}\otimes\overline{K}$ by non-empty affine open subsets $(Y_{ij})_{j\in I_i}$, then  $Y_{ij}\cap Y_{ij'}\neq \emptyset$ for all $j,j'$. Since $X_{i}$ is a locally  complete intersection, we can suppose that each $Y_{ij}$ has a closed embedding to an affine space $\AA_{\overline{K}}^{n_{ij}}$  whose image is a complete intersection. We use the above argument to see that $\Phi_{\overline{K}}^{(\ell)}|_{Y_{1j_1}\times...\times Y_{\ell j_\ell}}$ is $FGI$ for all tuples $(j_1,\dots,j_\ell)\in I_1\times...
\times I_\ell$. Thus, for each $Q\in \overline{K}^r$ and each tuple $(j_1,\dots,j_\ell)\in I_1\times...
\times I_\ell$ we have $(\Phi_{\overline{K}}^{(\ell)})^{-1}(Q)\cap Y_{1j_1}\times...\times Y_{\ell j_\ell}$ is geometrically irreducible of dimension $\dim(\cX^{(\ell)})-r$. For two tuples $(j_1,\dots,j_\ell),(j'_1,\dots,j'_\ell)\in I_1\times...\times I_\ell$, we can repeat the above argument with $X_i=Y_{ij_i}\cap Y_{ij'_i}$ for $1\leq i\leq \ell$ to obtain a  suitable non-empty open subset $Z$ of $Y_{1j_1}\times...\times Y_{\ell j_\ell}\cap Y_{1j'_1}\times...\times Y_{\ell j'_\ell}$ such that  $(\Phi_{\overline{K}}^{(\ell)})^{-1}(Q)\cap Z$ is geometrically irreducible of dimension  $\dim(\cX^{(\ell)})-r$. So for each $Q\in \overline{K}^r$, $\big((\Phi_{\overline{K}}^{(\ell)})^{-1}(Q)\cap Y_{1j_1}\times...\times Y_{\ell j_\ell}\big)_{(j_1,\dots,j_\ell)\in I_1\times...
\times I_\ell}$ forms an open cover of $(\Phi_{\overline{K}}^{(\ell)})^{-1}(Q)$ by geometrically irreducible varieties such that $(\Phi_{\overline{K}}^{(\ell)})^{-1}(Q)\cap Y_{1j_1}\times...\times Y_{\ell j_\ell}\cap Y_{1j'_1}\times...\times Y_{\ell j'_\ell}\neq \emptyset$ for all pairs $(j_1,\dots,j_\ell),(j'_1,\dots,j'_\ell)$. Hence, 
$(\Phi_{\overline{K}}^{(\ell)})^{-1}(Q)$ is geometrically irreducible of dimension $\dim(\cX^{(\ell)})-r$ for all $Q\in \overline{K}^r$. Therefore, the miracle flatness theorem implies that $\Phi^{(\ell)}$ is $FGI$.

\end{proof}

\begin{proof}[Proof of $(ii)$ of \ref{FRS}] 
The claim is trivial if $D=1$, so we can suppose that $D>1$. As in \cite[Section 6]{G-H1}, it is sufficient to prove our claims for any number field $K$. The proof will be divided into several steps.

\textbf{Step 1:}\textit{ Reduce the statement to proving a lower bound of the motivic oscillation index of ideals.}

 Let $r\geq 1$ and $(X_i)_{i\geq 1}$ be $K$-varieties of degree complexity at most $(R,D)$ such that $X_i$ is smooth for all $i\geq 1$. Let $(\varphi_i:X_i\to \AA_K^r)_{1\leq i}$ be $K$-morphisms as in the statement. We will show that if $\ell>N(r,R,D)$ (resp. $N'(r,R,D)$) then $\Phi^{(\ell)}=\varphi_1*...*\varphi_\ell: \cX^{(\ell)}=X_1\times...\times X_\ell\to\AA_K^r$ is an $FRS$ (resp. $FTS$) morphism. Let $P=(x_1,\dots,x_\ell)\in \cX^{(\ell)}(\overline{K})=X_1(\overline{K})\times... \times X_\ell(\overline{K})$, we will check when the fiber $\cX_{Q}^{(\ell)}$ is a locally complete intersection of dimension $-r+\sum_{i=1}^{\ell} \dim(X_i)$ and has only rational singularities in a small neighbourhood of $P$, where $Q=\Phi_{K(P)}^{(\ell)}(P)$. Replacing $K$ by $K(P)$ if needed, we can suppose that $P\in \cX^{(\ell)}(K)$.  This is a local question, so we can suppose that  for each $i$, $X_i$ is an affine closed subvariety of $\AA_K^{n_i}$ for some $ n_i>0$. We also can suppose that $I(X_i)$ is generated by polynomials $f_{i1},\dots,f_{iN_i}$ of degree at  most $D$ in the variables $x_{i1},\dots,x_{in_i}$ over $\cO_K$ and 
$N_i=n_i-\dim(X_i)\leq R$ for all $i$. Since $\varphi_{i\overline{K}}(X_{ij})$ is not contained in any proper affine subspace of $\AA_{\overline{K}}^r$ for each irreducible component $X_{ij}$ of $X_{i}\otimes\overline{K}$ and each $1\leq i\leq \ell$, we have $\dim(X_i)\geq 1$ for all $i$, thus 
\begin{equation}\label{nNi}
N_i<n_i
\end{equation} 
for all $i$.  For each $i$, let $\varphi_i=(g_{i1},\dots,g_{ir})$ for polynomials $(g_{ij})_{1\leq j\leq r}$ of degree at most $D$ in the variables $x_{i1},\dots,x_{in_i}$ over $K$.   Then $\cX^{(\ell)}$ is the closed subvariety of $\AA_K^{\sum_{1\leq i\leq \ell}n_i}$ associated to the ideal $I(\cX^{(\ell)})=(f_{11},\dots,f_{1N_1},\dots,f_{\ell 1},\dots,f_{\ell N_\ell})$. On the other hand, one has $$\Phi^{(\ell)}=\bigg(\sum_{1\leq i\leq \ell}g_{i1},\dots,\sum_{1\leq i\leq \ell}g_{ir}\bigg).$$ By a similar argument as in the proof of Part $(i)$ of \ref{FRS}, we can suppose that $P=0\in \AA_K^{\sum_{i=1}^\ell n_i}, Q=0\in \AA_K^r$ and the coefficients of $g_{ij}$ belong to $\cO_K$ for all $i,j$. Let $\cI$ be the ideal of $\cO_K[x_{ij}]_{1\leq i\leq \ell, 1\leq j\leq n_{i}}$ generated by the polynomials  $f_{11},\dots,f_{1N_1},\dots,f_{\ell 1},\dots,f_{\ell N_\ell}$ and  $\sum_{1\leq i\leq \ell}g_{i1},\dots,\sum_{1\leq i\leq \ell}g_{ir}$. Then we have $I(\cX_{Q}^{(\ell)})=\cI_K$.
 
We set $g_e=\sum_{1\leq i\leq \ell}g_{ie}, n= \sum_{1\leq i\leq \ell} n_i,\cR=r+\sum_{i=1}^{\ell} N_i \textnormal{ and } N=\sum_{i=1}^{\ell} N_i.$ By \ref{locpole}, Propositions \ref{gl-lc}, \ref{FRST} and the miracle flatness theorem, it is sufficient to show that 
\begin{equation}\label{ineq00}
\moi_{K, P}^{(\cR)}(\cI)\geq \ell D
\end{equation}
if $X_i=\AA_K^{n_i}$ for all $i$, 
and 
\begin{equation}\label{ineq01}
\moi_{K', P}^{(\cR)}(\cI)\geq N+\ell/(2(D^{R+1}-1))
\end{equation}
for a finite extension $K'$ of $K$ in the general case.

We put $$\sigma=\max_{K\subset K'}\moi_{K',P}^{(\cR)}(\cI),$$
where the maximum is taken over all finite extensions $K'$ of $K$. Note that the maximum is attained since there are only finitely many candidates for the value of $\moi_{K',P}^{(\cR)}(\cI)$ as seen in Corollary \ref{llct}. Then (\ref{ineq01}) is equivalent to
\begin{equation}\label{ineq02}
\sigma\geq N+\ell/(2(D^{R+1}-1)).
\end{equation}

\textbf{Step 2:} \textit{Simplify the exponential sums $E_{L,P,\cI}^{(\cR)}(m)$ for $m\geq 2$ if $L$ is a local field over $\cO_K$ of large enough residue field characteristic.}

Let $L\in\cL_{K,1}$, $\psi$ be an additive character of $L$ of conductor $m\geq 2$, we have 
\begin{align*}
E_{L,P,\cI}^{(\cR)}(m)&=\int_{\varpi_L\cO_L^{n}\times (\cO_L^{\cR}\setminus \varpi_L\cO_L^{\cR})}\psi\bigg(\sum_{1\leq e\leq r}y_eg_{e}+\sum_{i=1}^\ell \sum_{1\leq j\leq N_i} y_{ij}f_{ij}\bigg)\mu_{L^{n+\cR}}\\
&=\underbrace{\int_{\varpi_L\cO_L^{n}\times \cO_L^{N}\times (\cO_L^r\setminus \varpi_L\cO_L^{r})}\psi\bigg(\sum_{1\leq e\leq r}y_eg_{e}+\sum_{i=1}^\ell \sum_{1\leq j\leq N_i} y_{ij}f_{ij}\bigg)\mu_{L^{n+\cR}}}_{I_1(L,m)}\\
&\hspace{0.3cm}+\underbrace{\int_{\varpi_L\cO_L^{n}\times (\cO_L^N\setminus \varpi_L\cO_L^{N})\times \varpi_L\cO_L^{r}}\psi\bigg(\sum_{1\leq e\leq r}y_eg_{e}+\sum_{i=1}^\ell \sum_{1\leq j\leq N_i} y_{ij}f_{ij}\bigg)\mu_{L^{n+\cR}}}_{I_2(L,m)}.
\end{align*}
Here, we use the variables  $(y_1,\dots,y_r)\in L^r$, $(x_{ij})_{1\leq i\leq \ell, 1\leq j\leq n_i}\in L^n$ and $(y_{ij})_{1\leq i\leq\ell, 1\leq j\leq N_i}\in L^{N}$. We will show that $I_2(L,m)=0$ for all $m\geq 2$ if $L$ is of large enough residue field characteristic. Let $\tilde{\cX}^{(\ell)}\subset\AA_{\cO_K}^n$ be the $\cO_K$-scheme associated to the ideal generated by $f_{ij}, 1\leq i\leq \ell, 1\leq j\leq N_i$. Let $\cY^{(\ell)}\subset\AA_{\cO_K}^n$ be the $\cO_K$-scheme associated to the ideal generated by $g_1,\dots,g_r$. We set $\cV^{(\ell)}=\cY^{(\ell)}\cap\tilde{\cX}^{(\ell)}$ and recall the maps $\pi_{ij}:\tilde{\cX}^{(\ell)}(\cO_L/(\varpi_L^i))\to \tilde{\cX}^{(\ell)}(\cO_L/(\varpi_L^j))$ for $i\geq j\geq 1$ from Section \ref{nuloc}. By using orthogonality of characters as in Lemma-Definition \ref{moidef}, if $m\geq 2$ then one has
\begin{align*}
&I_2(L,m)\\=&q_L^{-mn}\biggl(\#\left(\pi_{m1}^{-1}(0)\cap \pi_{m(m-1)}^{-1}\left(\cV^{(\ell)}\left(\cO_L/\left(\varpi_L^{m-1}\right)\right)\right)\right)-q_L^{n-N}\#\left(\pi_{(m-1)1}^{-1}(0)\cap \cV^{(\ell)}(\cO_L/(\varpi_L^{m-1}))\right)\biggr)\\
=&q_L^{-mn}\sum_{y\in \cV^{(\ell)}(\cO_L/(\varpi_L^{m-1}))\cap\pi_{(m-1)1}^{-1}(0)}\left(\#\left(\pi_{m(m-1)}^{-1}(y)\right)-q_L^{n-N}\right).
\end{align*}
By Hensel's lemma and the fact that $\cX^{(\ell)}$ is smooth of dimension $n-N$, if $L$ is of large enough residue field characteristic, one has  
$$\#\left(\pi_{m(m-1)}^{-1}(y)\right)=q_L^{n-N}$$
for all $m\geq 2$ and all $y\in \tilde{\cX}^{(\ell)}(\cO_L/(\varpi_L^{m-1}))$. Thus, there is an integer $M$ such that $I_2(L,m)=0$ for all $m\geq 2$ if $L\in\tilde{\cL}_{K,M}$.

\textbf{Step 3:} \textit{Prove Inequality (\ref{ineq00}).}

Suppose that $X_i=\AA_K^{n_i}$ for all $i$, i.e.,  $N_i=0$ for all $i$. Since $\varphi_i(\AA_{\overline{K}}^{n_i})$ is not contained in any proper affine subspace of $\AA_{\overline{K}}^r$, we can use \cite[Theorem 6.1]{Cluckers-multi} to obtain a constant $a$ depending only on $L$, $(g_{ij})_{1\leq i\leq \ell, 1\leq j\leq r}$ and a positive constant $C=C(n_1,\dots,n_\ell, D)$ such that for all $(y_1,\dots,y_e)\in \cO_L^r\setminus \varpi_L\cO_L^{r}$, one has
\begin{align*}
\bigg|\int_{\varpi_L\cO_L^{n}}\psi\bigg(\sum_{1\leq e\leq r}y_eg_{e}\bigg)\mu_{L^{n}}\bigg|&=\bigg|\prod_{1\leq i\leq \ell}\int_{\varpi_L\cO_L^{n_i}}\psi\bigg(\sum_{1\leq e\leq r}y_eg_{ie}\bigg)\mu_{L^{n_i}}\bigg|\\
&\leq Cq_L^am^{n-1}q_L^{\frac{-m\ell}{D}}
\end{align*}
for all $m\geq 1$. By using fiber integration and \textbf{Step 2}, we have 
$$\left|E_{L,P,\cI}^{(\cR)}(m)\right|=\left|I_1(L,m)\right|\leq Cq_L^am^{n-1}q_L^{\frac{-m\ell}{D}}\int_{\cO_L^r\setminus \varpi_L\cO_L^{r}}\mu_{L^r}=(1-q_L^{-r})Cq_L^am^{n-1}q_L^{\frac{-m\ell}{D}}$$
 for all $m\geq 2$. This inequality together with Corollaries \ref{naig} and \ref{coboun} imply (\ref{ineq00}). 
 
 
\textbf{Step 4:} \textit{Construct an infinite sequence $(m_d)_{d\geq 1}$ of positive integers and give a lower bound of $|I_1(L,m_d)|$ in terms of $\sigma,m_d,q_L$ for good choices of $L$.}


 Let $K'$ be a finite extension of $K$ such that $\moi_{K',P}^{(\cR)}(\cI)=\sigma$. By the definitions of $\sigma$ and $\moi_{K',P}^{(\cR)}(\cI)$, we see that $\moi_{K'',P}^{(\cR)}(\cI)=\sigma$ for all finite extensions $K''$ of $K'$. Replacing $K$ by $K'$ if needed, one may suppose that $\moi_{K',P}^{(\cR)}(\cI)=\sigma$ for all finite extensions $K'$ of $K$.
By Corollary \ref{coboun},  there is a rational constant $c$, a strictly increasing sequence $(m_i)_{i\geq 1}$ of positive integers and a subset $\cA_i$ of $\cL_{K,1}$ for each $i\geq 1$ such that $\cA_i\cap \cL_{K',M}\neq \emptyset$ for all $i$, all finite extensions $K'$ of $K$ and all  $M\geq 1$, moreover one has  
\begin{equation}\label{lowine}
\left|I_1(L,m_i)\right|=\left|E_{L,P,\cI}^{(\cR)}(m_i)\right|\geq q_{L}^cq_{L}^{-m_i\sigma}
\end{equation}
for all $i\geq 1$ and all $L\in\cA_i$. 

\textbf{Step 5:} \textit{For each large enough integer $d$, rewrite $I_1(L,m_d)$ as an exponential sum over $k_L$ if  $L$ is a local field over $\cO_K$ of large enough residue field characteristic.}

\hspace{1.5cm}\textbf{Substep 5a:} \textit{For each integer $d$, show that $I_1(L,m_d)=I_1(k_L((t)),m_d)$ if  $L$ is a local field over $\cO_K$ of large enough residue field characteristic.}

Let us take $d\geq 1$. If $L\in\tilde{\cL}_{K,1}$, then we set $\tilde{L}=k_L((t))$.  By using the transfer principle for exponential sums in \cite[Proposition 5.4]{NguyenVeys}, there is an integer $M_d$ such that if $L\in\tilde{\cL}_{K,M_d}$ and $\psi$ is an additive character of $L$ of conductor $m_d$, then we have an additive character $\tilde{\psi}$ of $\tilde{L}$ of conductor $m_d$ satisfying
\begin{align}
I_{1}(L,m_d)=I_1(\tilde{L},m_d)=\int_{\varpi_{\tilde{L}}\cO_{\tilde{L}}^{n}\times \cO_{\tilde{L}}^{N}\times (\cO_{\tilde{L}}^r\setminus \varpi_L\cO_{\tilde{L}}^{r})}\tilde{\psi}\bigg(\sum_{1\leq e\leq r}y_eg_{e}+\sum_{i=1}^\ell \sum_{1\leq j\leq N_i} y_{ij}f_{ij}\bigg)\mu_{\tilde{L}^{n+\cR}}.\label{transf}
\end{align}

\hspace{1.5cm}\textbf{Substep 5b:} \textit{Fix a large enough integer $d$, for each $i$, describe the reduction modulo the $m_d^{\textnormal{th}}$ power of the maximal ideal of $\cO_{\AA_{K}^{n_i},P_i}$ of the polynomials $g_{ie}$ for all $e$.} 

For each $i$, since $X_{i}$ is smooth, by using a linear change of coordinates we can suppose that if $1\leq j\leq N_i$, then  $f_{ij}(x_{i1},\dots,x_{in_i})=x_{ij}+\tilde{f}_{ij}(x_{i1},\dots,x_{in_i})$  for some polynomial $\tilde{f}_{ij}$ of multiplicity at least $2$ at $P_i=0\in \AA_{\overline{K}}^{n_i}$. By our assumption, $\mathfrak{m}_i=(f_{i1},\dots,f_{iN_i},x_{i(N_i+1)},\dots,x_{in_i})\cO_{\AA_{\overline{K}}^{n_i},P_i}$ is the maximal ideal of the local ring $\cO_{\AA_{\overline{K}}^{n_i},P_i}$. Thus, for each $1\leq i\leq \ell$, each $1\leq e\leq r$ and each $m\geq 1$, there exists a polynomial $g_{iem}\in \overline{K}[x_{i(N_i+1)},\dots,x_{in_i}]$ of degree at most $m-1$ such that $g_{ie}-g_{iem}\in \mathfrak{m}_i^m+(f_{i1},\dots,f_{iN_i})\cO_{\AA_{\overline{K}}^{n_i},P_i}$. Let us fix  $d\geq 1$ such that $m_d>nD^{R+1}$. 
For each $i$ and each $e$,  we write $g_{ie}=g_{iem_d}+h_{iem_d}+\sum_{j=1}^{N_i}f_{ij}a_{ijm_d}$, where $h_{iem_d}\in \mathfrak{m}_i^{m_d}, a_{ijm_d}\in \cO_{\AA_{\overline{K}}^{n_i},P_i}.$
Replacing $K$ by its finite extension, one may suppose that $g_{iem_d},h_{iem_d},a_{ijm_d}$ are defined over $K$ for all $i,e,j$ and $\cA_{d}\cap \cL_{K,M}\neq \emptyset$ for all $M$. 

\hspace{1.5cm}\textbf{Substep 5c:} \textit{Let $d$ be as in} \textbf{Substep 5b}, \textit{introduce some notation related to the $m_d^{\textnormal{th}}$-jet scheme of affine spaces over $\cO_K$.}

As in Section \ref{jtrans} below,  we will write $$y_{e}(t)=\sum_{0\leq s}y_{es}t^s,y_{ij}(t)=\sum_{0\leq s}y_{ijs}t^s,x_{iu}(t)=\sum_{1\leq s}x_{ius}t^s, x_{iu}^{(v)}=(x_{ius})_{1\leq s\leq v},$$
$$ f_{ij}\left(x_{i1}(t),\dots,x_{in_i}(t)\right)=\sum_{1\leq s}f_{ijs}\left(x_{i1}^{(s)},\dots,x_{in_i}^{(s)}\right)t^s,$$
$$g_{ie}(x_{i1}(t),\dots,x_{in_i}(t))=\sum_{1\leq s}g_{ie}^{(s)}\left(x_{i1}^{(s)},\dots,x_{in_i}^{(s)}\right)t^s,$$
$$g_{iem_d}(x_{i(N_i+1)}(t),\dots,x_{in_i}(t))=\sum_{1\leq s}\tilde{g}_{ies}\left(x_{i(N_i+1)}^{(s)},\dots,x_{in_i}^{(s)}\right)t^s,$$
$$\tilde{y}^{(d1)}=(y_{ijs})_{1\leq i\leq \ell, 1\leq j\leq N_i, 0\leq s\leq m_d-2},\tilde{x}^{(d)}=\left(x_{i1}^{(m_d-1)},\dots,x_{in_i}^{(m_d-1)}\right)_{1\leq i\leq \ell},$$
$$\tilde{y}^{(d0)}=(y_{es})_{0\leq s\leq m_d-2, 1\leq e\leq r},\tilde{z}^{(d)}=\left(\tilde{y}^{(d0)},\tilde{y}^{(d1)},\tilde{x}^{(d)}\right),$$
$$ \hat{x}^{(d)}=\left(x_{i(N_i+1)}^{(m_d-1)},\dots,x_{in_i}^{(m_d-1)}\right)_{1\leq i\leq \ell}, \hat{z}^{(d)}=\left(\tilde{y}^{(d0)},\hat{x}^{(d)}\right),$$
$$F^{(d)}\big(\tilde{z}^{(d)}\big)=\sum_{1\leq i\leq \ell, 1\leq j\leq N_i, 0\leq s\leq m_d-2}y_{ijs}f_{ij(m_d-s-1)},$$
$$G^{(d)}\big(\tilde{z}^{(d)}\big)=\sum_{1\leq e\leq r, 0\leq s\leq m_d-2}y_{es}\bigg(\sum_{1\leq i\leq \ell}g_{ie}^{(m_d-s-1)}\bigg),$$
$$\hat{G}^{(d)}\big(\hat{z}^{(d)}\big)=\sum_{1\leq e\leq r, 0\leq s\leq m_d-2}y_{es}\bigg(\sum_{1\leq i\leq \ell}\tilde{g}_{ie(m_d-s-1)}\bigg).$$
For each $L\in\tilde{\cL}_{K,1}$, we set $$A_{1L}=k_L^{r(m_d-1)}\setminus\{ y_{e0}=0 \hspace{0.1cm}\forall e\}, A_{2L}=k_L^{N(m_d-1)},$$ and $$A_{3L}=k_L^{n(m_d-1)}, \tilde{A}_{3L}=k_L^{N(m_d-1)}, \hat{A}_{3L}=k_L^{(n-N)(m_{d-1})}.$$ 
Then we identify $A_{3L}$ with $\tilde{A}_{3L}\times\hat{A}_{3L}$ and consider the projection $\lambda_L:A_{1L}\times A_{2L}\times A_{3L}\to A_{1L}\times \hat{A}_{3L}$. Moreover, we denote by $Z$  the $\cO_K$-subscheme of $\AA_{\cO_K}^{n(m_d-1)}$ associated to the ideal generated by $f_{ijs}$ for  $1\leq i\leq \ell, 1\leq j\leq N_i, 1\leq s\leq m_d-1$. Then $F^{(d)}$ and $G^{(d)}$ define maps from $A_{1L}\times A_{2L}\times A_{3L}$ to $k_L$ in an obvious way. Similarly we can regard $\hat{G}^{(d)}$ as a map from  $A_{1L}\times\hat{A}_{3L}$ to $k_L$.

\hspace{1.5cm}\textbf{Substep 5d:} \textit{Let $d$ be as in} \textbf{Substep 5b}, \textit{use} \textbf{Substep 5b} \textit{and} \textbf{Substep 5c} \textit{to rewrite $I_1(k_L((t)),m_d)$ as an exponential sum over $k_L$ if $L$ is a local field over $\cO_K$ of large enough residue field characteristic.}

By enlarging $M_d$ if needed, we use \cite[Proposition 4.3]{NguyenVeys} to obtain a non-trivial character $\Psi_L$ of $k_L$ for each $L\in \tilde{\cL}_{K,M_d}$ such that 
$$I_{1}(\tilde{L},m_d)=q_L^{-m_d(n+\cR)+N+r}\sum_{\tilde{z}^{(d)}\in A_{1L}\times A_{2L}\times A_{3L}} \Psi_L\left(F^{(d)}\big(\tilde{z}^{(d)}\big)+G^{(d)}\big(\tilde{z}^{(d)}\big)\right).$$
Orthogonality of characters yields
$$\sum_{\tilde{z}^{(d)}\in A_{1L}\times A_{2L}\times \{\tilde{x}^{(d)}\}} \Psi_L\left(F^{(d)}\big(\tilde{z}^{(d)}\big)+G^{(d)}\big(\tilde{z}^{(d)}\big)\right)=0$$
for all $L\in \tilde{\cL}_{K,M_d}$ and all $\tilde{x}^{(d)}\in A_{3L}\setminus Z(k_L)$. Thus, we have 
\begin{align}I_{1}(\tilde{L},m_d)&=q_L^{-m_d(n+\cR)+N+r}\sum_{\tilde{z}^{(d)}\in A_{1L}\times A_{2L}\times Z(k_L)} \Psi_L\left(F^{(d)}\big(\tilde{z}^{(d)}\big)+G^{(d)}\big(\tilde{z}^{(d)}\big)\right)\nonumber\\
&=q_L^{-m_d(n+\cR)+N+r}\sum_{\tilde{z}^{(d)}\in A_{1L}\times A_{2L}\times Z(k_L)} \Psi_L\left(G^{(d)}\big(\tilde{z}^{(d)}\big)\right)\label{separ}
\end{align}
for all $L\in \tilde{\cL}_{K,M_d}$.

\hspace{1.5cm}\textbf{Substep 5e:} \textit{Simplify the exponential sums over finite fields obtained in} \textbf{Substep 5d} \textit{to complete} \textbf{Step 5}.

Since $g_{ie}=g_{iem_d}+h_{iem_d}+\sum_{j=1}^{N_i}f_{ij}a_{ijm_d}$ and $h_{iem_d}\in\mathfrak{m}_i^{m_d}$, we can  enlarge $M_d$ to have that $(\tilde{g}_{ies}-g_{ie}^{(s)} \mod (\varpi_L))$ belongs to the ideal of $k_L[x_{i1}^{(m_d-1)},\dots,x_{in_i}^{(m_d-1)}]$ generated by $(f_{iju} \mod (\varpi_L))$ for $1\leq i\leq \ell, 1\leq j\leq N_i, 1\leq u\leq m_d-1$, for all $i,e,s$ and all $L\in \tilde{\cL}_{K,M_d}$. Thus, $\tilde{g}_{ies}|_{Z(k_L)}=g_{ie}^{(s)}|_{Z(k_L)}$ for all $i,e,s$ and all $L\in \tilde{\cL}_{K,M_d}$. Therefore, if $L\in\tilde{\cL}_{K,M_d}$, $\tilde{z}^{(d)}= \left(\tilde{y}^{(d0)},\tilde{y}^{(d1)},\tilde{x}^{(d)}\right)\in A_{1L}\times A_{2L}\times Z(k_L)$ and $\hat{z}^{(d)}=\lambda_L\left(\tilde{z}^{(d)}\right)$ is the image of $\tilde{z}^{(d)}$ under the projection $\lambda_L: A_{1L}\times A_{2L}\times A_{3L}\to A_{1L}\times \hat{A}_{3L}$ in the above,  then it is clear to see that 
\begin{equation}\label{desc}G^{(d)}(\tilde{z}^{(d)})=\hat{G}^{(d)}(\hat{z}^{(d)}).
\end{equation} On the other hand, since $f_{ij}(x_{i1},\dots,x_{in_i})=x_i+\tilde{f}_{ij}(x_{i1},\dots,x_{in_i})$ and $\ord_{P_i}\tilde{f}_{ij}\geq 2$ for all $i,j$, one can show that the projection 
\begin{align*}
\theta_L:Z(k_L)&\to \hat{A}_{3L}=k_L^{(n-N)(m_d-1)}\\
(x_{ijs})_{1\leq i\leq \ell, 1\leq j\leq n_i, 1\leq s\leq m_d-1} &\mapsto (x_{ijs})_{1\leq i\leq \ell, N_i+1\leq j\leq n_i, 1\leq s\leq m_d-1}
\end{align*}
is a bijection for all $L\in\tilde{\cL}_{K,M_d}$.   
This together with (\ref{transf}), (\ref{separ}) and (\ref{desc}) imply
\begin{align}
\hspace{-3cm}I_{1}(L,m_d)&=q_L^{-m_d(n+\cR)+N+r}q_L^{N(m_d-1)}\sum_{\hat{z}^{(d)}\in A_{1L}\times \hat{A}_{3L}} \Psi_L\left(\hat{G}^{(d)}\big(\hat{z}^{(d)}\big)\right)\nonumber\\
&=q_L^{-m_d(n+r)+r}\sum_{\tilde{y}^{(d0)}\in A_{1L}}\prod_{i=1}^{\ell}\sum_{k_L^{(n_i-N_i)(m_d-1)}}\Psi_L\left(\hat{G}_i^{(d)}\big(\tilde{y}^{(d0)}, \hat{x}_i^{(d)}\big)\right)\label{reduct111}
\end{align}
for all $L\in \tilde{\cL}_{K,M_d}$, where $\hat{x}_i^{(d)}=\left(x_{i(N_i+1)}^{(m_d-1)},\dots,x_{in_i}^{(m_d-1)}\right)$ and $$\hat{G}_i^{(d)}\big(\tilde{y}^{(d0)}, \hat{x}_i^{(d)}\big)=\sum_{1\leq e\leq r, 0\leq s\leq m_d-2}y_{es}\tilde{g}_{ie(m_d-s-1)}$$
for $1\leq i\leq \ell$.

\textbf{Step 6:} \textit{For each large enough integer $d$, give an upper bound of $\left|I_1(L,m_d)\right|$ uniformly in local fields $L$ over $\cO_K$ of large enough residue field characteristic.}

We consider the weight $w$ on the variables $x_{ijs}$ given by $w(x_{ijs})=s$ for $1\leq i\leq \ell, N_i+1\leq j\leq n_i$ and $1\leq s\leq m_d-1$. Since $\varphi_i(X_{ij})$ is not contained in any proper affine subspace of $\AA_{\overline{K}}^{r}$ for all $i$ and  all irreducible components $X_{ij}$ of $X_i\otimes \overline{K}$, we see that  $\sum_{1\leq e\leq r} a_eg_{ie}$ is a non-constant polynomial of degree at most $D$ for all $1\leq i\leq \ell$ and all $(a_1,\dots,a_r)\in \overline{K}^r\setminus \{0\}$. Lemma \ref{multi} implies that  $\ord_{P_i}\big(\sum_{1\leq e\leq r} a_eg_{ie}|_{X_i}\big)\leq D^{R+1}$ for all $i$ and all $(a_1,\dots,a_r)\in \overline{K}^r\setminus \{0\}$. But we supposed above that $\ord_{P_i}\big(g_{ie}|_{X_i}-g_{iem_d}\big)\geq m_d>nD^{R+1}$, thus $\ord_{P_i}\big(\sum_{1\leq e\leq r} a_eg_{iem_d}\big)\leq D^{R+1}$ for all $1\leq i\leq \ell$ and all $(a_1,\dots,a_r)\in \overline{K}^r\setminus \{0\}$.  By enlarging $M_d$ and using Lemma \ref{boundimen}, if $L\in \tilde{\cL}_{K,M_d}$, then for each given point $\tilde{y}^{(d0)}=(y_{es})_{0\leq s\leq m_d-2, 1\leq e\leq r}\in A_{1L}$ and each $1\leq i\leq \ell$, the $w$-weighted homogeneous polynomial $\sum_{1\leq e\leq r}y_{e0}\tilde{g}_{ie(m_d-1)}\in k_L[\hat{x}_i^{(d)}]$ has singular locus of dimension at most $$m_d(n_i-N_i)-(m_d-1)/(D^{R+1}-1)<(m_d-1)(n_i-N_i)$$ since $m_d>nD^{R+1}$. Note that for each given point $\tilde{y}^{(d0)}\in A_{1L}$, the $w$-homogeneous part of highest $w$-degree of $\hat{G}_i^{(d)}\big(\tilde{y}^{(d0)}, \hat{x}_i^{(d)}\big)\in k_L[\hat{x}_i^{(d)}]$ is $\sum_{1\leq e\leq r}y_{e0}\tilde{g}_{ie(m_d-1)}$. So we can use Corollary \ref{quasisums} for the polynomial $\hat{G}_i^{(d)}\big(\tilde{y}^{(d0)}, \hat{x}_i^{(d)}\big)$ of $w$-degree $m_d-1$ for each $\tilde{y}^{(d0)}\in A_{1L}$ to obtain  a constant $C$ depending only on $n_1,\dots,n_\ell, m_d, w$ such that
\begin{align}
\sum_{\hat{x}_i^{(d)}\in k_L^{(n_i-N_i)(m_d-1)}}\Psi_L\left(\hat{G}_i^{(d)}\big(\tilde{y}^{(d0)}, \hat{x}_i^{(d)}\big)\right)
\leq &Cq_L^{\frac{(n_i-N_i)(2m_d-1)-(m_d-1)/(D^{R+1}-1)}{2}}\nonumber\\ \leq&Cq_L^{m_d(n_i-N_i)}q_L^{\frac{N_i-n_i+1}{2}}q_L^{\frac{-m_d}{2(D^{R+1}-1)}}\nonumber\\ \leq& Cq_L^{m_d(n_i-N_i)}q_L^{\frac{-m_d}{2(D^{R+1}-1)}}\label{reduct001}
\end{align}
for all $1\leq i\leq \ell$, all $L\in\tilde{\cL}_{K,M_d}$ and all $\tilde{y}^{(d0)}\in A_{1L}$. Here, we need to remark that $N_i<n_i$ for all $1\leq i\leq \ell$ as mentioned in (\ref{nNi}). By combining (\ref{reduct111}) and (\ref{reduct001}), we have
\begin{align}
\left|I_{1}(L,m_d)\right|
&\leq q_L^{-m_d(n+r)+r} \left(q_L^{r(m_d-1)}-q_L^{r(m_d-2)}\right)\prod_{1\leq i\leq \ell}Cq_L^{m_d(n_i-N_i)}q_L^{\frac{-m_d}{2(D^{R+1}-1)}}\nonumber\\
&\leq C^{\ell}q_L^{-m_d(N+\ell/(2(D^{R+1}-1)))}\label{reduct000}
\end{align}
for all $L\in\cL_{K,M_d}$. 

\textbf{Step 7:} \textit{Complete the proof.}

Suppose that  (\ref{ineq02}) does not hold then $\sigma-N-\ell/(2(D^{R+1}-1))<0$. By taking a large enough integer $d$, we can suppose that 
\begin{equation}\label{choc}
c>m_d\left(\sigma-N-\frac{\ell}{2(D^{R+1}-1)}\right).
\end{equation}

For all $L\in\cA_{d}\cap \cL_{K,M_d}$, it follows from (\ref{lowine}) and (\ref{reduct000}) that
$$q_{L}^cq_{L}^{-m_d\sigma}\leq \left|I_{1}(L,m_d)\right|\leq C^{\ell}q_{L}^{-m_d(N+\ell/(2(D^{R+1}-1)))}.$$
This contradicts with (\ref{choc}) and  the condition that $\cA_{d}\cap \cL_{K,M}\neq\emptyset$ for all $M$. Therefore, (\ref{ineq02}) holds as desired.

\end{proof}
\begin{lem}\label{multi}Let $X$ be an irreducible closed subvariety of $\AA_\CC^n$ and $P$ be a smooth point of $X$. Let $g\in\CC[x_1,\dots,x_n]$ be a polynomial of degree $d$  such that $g(P)=0$ and $g|_X$ is not constant, then the multiplicity $\ord_{P}(g|_X)$ of $g|_X$ at $P$ is at most $d\deg(X)$.
\end{lem}
\begin{proof}Let $r=\dim(X)$. By using a linear change of coordinates, we can suppose that $P=0\in\AA_\CC^n$. Since $X$ is smooth at $P$, there are polynomials $f_{r+1},\dots,f_n$ such that  $I=I(X)\cO_{\AA_\CC^n,P}=(f_{r+1},\dots,f_n)\cO_{\AA_\CC^n,P}$. By Bertini's theorem  (see \cite[Chapter II, Theorem 8.18]{Harts}) and a suitable change of coordinates, we can suppose that if $S$ is a subset of $\{1,...,r\}$ then $P$ is a smooth point of the scheme associated to the ideal $(f_{r+1},\dots,f_{n})+\sum_{i\in S}\CC[x_1,\dots,x_n]x_i$. By reordering the coordinates, we can suppose that the ideal $$\sqrt{(g,x_2,\dots,x_r,f_{r+1},\dots,f_n)\cO_{\AA_\CC^n,P}}=(x_1,\dots,x_r,f_{r+1},\dots,f_n)\cO_{\AA_\CC^n,P}$$ is the maximal ideal $\mathfrak{m}_P$ of $\cO_{\AA_\CC^n,P}$. 

Let $Y$ be the scheme associated to the ideal $(x_2,\dots,x_r,f_{r+1},\dots,f_n)$. By our assumption, $P$ is a smooth point of $Y$. Thus, there is a unique  irreducible component $W$ of $Y$ containing $P$. We have $\deg(W)\leq \deg(X)$ by Bézout's theorem (see \cite[Theorem 18.4]{Harris}). The multiplicity of $g|_X$ at $P$ is the integer $m_P$ such that $g\in \mathfrak{m}_P^{m_P}+I$ but $g\notin \mathfrak{m}_P^{m_P+1}+I$. Similarly, the multiplicity of $g|_{W}$ at $P$ is the integer $m_1$ such that $g\in \mathfrak{m}_P^{m_1}+J$ but $g\notin \mathfrak{m}_P^{m_1+1}+J$, where $J=I(W)\cO_{\AA_\CC^n,P}=(x_2,\dots,x_r,f_{r+1},\dots,f_n)\cO_{\AA_\CC^n,P}$. Since $I\subset J$, one has $m_P\leq m_1$, so we only need to prove that $m_1\leq d\deg(W)\leq d\deg(X)$.

Let $Z(g)$ be the closed subscheme of $\AA_\CC^n$ defined by $g$. By Bézout's theorem (see \cite[Theorem 18.4]{Harris}), it suffices to show that $m_1\leq m_P(W,Z(g))$, where $m_P(W,Z(g))$ is the intersection multiplicity of $W$ and $Z(g)$ along $P$. To compute $m_P(W,Z(g))$, we may follow \cite[Chapter 1, Section 7]{Harts}. However, in an easier way, we can use the discussion after \cite[Theorem 18.4]{Harris} to have
$$m_P(W,Z(g))=\dim_{\CC}\cO_{\AA_\CC^n,P}/(g,x_2,\dots,x_r,f_{r+1},\dots,f_n)\cO_{\AA_\CC^n,P}.$$
Since $J$ cannot contain a power of $\mathfrak{m}_P$, by using Nakayama's lemma, we obtain a strictly decreasing sequence of ideals 
$$\cO_{\AA_\CC^n,P}\supset\mathfrak{m}_0\supset \mathfrak{m}_P^2+ J\supset...\supset \mathfrak{m}_P^{m_1}+ J$$ containing $(g,x_2,\dots,x_r,f_{r+1},\dots,f_n)\cO_{\AA_\CC^n,P}$. Thus, $m_1\leq m_P(W,Z(g))$. Therefore, our claim follows from the above discussion.

\end{proof}
\begin{lem}\label{boundimen}Let $K$ be a number field and $M$ be a non-zero integer. Let $f_1,\dots,f_r$ be non-constant polynomials in $\cO_K[1/M][x_1,\dots,x_n]$ such that $f_1(0)=...=f_r(0)=0$. Let $d\geq 1$. Suppose that  $\ord_0(a_1f_1+...+a_rf_r)\leq d$ for all $(a_1,\dots,a_r)\in K^n\setminus \{0\}$. We write 
$$x_i(t)=\sum_{\ell\geq 1}x_{i\ell}t^\ell, x^{(m)}=(x_{ij})_{1\leq i\leq n,1\leq j\leq m}$$ 
and $$f_j(x_1(t),\dots,x_n(t))=\sum_{m\geq 1} f_{jm}(x^{(m)})t^m \textnormal{ for } 1\leq j\leq r$$
as in Section \ref{jtrans} below. For each $m\geq 1$, there exists an integer $M_m>M$ such that for all $L\in\tilde{\cL}_{K,M_m}$ and all $(a_1,\dots,a_r)\in k_L^r\setminus \{0\}$, the singular locus of $\sum_{1\leq i\leq r}a_if_{i(m-1)}$ in $\AA_{\overline{k}_L}^{n(m-1)}$ is of dimension at most $mn-(m-1)/(d-1)$.
\end{lem}
\begin{proof}
By logical compactness (see \cite[Corollary 2.2.10]{Marker}), it suffices to show that for all $(a_1,\dots,a_r)\in \CC^r\setminus \{0\}$, the singular locus of $\sum_{1\leq i\leq r}a_if_{i(m-1)}$ in $\AA_\CC^{n(m-1)}$ is of dimension at most $mn-(m-1)/(d-1)$. This fact follows from the proof of \cite[Lemma 3.4, Corollary 3.10]{Nguyennsd} and our assumption that $\ord_0(a_1f_1+...+a_rf_r)\leq d$ for all $(a_1,\dots,a_r)\in K^n\setminus \{0\}$.
\end{proof}

\begin{proof}[Proof of \ref{adduni}] 
 For each $1\leq i\leq \ell$, since $X_{i}\otimes\QQ$ is geometrically irreducible, we can choose an open affine subset $Y_i$ of $X_i$ such that $Y_i$ is of finite type over $\ZZ$ and $Y_{i}\otimes\QQ$ is also geometrically irreducible. It suffices to prove the first claim for $Y_i$ instead of $X_i$, thus we can suppose that $X_i$ is a closed subscheme  of $\AA_\ZZ^{n_i}$ for some $n_i$. Let $Z_i$ be the reduced scheme associated to $X_{i}\otimes\QQ$ and $\cI_i$ be an ideal of $\ZZ[x_1,\dots,x_n]$ such that the ideal of $Z_i$ is $\cI_{i}\otimes\QQ$. Then it is easy to find a non-zero integer $M$ such that $\spec(\ZZ[1/M][x_1,\dots,x_n]/\cI_i)\subset X_i\otimes \ZZ[1/M]$ for all $i$. So it is also sufficient to prove our first claim for $\spec(\ZZ[1/M][x_1,\dots,x_n]/\cI_i)$ instead of $X_i$. Thus, we can suppose that $Y_i=X_{i}\otimes\QQ$ is reduced for all $i$. By \cite[Chapter 3, Proposition 2.7]{LiuQ}, $Y_{i}\otimes\overline{\QQ}$ is also reduced.  Thus, if $\tilde{X}_i$ is the smooth locus of $Y_{i}\otimes\QQ$, then $\tilde{X}_i\neq \emptyset$. Moreover, $\tilde{X}_{i}\otimes\overline{\QQ}$ is the smooth locus of  $Y_{i}\otimes\overline{\QQ}$ and $\tilde{X}_i$ can be defined over $\ZZ$. Let $W_i$ be a $\ZZ$-scheme of finite type such that $W_{i}\otimes\QQ=\tilde{X}_i$ then for a suitable choice of a non-zero integer $M$ we have $W_i\otimes \ZZ[1/M]\subset X_i\otimes\ZZ[1/M]$. Repeating the above argument,  we can suppose that $Y_i=X_{i}\otimes\QQ$ is smooth for all $i$. Now, we can use \ref{FRS} to see that if $\ell>2r$ then $\Phi_\QQ^{(\ell)}:X_\QQ^{(\ell)}=(X_{1}\otimes\QQ)\times...\times (X_{\ell}\otimes\QQ)\to \AA_\QQ^r$ is $FGI$. It follows from the argument in \cite[Page 224]{Harris}) that the geometrical irreducibility of varieties of bounded complexity can be expressed by a formula in the language of rings.  Since the logical compactness  in \cite[Corollary 2.2.10]{Marker}, there exists a constant $M$ such that for all fields $k$ of characteristic at least $M$, every fiber of $\Phi_k^{(\ell)}:X_k^{(\ell)}=(X_{1}\otimes k)\times...\times (X_{\ell}\otimes k)\to \AA_k^r$ is  geometrically irreducible.  By the Lang-Weil estimate in Proposition \ref{Lang-Weil}, we can enlarge $M$ if needed to have that for all $p>M$ and all $y\in \FF_p^r$, the fiber $(\Phi_{\FF_p}^{(\ell)})^{-1}(y)$ has at least one smooth point defined over $\FF_p$. Thus, our first claim follows from Hensel's lemma.

In order to prove the second claim, we just need to use \ref{FRS} together with \cite[Theorem 9.3 (ii)]{G-H3}. In other words, the second claim  is equivalent to the fact that $\Phi_\QQ^{(\ell)}$ is $FRS$ and $FGI$, so it follows immediately from the condition $\ell>2r(D^{R+1}-1)$ and \ref{FRS}.

\end{proof}

\section{Bounds of exponential sums}\label{exsums}
In this section, we will prove some results on exponential sums over finite fields. In addition, we will recall   the transfer principle for bounds of exponential sums modulo $p^m$ in \cite{NguyenVeys}. After that we will prove Proposition \ref{m=1}, \ref{wnsd} and Corollary \ref{countinguni}. 
\subsection{Exponential sums over finite fields}\label{modp}
 If $\mathsf{F}$ is a field and $f(x_1,\dots,x_n)$ is a polynomial in $\mathsf{F}[x_1,\dots,x_n]$, we denote by $C_f$ the closed subscheme of $\AA_\mathsf{F}^n$ associated to the Jacobian ideal $\cJ_f$ of $f$ generated by polynomials $\frac{\partial f}{\partial x_i}, 1\leq i\leq n$.  We denote by $s(f)$ the dimension of $C_{f}$ with the convention that $s(f)=-1$ if $C_f=\emptyset$ as in Section \ref{intro}. Suppose that $f=\sum_{I\in \NN^n}a_Ix^I$ where the coefficients $a_I$ are contained in a commutative ring $R$ and  $x^I=x_1^{i_1}\cdot...\cdot x_n^{i_n}$ if $I=(i_1,\dots,i_n)$. We denote by $\Supp(f)=\{I\in \NN^n|a_I\neq 0\}$. Let $w=(w_1,\dots,w_n)\in\NN_{\geq 1}^n$, we set $\left|w\right|=\sum_{1\leq j\leq n}w_j$. We recall that a polynomial $f$ is of $w$-degree $d$ if $$d=d_w(f):=\max_{I=(i_1,\dots,i_n)\in\Supp(f)}(w,I),$$ where if $I=(i_1,\dots,i_n)$ then $(w,I)=\sum_{j}w_ji_j$ is the usual scalar product in $\RR^n$. If $(w,I)=d$ for all $I\in \Supp(f)$, then $f$ is a $w$-weighted homogeneous polynomial of $w$-degree $d$. Given a polynomial $f$ of $w$-degree $d$, we can write $f=\sum_{0\leq j\leq d}f_{j,w}$ where $f_{j,w}$ is the $w$-weighted homogeneous part of $w$-degree $j$ of $f$. We will set $\tilde{f}_w=f_{d,w}$. If $w=(1,1,\dots,1)$, then $\tilde{f}_w$ will be the highest degree homogeneous part $\tilde{f}$ of $f$ as usual.
 We recall the following lemma in \cite{CDenSperlocal}.
\begin{lem}[\cite{CDenSperlocal}, Corollary 7.2]\label{torus}
Let $w=(w_1,...,w_n)\in \NN_{\geq 1}^n$ and $f(x_1,...,x_n)$ be a polynomial over a field $k$. Let $1\leq i\leq n$, we consider the polynomial $$f_1(x_1,...,x_n,y)=f(x_1,...,x_{i-1},x_iy,x_{i+1},...,x_n).$$ We say that $f_1$ is obtained from $f$ by a basic torus transformation. Suppose that $f$ is $w$-weighted  homogeneous of $w$-degree $d$ and $d\neq 0 \in k$. Let $C_f\subset \AA_k^n$ and $C_{f_1}\subset \AA_k^{n+1}$ be as in the above. Then $s(f_1)=\dim(C_{f_1})\leq s(f)+1=\dim(C_f)+1$. Moreover, if we set $$\hat{f}(x_{11},...,x_{1w_1},...,x_{n1},...,x_{nw_n})=f(x_{11}\cdots x_{1w_1},...,x_{n1}\cdots x_{nw_n}),$$
then $\hat{f}$ is a homogeneous polynomial of degree $d$ obtained from $f$ by doing $|w|-n$ basic torus transformations and we have $s(\hat{f})\leq s(f)+|w|-n$. 
\end{lem}
\begin{lem}\label{homosums} Let $k$ be a finite field and $f$ be a homogeneous polynomial of degree $d>1$ in $k[x_1,\dots,x_n]$. Suppose that $d$ is invertible in $k$, then there exists a constant $C$ depending only on $n,d$ such that for all polynomials $g\in k[x_1,\dots,x_n]$ of degree at most $d-1$, all subsets $J_1, J_2$ of $\{1,\dots,n\}$ with $J_1\cap J_2=\emptyset$ and all non-trivial characters $\Psi:k\to\CC^*$, we have
\begin{equation}\label{finitefield}
E_{J_1,J_2,k,\Psi}(f,g)=\bigg|\sum_{(x_1,\dots,x_n)\in k^n, x_i=0\hspace{0.1cm}\forall i\in J_1, x_i\neq 0\hspace{0.1cm}\forall i\in J_2} \Psi(f(x)+g(x))\bigg|\leq C\#k^{\frac{n+s(f)}{2}}.
\end{equation}
\end{lem}
\begin{proof}For each subset $J$ of $\{1,...,n\}$, we set $H_J=\{(x_1,...,x_n)\in k^n\mid x_i=0 \textnormal{ for all } i\in J\}$. Then for each $J_1,J_2\subset \{1,...,n\}$, the set $\{(x_1,...,x_n)\in k^n\mid x_i=0 \textnormal{ for all } i\in J_1, x_i\neq 0 \textnormal{ for all } i\in J_2 \}$ is a finite Boolean combination of the sets $H_J$ for $J\subset \{1,...,n\}$. Thus, it suffices to prove the lemma in the case that $J_2=\emptyset$. If $n\leq s(f)+2\left|J_1\right|$, then (\ref{finitefield}) holds trivially for $C=1$. So we can suppose that $2|J_1|+s(f)<n$ and $J_2=\emptyset$. Now, we fix a set $J_1\subset \{1,\dots,n\}$ with $2\left|J_1\right|+s(f)<n$ and put $f_{J_1}=f|_{\{x_i=0, i\in J_1\}}$. Then either $f_{J_1}$ is a homogeneous polynomial of degree $d$ in $n-|J_1|$ variables or $f_{J_1}=0$. As in the proof of \cite[Theorem 7.4]{CDenSperlocal} we have $s(f_{J_1})\leq s(f)+\left|J_1\right|<n-|J_1|$, thus we must have $f_{J_1}\neq 0$. By \cite[Theorem 4]{Katz}, there exists a constant $C_{J_1}$ depending only on $n,d$ such that for all polynomials $g\in k[x_1,\dots,x_n]$ of degree at most $d-1$ and all non-trivial characters $\Psi:k\to\CC^*$, we have 
$$\bigg|\sum_{(x_1,\dots,x_n)\in k^n, x_i=0\hspace{0.1cm}\forall i\in J_1} \Psi(f(x)+g(x))\bigg|\leq C_{J_1}\#k^{\frac{(n-\left|J_1\right|)+s(f_{J_1})}{2}}.$$
Since $s(f_{J_1})\leq s(f)+\left|J_1\right|$, we have 
$$E_{J_1,\emptyset,k,\Psi}(f,g)\leq C_{J_1}\#k^{\frac{n+s(f)}{2}}$$
 for all polynomials $g\in k[x_1,\dots,x_n]$ of degree at most $d-1$ and all non-trivial characters $\Psi:k\to\CC^*$. We can finish this proof by setting $C=1+\max_{2\left|J_1\right|+s(f)<n}C_{J_1}$.
\end{proof}
\begin{cor}\label{quasisums}
Let $k$ be a finite field and $f$ be a $w$-weighted homogeneous polynomial of $w$-degree $d>1$ in $k[x_1,\dots,x_n]$, where $w=(w_1,\dots,w_n)$ is a tuple in $\NN_{\geq 1}^n$. Suppose that $d$ is invertible in $k$, then there exists a constant $C$ depending only on $n,d,w$ such that for all polynomials $g\in k[x_1,\dots,x_n]$ with $d_w(g)\leq d-1$ and all non-trivial characters $\Psi:k\to\CC^*$, we have
\begin{equation}\label{finitefield1}
\bigg|\sum_{x\in k^n} \Psi\left(f(x)+g(x)\right)\bigg|\leq C\#k^{\frac{n+s(f)}{2}}.
\end{equation}
\end{cor}
\begin{proof}We use the idea from the proof of  \cite[Theorem 7.4]{CDenSperlocal}. Let us consider a basic torus transformation in Lemma \ref{torus}. We set $f_1(x_1,\dots,x_n,x_{n+1})=f(x_1x_{n+1},x_2,\dots,x_n)$ and $g_1(x_1,\dots,x_n,x_{n+1})=g(x_1x_{n+1},x_2,\dots,x_n)$ then we have 
\begin{equation}\label{toeq}
\frac{1}{\#k-1}\sum_{(x,x_{n+1})\in k^n\times k^*} \Psi\left(f_1(x,x_{n+1})+g_1(x,x_{n+1})\right)=\sum_{x\in k^n} \Psi\left(f(x)+g(x)\right).
\end{equation}
 Let $\hat{f}, \hat{g}$ be polynomials in $|w|$ variables as in Lemma \ref{torus}. Then $\hat{f}$ is homogeneous of degree $d$, $\deg(\hat{g})\leq d-1$ and $s(\hat{f})\leq s(f)+\left|w\right|-n$.  We can use (\ref{toeq}) and $|w|-n$ basic torus transformations as in Lemma \ref{torus} to have
\begin{equation}\label{finitesum1}
\frac{1}{(\#k-1)^{\left|w\right|-n}}\sum_{y\in k^n\times (k^*)^{\left|w\right|-n}} \Psi\left(\hat{f}(y)+\hat{g}(y)\right)=\sum_{x\in k^n}\Psi\left(f(x)+g(x)\right).
\end{equation}
 By Lemma \ref{homosums}, there exists a constant $c$ depending only on $n,w,d$ such that for all polynomials $g\in k[x_1,\dots,x_n]$ with $d_w(g)\leq d-1$ and all non-trivial characters $\Psi:k\to\CC^*$, we have
\begin{align}\label{finitesum2}
\bigg|\sum_{y\in k^n\times (k^*)^{\left|w\right|-n}} \Psi\left(\hat{f}(y)+\hat{g}(y)\right)\bigg|\leq c\#k^{\frac{\left|w\right|+s(\hat{f})}{2}}
\leq c\#k^{\frac{-n+2\left|w\right|+s(f)}{2}}.
\end{align}
By combining (\ref{finitesum1}) and  (\ref{finitesum2}), we obtain (\ref{finitefield1})  for $C=2^{\left|w\right|-n}c$.
\end{proof}

\subsection{Transfer principle for exponential sums modulo $p^m$}\label{jtrans}
Let $K$ be a number field. Let $f\in\cO_K[x_1,\dots,x_n]$ be a non-constant polynomial and $Z\subset\AA_{\cO_K}^n$ be an $\cO_K$-scheme of finite type. Let $x=(x_1(t),\dots,x_n(t))\in \cO_K[[t]]^n$, then we write 
$$x_i(t)=\sum_{j\geq 0}x_{ij}t^j.$$
For each integer $m\geq 0$, we set $x^{(m)}=(x_{ij})_{0\leq j\leq m, 1\leq i\leq n}$ and view it as a closed point of $\AA^{(m+1)n}$. By expanding $f(x)$ as a power series in $t$, we can write
$$f(x)=\sum_{i\geq 0}f_i\left(x^{(i)}\right)t^i\in\cO_K[[t]]$$
for polynomials $f_i\in\cO_K[x^{(i)}]$ depending only on $f$ for $i\geq 0$. We call $f_i$ the $i^{\text{th}}$ jet polynomial of $f$. Let $k$ be a field over $\cO_K$, i.e., $k$ is equipped with a structure of  $\cO_K$-algebra $\varphi:\cO_K\to k$, we still denote by $f, f_{i}$ the polynomials  $\varphi(f), \varphi(f_i)$ respectively. We identify $(Z(k)+tk[[t]]^n)/(t^{m+1})$ with the set of $k$-points of the $k$-scheme $Z_k^{(m)}:=(Z\otimes \AA_{\cO_K}^{mn})\otimes \spec(k)$ and view $f_{m}$ as a regular function on $Z_k^{(m)}$. Now, we recall an important result from \cite[Propositions 4.3 and 5.4, Corollary 5.5]{NguyenVeys}.

\begin{prop}\label{transfer} Let $f\in\cO_K[x_1,\dots,x_n]$ be a non-constant polynomial, $Z\subset\AA_{\cO_K}^n$ be an $\cO_K$-scheme of finite type and $\sigma$ be a  positive real number. Suppose that $f(Z(\CC))$ contains at most one critical value of $f$. Moreover, suppose that for each $m> 1$, there is an integer $M_m$ and a positive constant $C_m$ such that for all finite fields $k$ over $\cO_K$ with $\Ch(k)>M_m$, all $a\in k^*$ and all non-trivial characters $\Psi:k\to \CC^{*}$ of $k$, we have 
$$\bigg|\#k^{-mn}\sum_{x^{(m-1)}\in Z_{k}^{(m-1)}(k)}\Psi\left(af_{m-1}\left(x^{(m-1)}\right)\right)\bigg|\leq C_m\#k^{-m\sigma}.$$ 
Then there is an integer $M$ and a positive constant $C$ such that
$$\bigg|E_{L,Z,f}(\psi)\bigg|\leq Cm_{\psi}^{n-1}q_L^{-m_\psi\sigma}$$
for all $L\in\tilde{\cL}_{K,M}$ and all additive characters $\psi$ of $L$ of conductor $m_\psi>1$.
\end{prop}
\subsection{Proof of Proposition \ref{m=1} and \ref{wnsd}}
Throughout this section, we will use the following notation and assumption. Let $w$ be a tuple in $\NN_{\geq 1}^n$ and $J$ be a non-empty finite subset of $\NN_{\geq 1}$. Let $K$ be a number field. Let $f_{ij}\in \cO_K[x_1,\dots,x_n]$ be non-constant polynomials for $i\in J$ and $1\leq j\leq r_i$, be as in \ref{wnsd}, where $r_i>0$ for $i\in J$.  In particular, $d_{w}(f_{ij})=i$ for all $i,j$. Let $\cI$ be the ideal of $\cO_K[x_1,\dots,x_n]$ generated by $f_{ij}$ for $i\in J$ and $1\leq j\leq r_i$. 
We will assume that $\cI_K\neq (1)$. We will use the notations $ F_{ijw}, \sigma_0\big((f_{ij})_{i\in J, 1\leq j\leq r_i}\big), \tilde{\sigma}_{0w}\big((f_{ij})_{i\in J, 1\leq j\leq r_i}\big), s_{wi}, r$ in \ref{wnsd} and Proposition \ref{m=1}. To simplify the notation, we will write  $\sigma_0, \tilde{\sigma}_{0w}$ instead of $\sigma_0\big((f_{ij})_{i\in J, 1\leq j\leq r_i}\big)$ and $\tilde{\sigma}_{0w}\big((f_{ij})_{i\in J, 1\leq j\leq r_i}\big)$ respectively. In particular, we have $r=\sum_{i\in J}r_i$, $\sigma_0=\min_{\ell\in J}\frac{n-s_{w\ell}}{\ell}$ and $\tilde{\sigma}_{0w}=\min_{\ell\in J}\frac{n-s_{w\ell}}{2(\ell-1)},$ where we use the convention that $\frac{n}{0}=+\infty$. We also use the notations  $\mathfrak{D}^{r}_{\cO_K}=\AA_{\cO_K}^r\setminus \Spec(\cO_K[x_1,...,x_n]/(x_1,...,x_n))$, $\mathfrak{D}^{n,r}_{\cO_K}=\AA_{\cO_K}^n\times_{\Spec(\cO_K)}\mathfrak{D}^{r}_{\cO_K}, Y\mathfrak{D}^{r}_{\cO_K}=Y\times_{\Spec(\cO_K)}\mathfrak{D}^{r}_{\cO_K}$ if $Y$ is an $\cO_K$-scheme (see Section \ref{SV}). With the notation $\BSing(F_{i1w},\dots,F_{ir_{i}w})$ in Section \ref{Conjexp}, we have $s_{wi}=\dim(\BSing(F_{i1w},\dots,F_{ir_{i}w}))$. If the set $\{F_{ijw}|i\in J, 1\leq j\leq r_i\}$ is linearly dependent then $s_{wi}=n$ for some $i\in J$. Thus, $\tilde{\sigma}_{0w}=\sigma_0=0$ and there is nothing to do with Proposition \ref{m=1} and \ref{wnsd}. From now on, we will suppose that $\{F_{ijw}|i\in J, 1\leq j\leq r_i\}$ is linearly independent. We set $D=\max\{i|i\in J\}$ 
and $$g\big((a_{ij})_{i\in J, 1\leq j\leq r_i},x_1,\dots,x_n\big)=\sum_{i\in J, 1\leq j\leq r_i}a_{ij}f_{ij}(x).$$
Since $\cI_K\neq (1)$,  it is easily seen that $0$ is the only critical value of $g:\CC^{n+r}\to\CC$.

By Remark \ref{absexp}, in \ref{wnsd} and Proposition \ref{m=1}, we need to study the exponential sums 
\begin{equation}\label{Cm=1}
E^{(r)}_{L,\cI}(m):=E_{L,\mathfrak{D}^{n,r}_{\cO_K},g}(\psi)
\end{equation}
when $L\in\tilde{\cL}_{K,1}$, $m\geq 1$ and $\psi$ is an arbitrary additive character of $L$ of conductor $m$.

Let $\mathsf{F}$ be a field over $\cO_K$,  i.e., $\mathsf{F}$ is a field  endowed with a structure of $\cO_K$-algebra. For each $\textbf{a}_0=(a_{ij})_{i\in J, 1\leq j\leq r_i}\in \mathsf{F}^r$, we set $$g_{\textbf{a}_0}(x_1,\dots,x_n)=g((a_{ij})_{i\in J, 1\leq j\leq r_i},x_1,\dots,x_n).$$
Let $\tilde{g}_{\textbf{a}_0w}$ be the $w$-weighted homogeneous part of highest $w$-degree of $g_{\textbf{a}_0}$. 

For each $\textbf{a}\in \mathsf{F}[[t]]^r$, we write $\textbf{a}=(a_{ij}(t))_{i\in J, 1\leq j\leq r_i}$, where $a_{ij}(t)=\sum_{\ell\geq 0}a_{ij\ell}t^{\ell}$. We set $\textbf{a}^{(\ell)}=(a_{iju})_{i\in J, 1\leq j\leq r_i, \ell\geq u\geq 0}$ and $\textbf{a}_0=\textbf{a}^{(0)}\in \mathsf{F}^r$. Let $x=(x_1(t),\dots,x_n(t))\in \mathsf{F}[[t]]^n$, we write $x_i(t)=\sum_{\ell\geq 0}x_{i\ell}t^\ell$, $x^{(\ell)}=(x_{ij})_{1\leq i\leq n, 0\leq j\leq \ell},$ 
$$g(\textbf{a},x)=\sum_{\ell\geq 0}g_{\ell}\left(\textbf{a}^{(\ell)},x^{(\ell)}\right)t^\ell,
g_{\textbf{a}^{(\ell)}}\left(x^{(\ell)}\right)=g_{\ell}\left(\textbf{a}^{(\ell)},x^{(\ell)}\right),$$
$$\tilde{g}_{\textbf{a}_0w}(x)=\sum_{\ell\geq 0}\tilde{g}_{\textbf{a}_0w\ell}\left(x^{(\ell)}\right)t^\ell$$
as in Section \ref{jtrans}.
We consider the weight $\tilde{w}$ for the variables $(x_{i\ell})_{1\leq i\leq n, 0\leq \ell}$ given by $\tilde{w}(x_{i\ell})=w_i+D\ell$. Since $\{F_{ijw}|i\in J, 1\leq j\leq r_i\}$ is linearly independent, we see that $\tilde{g}_{\textbf{a}_0w}\neq 0$  if $\textbf{a}_0\in \CC^r\setminus\{0\}$. Thus, $\tilde{g}_{\textbf{a}_0w\ell}\neq 0$ for all $\ell\geq 0$ if $\textbf{a}_0\in \CC^r\setminus\{0\}$.

\begin{lem}\label{singlocus}If  $\textnormal{\textbf{a}}_0=(a_{ij})_{i\in J, 1\leq j\leq r_i}\in\CC^r\setminus\{0\}$ and $d_w(g_{\textnormal{\textbf{a}}_0})=\ell$, then one has $\ell\in J$ and $s(\tilde{g}_{\textnormal{\textbf{a}}_0w})=\dim(C_{\tilde{g}_{\textnormal{\textbf{a}}_0w}})\leq s_{w\ell}$.
\end{lem}
\begin{proof} Suppose that $d_w(g_{\textbf{a}_0})=\ell$. By our assumption, $\{F_{ijw}|i\in J, 1\leq j\leq r_i\}$ is linearly independent, thus $\ell\in J$ and $a_{ij}=0$ if $i>\ell$. Moreover, there exists a non-empty subset $I$ of $\{1,\dots,r_{\ell}\}$ such that $\tilde{g}_{\textnormal{\textbf{a}}_0w}=\sum_{j\in I}a_{\ell j}F_{\ell jw}$ and $a_{\ell j}\neq 0$ for all $j\in I$. It is easy to verify that $C_{\tilde{g}_{\textnormal{\textbf{a}}_0w}}\subset \BSing(F_{\ell 1w},\dots,F_{\ell r_{\ell}w})$. Thus, $s(\tilde{g}_{\textnormal{\textbf{a}}_0w})\leq s_{w\ell} =\dim(\BSing(F_{\ell 1w},\dots,F_{\ell r_{\ell}w}))$.
\end{proof}
\begin{lem}\label{highpart}If $\textnormal{\textbf{a}}\in \CC[[t]]^r\setminus t\CC[[t]]^r$ and $m\geq 0$, then the $\tilde{w}$-weighted homogeneous part of highest $\tilde{w}$-degree of $g_{\textnormal{\textbf{a}}^{(m)}}$ is $\tilde{g}_{\textnormal{\textbf{a}}_0wm}$.

\end{lem}
\begin{proof}
Let $\textnormal{\textbf{a}}=(a_{ij}(t))_{i\in J, 1\leq j\leq r_i}\in \CC[[t]]^r\setminus t\CC[[t]]^r$.  We put $\ell=d_w(g_{\textbf{a}_0})$. We write $$g_\textbf{a}(x(t))=\sum_{i\leq \ell, i\in J, 1\leq j\leq r_i}a_{ij}(t)f_{ij}(x(t))+ \sum_{i>\ell,i\in J, 1\leq j\leq r_i}a_{ij}(t)f_{ij}(x(t)),$$
$$\sum_{i\leq \ell, i\in J, 1\leq j\leq r_i}a_{ij}(t)f_{ij}(x(t))=\sum_{0\leq m}h_{1m}\left(x^{(m)}\right)t^m,$$
$$\sum_{\ell< i, i\in J, 1\leq j\leq r_i}a_{ij}(t)f_{ij}(x(t))=\sum_{0\leq m}h_{2m}\left(x^{(m)}\right)t^m.$$
Let $m\geq 0$. We have $g_{\textnormal{\textbf{a}}^{(m)}}\left(x^{(m)}\right)=h_{1m}\left(x^{(m)}\right)+h_{2m}\left(x^{(m)}\right)$. Note that $\textnormal{\textbf{a}}_0\in\CC^r\setminus\{0\}$, thus we have $\tilde{g}_{\textnormal{\textbf{a}}_0wm}\neq 0$ as explained above.  Since $\{F_{ijw}|i\in J, 1\leq j\leq r_i\}$ is linearly independent and $\ell=d_w(g_{\textbf{a}_0})$, we have  $a_{ij}(t)\in t\CC[[t]]$ if $i>\ell$ and $a_{\ell j}(t)\in \CC[[t]]\setminus t\CC[[t]]$ for some $j$. Moreover, the $\tilde{w}$-weighted homogeneous part of highest $\tilde{w}$-degree of $h_{1m}$ is $\tilde{g}_{\textnormal{\textbf{a}}_0wm}$. It is easy to verify that $d_{\tilde{w}}(\tilde{g}_{\textnormal{\textbf{a}}_0wm})=mD+\ell$ and $d_{\tilde{w}}(h_{2m})\leq (m-1)D+D<mD+\ell$. Thus, the $\tilde{w}$-weighted homogeneous part of highest $\tilde{w}$-degree of $g_{\textnormal{\textbf{a}}^{(m)}}$ is $\tilde{g}_{\textnormal{\textbf{a}}_0wm}$.

\end{proof}
\begin{prop}\label{estimatefinitesums}
There is an integer $M$ depending only on $(f_{ij})_{i\in J, 1\leq j\leq r_i}$ and a constant $C$ depending only on $n, D$ such that for all finite fields $k$ over $\cO_K$ with $\Ch(k)>M$ and all non-trivial characters $\Psi:k\to\CC^*$, we have 
\begin{equation}\label{estimatefinite}
\frac{1}{\#k^{n+r}}\bigg|\sum_{((a_{ij})_{i\in J, 1\leq j\leq r_i},x)\in (k^r\setminus\{0\})\times k^n}\Psi\bigg(\sum_{i\in J, 1\leq j\leq r_i}a_{ij}f_{ij}(x)\bigg)\bigg|\leq C\#k^{\frac{-n+s_w}{2}},
\end{equation}
where $s_w=\max_{\ell\in J} s_{w\ell}$.
\end{prop}
\begin{proof}
By logical compactness (see \cite[Corollary 2.2.10]{Marker}) and Lemma \ref{singlocus}, there exists an integer $M$ such that for all finite fields $k$ over $\cO_K$ with $\Ch(k)>M$ and all $\textbf{a}_0=(a_{ij})_{i\in J, 1\leq j\leq r_i}\in k^{r}\setminus\{0\}$, if $d_w(g_{\textbf{a}_0})=\ell$ then $\ell\in J$ and $s(\tilde{g}_{\textbf{a}_0w}(x))\leq s_{w\ell}\leq s_w$. By enlarging $M$, we can suppose that $M>D$.  We use Corollary \ref{quasisums} to obtain a constant $C$ depending only on $n,D$ such that 
\begin{align*}
\bigg|\sum_{(\textbf{a}_0,x)\in (k^r\setminus\{0\})\times k^n}\Psi\bigg(\sum_{i\in J, 1\leq j\leq r_i}a_{ij}f_{ij}(x)\bigg)\bigg|&=\bigg|\sum_{\textbf{a}_0\in k^r\setminus\{0\}}\sum_{x\in k^n}\Psi\left(g_{\textbf{a}_0}(x)\right)\bigg|\\
&\leq C(\#k^{r}-1)\#k^{\frac{n+s_w}{2}}
\end{align*}
for all finite fields $k$ over $\cO_K$ with $\Ch(k)>M$ and all non-trivial characters $\Psi$ of $k$. This yields (\ref{estimatefinite}).
\end{proof}
\begin{proof}[Proof of Proposition \ref{m=1}]
Recall that we need to find an integer $M$, a positive constant $c$ and a positive constant $c_L$ for each $L\in\tilde{\cL}_{K,1}$ such that $c_L=c$ if $L\in \tilde{\cL}_{K,M}$ and   $$\left|E_{L,\cI}^{(r)}(1)\right|\leq c_Lq_L^{-\sigma_0}$$
for all $L\in\tilde{\cL}_{K,1}$.
By (\ref{Cm=1}), we have $$E_{L,\cI}^{(r)}(1)=q_L^{-n-r}\sum_{(\textbf{a}_0,x)\in (k_L^r\setminus\{0\})\times k_L^n}\Psi\bigg(\sum_{i\in J, 1\leq j\leq r_i}a_{ij}f_{ij}(x)\bigg)$$
for all non-trivial characters $\Psi$ of $k_L$.

If $D>1$, then our claim follows from Proposition \ref{estimatefinitesums} with the remark that if $w=(1,1,\dots,1)$, then 
$$\sigma_0=\min_{\ell\in J} \frac{n-s_{w\ell}}{\ell}\leq \frac{n-s_w}{2}.$$

If $D=1$, since $\{F_{1jw}|1\leq j\leq r_i\}$ is linearly independent, one has 
$$\sum_{((a_{1j})_{1\leq j\leq r_1},x)\in (k^r\setminus\{0\})\times k^n}\Psi\bigg(\sum_{1\leq j\leq r_1}a_{1j}f_{1j}(x)\bigg)=0$$
for all finite fields $k$ over $\cO_K$ of large enough characteristic and all non-trivial characters $\Psi:k\to\CC^{*}$. Thus, our claim is trivial.
\end{proof}

With the above notation, we set
$$\gamma_w=\gamma_w\left((f_{ij})_{i\in J, 1\leq j\leq r_i}\right)=\min_{\textbf{a}_0\in \CC^r\setminus\{0\}}\lct\left(\cJ_{\tilde{g}_{\textbf{a}_0w}}\right),$$
where if $f$ is a non-constant polynomial in $\CC[x_1,\dots,x_n]$, then $\lct(\cJ_f)$ is the log canonical threshold (see Section \ref{SJ}) of the Jacobian ideal $\cJ_f$ of $f$.  Here, we use the convention that $\lct((1))=+\infty$.
\begin{lem}\label{sigw}If $w=(1,1,\dots,1)$ then $\gamma_w\geq 2\tilde{\sigma}_{0w}.$
\end{lem}
\begin{proof}
If $w=(1,\dots,1)$, then  the proof of \cite[Corollary 3.9]{Nguyennsd}, Lemma \ref{singlocus} and  the definition of $\tilde{\sigma}_{0w}$ imply that
$$\lct\left(\cJ_{\tilde{g}_{\textbf{a}_0w}}\right)\geq \frac{n-s(\tilde{g}_{\textbf{a}_0w})}{d_w(\tilde{g}_{\textbf{a}_0w})-1}\geq 2\tilde{\sigma}_{0,w}$$
for all $\textbf{a}_0\in\CC^r\setminus\{0\}$. By the definition of $\gamma_w$, we have 
$\gamma_w\geq 2\tilde{\sigma}_{0w}.$
\end{proof}
\begin{lem}\label{lctdimen} For all $\textnormal{\textbf{a}}_0\in \CC^r\setminus\{0\}$ and all $m\geq 1$, we have 
$$s(\tilde{g}_{\textnormal{\textbf{a}}_0w(m-1)})\leq mn-m\gamma_w.$$
\end{lem}
\begin{proof}
It follows from the proof of \cite[Lemma 3.1]{Nguyennsd} that
$$C_{\tilde{g}_{\textbf{a}_0w(m-1)}}(\CC)=\{x^{(m-1)}\in\CC^{mn}|\ord_t\left(\cJ_{\tilde{g}_{\textbf{a}_0w}}\left(\beta\left(x^{(m-1)}\right)\right)\right)\geq m\}$$
where  $$\beta\left(x^{(m-1)}\right)=\bigg(\sum_{0\leq j\leq m-1}x_{1j}t^j,\dots,\sum_{0\leq j\leq m-1}x_{nj}t^j\bigg)$$
and  if $y\in \CC[[t]]^n$ then $\ord_t\left(\cJ_{\tilde{g}_{\textbf{a}_0w}}(y)\right)=\min_{h\in \cJ_{\tilde{g}_{\textbf{a}_0w}}}\ord_t(h(y)).$
By Proposition \ref{jetlog} and the definition of $\gamma_w$, one has
$$s(\tilde{g}_{\textbf{a}_0w(m-1)})\leq mn-m\lct(\cJ_{\tilde{g}_{\textbf{a}_0w}})\leq mn-m\gamma_w.$$
\end{proof}
\begin{thm}\label{boundexpo} There is a positive constant $C$ and an integer $M$ depending only on $g$ such that for all $L\in\tilde{\cL}_{K,M}$ and all additive characters $\psi$ of $L$ of conductor $m_\psi\geq 1$, we have 
$$\left|E_{L,\mathfrak{D}_{\cO_K}^{n,r},g}(\psi)\right|\leq Cm_\psi^{n+r-1}q_L^{-\frac{\gamma_wm_\psi}{2}}.$$
\end{thm}
\begin{proof}
By Lemmas \ref{highpart}, \ref{lctdimen} and the logical compactness (see \cite[Corollary 2.2.10]{Marker}), for each $m\geq 1$, there is an integer $M_m>mD$ such that for all finite fields $k$ over $\cO_K$ with $\Ch(k)>M_m$ and all $\textbf{a}\in k[[t]]^r\setminus tk[[t]]^r$, the $\tilde{w}$-weighted homogeneous part of highest $\tilde{w}$-degree of $g_{\textbf{a}^{(m-1)}}\left(x^{(m-1)}\right)$ is 
$\tilde{g}_{\textbf{a}_0w(m-1)}\left(x^{(m-1)}\right)$ and 
\begin{equation}\label{upb}
s\left(\tilde{g}_{\textbf{a}_0w(m-1)}\right)\leq mn-m\gamma_w.
\end{equation}
By a simple calculation, we have  $d_{\tilde{w}}\left(g_{\textbf{a}^{(m-1)}}\left(x^{(m-1)}\right)\right)\leq mD$ for all $\textbf{a}\in k[[t]]^r\setminus tk[[t]]^r$ and all $m\geq 1$. By  Corollary \ref{quasisums} and  (\ref{upb}), there exists a constant $C_{m}$ depending only on $m,w,n,D$ such that for all finite fields $k$ over $\cO_K$ with $\Ch(k)>M_{m}$,  all $\textbf{a}\in k[[t]]^r\setminus tk[[t]]^r$ and all non-trivial characters $\Psi:k\to\CC^*$, we have 
\begin{align*}
\bigg|\sum_{x^{(m-1)}\in k^{mn}}\Psi\left(g_{\textbf{a}^{(m-1)}}\left(x^{(m-1)}\right)\right)\bigg|&\leq  C_{m}\#k^{\frac{mn+mn-m\gamma_w}{2}}.
\end{align*}
Therefore,  for each $m\geq 1$, we have 
\begin{align*}
&\bigg|\frac{1}{\#k^{m(n+r)}}\sum_{(\textbf{a}^{(m-1)},x^{(m-1)})\in k^{m(n+r)}, \textbf{a}_0\neq 0}\Psi\left(g_{m-1}\left(\textbf{a}^{(m-1)},x^{(m-1)}\right)\right)\bigg|\leq C_{m}\#k^{\frac{-m\gamma_w}{2}}
\end{align*}for all finite fields $k$ over $\cO_K$ with $\Ch(k)>M_{m}$ and all non-trivial characters $\Psi:k\to\CC^*$. From this and Proposition \ref{transfer} together with the fact that $0$ is the only critical value of $g:\CC^n\to \CC$, %
there is a constant $C'$ and an integer $M'$ depending only on $g$ such that for all $L\in\tilde{\cL}_{K,M'}$ and all additive characters $\psi$ of $L$ of conductor $m_\psi\geq 2$, we have 
$$\left|E_{L,\mathfrak{D}_{\cO_K}^{n,r},g}(\psi)\right|\leq C'm_\psi^{n+r-1}q_L^{-\frac{\gamma_wm_\psi}{2}}.$$
We can set $M=\max\{M_1,M'\}$ and $C=\max\{C_1,C')$ to finish our proof.
\end{proof}
\begin{cor}\label{subspace} Let $I_1, I_2$ be disjoint subsets of $\{1,\dots,n\}$ and $Z_{I_1I_2}=(\{x_i=0\hspace{0.1cm}\forall i\in I_1\}\setminus \{x_j=0 \hspace{0.1cm}\forall j\in I_2\})\subset\AA_{\cO_K}^n$, then there is a constant $C$ and an integer $M$ depending only on $g$ such that  
$$\left|E_{L,Z_{I_1I_2}\mathfrak{D}_{\cO_K}^{r},g}(\psi)\right|\leq Cm_\psi^{n+r-1}q_L^{-\frac{\gamma_wm_\psi}{2}}$$
for all $L\in\tilde{\cL}_{K,M}$ and all additive characters $\psi$ of $L$ of conductor $m_\psi\geq 1$.
\end{cor}
\begin{proof}
Note that Lemma \ref{homosums} also holds for weighted-homogeneous polynomials. Indeed, this follows from using 
Corollary \ref{quasisums} instead of \cite[Theorem 4]{Katz} in the argument of Lemma \ref{homosums}. Now, we use this fact to obtain a constant $C_m$ depending only on $n,D,w,m$ and an integer $M_m$ depending only on $g,m$ such that 
\begin{align*}
\bigg|\sum_{\{x^{(m-1)}\in k^{mn}| x_{i0}=0 \hspace{0.1cm}\forall i\in I_1, x_{i0}\neq 0 \hspace{0.1cm}\forall i\in I_2\}}\Psi\left(g_{\textbf{a}^{(m-1)}}\left(x^{(m-1)}\right)\right)\bigg|&\leq  C_{m}\#k^{\frac{mn+mn-m\gamma_w}{2}}
\end{align*}
for all $I_1, I_2\subset \{1,\dots,n\}$ with  $I_1\cap I_2=\emptyset$, all finite fields $k$ over $\cO_K$ with $\Ch(k)>M_m$, all $\textbf{a}\in k[[t]]^r\setminus tk[[t]]^r$ and all non-trivial characters $\Psi$ of $k$.  The proof of Theorem \ref{boundexpo} can be repeated to prove our claim.
\end{proof}
In order to obtain a version of Theorem \ref{boundexpo} for all $L\in \cL_{K,1}$, we use a corollary of \cite[Proposition 4.1]{Nguyennsd} as follows.
\begin{prop}\label{smallp}Let $Y$ be a subscheme of $\AA_{\cO_K}^n$ and $Z= Y\mathfrak{D}_{\cO_K}^{r}$. Then for each $L\in\cL_{K,1}$ and each $0<\sigma<\frac{\lct_{L,Z}(\cJ_g)}{2}$, there exists a constant $C_{g,Y,L,\sigma}>0$ depending only on $g,Y,L,\sigma$ such that for all additive characters $\psi$ of $L$, we have
$$\left|E_{L,Z,g}(\psi)\right|\leq C_{g,Y,L,\sigma}q_L^{-\sigma m_\psi},$$
where we recall that 
$$\lct_{L,Z}(\cJ_g)=\min_{(\textnormal{\textbf{a}},x)\in (\cO_L^r\setminus \varpi_L\cO_L^r)\times \{x\in\cO_L^n\mid x \textnormal{ mod } (\varpi_L)\in Y(k_L)\}}\lct_{(\textnormal{\textbf{a}},x)}(\cJ_g).$$
\end{prop}
\begin{proof}
The proof follows by using \cite[Proposition 4.1]{Nguyennsd} with respect to the polynomial $g$ and the Schwartz-Bruhat function $$\phi_{L,Z}=\11_{(\cO_L^r\setminus \varpi_L\cO_L^r)\times\{x\in\cO_L^n\mid x \textnormal{ mod } (\varpi_L)\in Y(k_L)\}}.$$
\end{proof}
To complete the proof of  \ref{wnsd}, we need the following lemma.
\begin{lem}\label{lowerboundlct}With the notation of Proposition \ref{smallp},  we have 
$$\lct_{L,Z}(\cJ_g)\geq \gamma_w$$
for all $L\in\cL_{K,1}$.
\end{lem}
\begin{proof}
By the definition, we have 
\begin{align*}
\lct_{L,Z}(\cJ_g)&\geq\min_{(\textbf{a},x)\in (\cO_L^r\setminus \varpi_L\cO_L^r)\times \cO_L^n}\lct_{(\textbf{a},x)}(\cJ_g)\geq \min_{(\textbf{a},x)\in (L^r\setminus \{0\})\times L^n}\lct_{(\textbf{a},x)}(\cJ_g)\\
&\geq \min_{(\textbf{a}_0,x^{(0)})\in (\CC^r\setminus \{0\})\times \CC^n}\lct_{(\textbf{a}_0,x^{(0)})}(\cJ_g).
\end{align*}
By \cite[Properties 1.12 and 1.17]{Mustata2}, if $(\textbf{a}_0,x^{(0)})\in (\CC^r\setminus \{0\})\times\CC^n$, then we have 
$$\lct_{(\textbf{a}_0,x^{(0)})}(\cJ_g)\geq \lct_{x^{(0)}}(\cJ_{g}|_{\textbf{a}=\textbf{a}_0})\geq \lct_{x^{(0)}}(\cJ_{g_{\textbf{a}_0}})=\lct_{0}(\cJ_{g_{\textbf{a}_0,x^{(0)}}}),$$
where $g_{\textbf{a}_0,x^{(0)}}(x)=g_{\textbf{a}_0}(x+x^{(0)})$. We recall that 
$$\cJ_{g_{\textbf{a}_0,x^{(0)}}}=\left(\frac{\partial g_{\textbf{a}_0,x^{(0)}}}{\partial x_1},\dots,\frac{\partial g_{\textbf{a}_0,x^{(0)}}}{\partial x_n}\right).$$
If $\lambda\neq 0$, we set $\theta_\lambda(x_1,\dots,x_n)=(\lambda^{-w_1}x_1,\dots,\lambda^{-w_n}x_n)$ and denote by $\cG_\lambda$  the ideal generated by $\{f(\theta_\lambda(x))|f\in \cJ_{g_{\textbf{a}_0,x^{(0)}}}\}$. Therefore, we have

\begin{align*}
\cG_\lambda&=\left(\frac{\partial g_{\textbf{a}_0,x^{(0)}}}{\partial x_1}(\theta_\lambda(x)),\dots,\frac{\partial g_{\textbf{a}_0,x^{(0)}}}{\partial x_n}(\theta_\lambda(x))\right)\\
&=\left(\lambda^{d_{\textbf{a}_0}-w_1}\frac{\partial g_{\textbf{a}_0,x^{(0)}}}{\partial x_1}(\theta_\lambda(x)),\dots,\lambda^{d_{\textbf{a}_0}-w_n}\frac{\partial g_{\textbf{a}_0,x^{(0)}}}{\partial x_n}(\theta_\lambda(x))\right),
\end{align*}
where $d_{\textbf{a}_0}=d_w(g_{\textbf{a}_0,x^{(0)}})=d_w(g_{\textbf{a}_0})$. Thus, $\lct_0(\cJ_{g_{\textbf{a}_0,x^{(0)}}})=\lct_0(\cG_\lambda)$ for all $\lambda\neq 0$. By using the
semi-continuity of log canonical thresholds in \cite[Property 1.24]{Mustata2}, we have 
\begin{align*}
\lct_0(\cG_\lambda)&=\lct_0\left(\lambda^{d_{\textbf{a}_0}-w_1}\frac{\partial g_{\textbf{a}_0,x^{(0)}}}{\partial x_1}(\theta_\lambda(x)),\dots,\lambda^{d_{\textbf{a}_0}-w_n}\frac{\partial g_{\textbf{a}_0,x^{(0)}}}{\partial x_n}(\theta_\lambda(x))\right)\\&\geq \lct_0(\cJ_{\tilde{g}_{\textbf{a}_0w}}).
\end{align*}
By the definition of $\gamma_w$, we have  $\lct_{(\textbf{a}_0,x^{(0)})}(\cJ_g)\geq \lct_0(\cJ_{\tilde{g}_{\textbf{a}_0w}})\geq \gamma_w$
for all $(\textbf{a}_0,x^{(0)})\in(\CC^r\setminus\{0\})\times\CC^n$. Thus, $\lct_{L,Z}(\cJ_g)\geq \gamma_w$.
\end{proof}

\begin{proof}[Proof of \ref{wnsd}]
If $w=(1,\dots,1)$, then \ref{wnsd} follows from combining (\ref{Cm=1}), Lemmas \ref{sigw}, \ref{lowerboundlct}, 
Theorem \ref{boundexpo} and Proposition \ref{smallp}.

To prove the general case, we use the torus transformations in Lemma \ref{torus}. From polynomials $f_{ij},F_{ijw}$ for $i\in J$ and $1\leq j\leq r_i$, we obtains polynomials $\hat{f}_{ij}, \hat{F}_{ijw}$ in $|w|$ variables $y_1,\dots,y_n,\dots,y_{|w|}$ by Lemma \ref{torus} with the following properties:
\begin{itemize}
\item[(\textit{i}),]$\deg(\hat{f}_{ij})=i$ for all $i,j$.
\item[(\textit{ii}),]The homogeneous part of highest degree of $\hat{f}_{ij}$ is $\hat{F}_{ijw}$ for all $i,j$.
\item[(\textit{iii}),]$\{\hat{F}_{ijw}| i\in J, 1\leq j\leq r_i\}$ is linearly independent.
\end{itemize} 
Let $G(\textbf{a},y)=\sum_{i\in J, 1\leq j\leq r_i}a_{ij}\hat{f}_{ij}$ and $\tilde{Z}=(\mathfrak{D}_{\cO_K}^{1})^{|w|-n}\times_{\Spec(\cO_K)}\AA_{\cO_K}^n.$ Then we have 
\begin{align*}
E_{L,\tilde{Z}\mathfrak{D}_{\cO_K}^{r},G}(\psi)&=\int_{(\cO_L\setminus\varpi_L\cO_L)^{\left|w\right|-n}\times \cO_L^n\times (\cO_L^r\setminus \varpi_L\cO_L^r)}\psi(G)\left|dy_1\wedge...\wedge dy_{\left|w\right|}\right|\left|d\textbf{a}\right|
\\&=(1-q_L^{-1})^{\left|w\right|-n}\int_{\cO_L^n\times (\cO_L^r\setminus \varpi_L\cO_L^r)}\psi(g)\left|dx_1\wedge...\wedge dx_{n}\right|\left|d\textbf{a}\right|
\\&=(1-q_L^{-1})^{\left|w\right|-n}E_{L,\mathfrak{D}_{\cO_K}^{n,r},g}(\psi)
\end{align*}for all $L\in\tilde{\cL}_{K,1}$ and all  additive characters $\psi$ of $L$ of conductor $m_\psi\geq 1$. Thus, we have
\begin{equation}\label{relation1}
\left|E_{L,\mathfrak{D}_{\cO_K}^{n,r},g}(\psi)\right|\leq 2^{\left|w\right|-n}\bigg|E_{L,\tilde{Z}\mathfrak{D}_{\cO_K}^{r},G}(\psi)\bigg|
\end{equation}
for all $L\in\tilde{\cL}_{K,1}$ and all  additive characters $\psi$ of $L$ of conductor $m_\psi\geq 1$.

We set $$G_{\textbf{a}_0}(y_1,\dots,y_{\left|w\right|})=\sum_{i\in J, 1\leq j\leq r_i}a_{ij}\hat{f}_{ij}(y_1,\dots,y_{\left|w\right|})$$ for each $\textbf{a}_0=(a_{ij})_{i\in J, 1\leq j\leq r_i}\in (\CC^r\setminus\{0\})$
and denote by  $\tilde{G}_{\textbf{a}_0}$ the homogeneous part of highest degree of $G_{\textbf{a}_0}$.  We also set $$\gamma=\min_{\textbf{a}_0\in\CC^r\setminus\{0\}}\lct(\cJ_{\tilde{G}_{\textbf{a}_0}}).$$ Then we will use Corollary \ref{subspace} for $G$ and suitable pairs $(I_1,I_2)$ of disjoint subsets of $\{1,\dots,|w|\}$ to obtain a constant $C$ and an integer $M$ such that 
\begin{equation}\label{relation2}
\bigg|E_{L,\tilde{Z}\mathfrak{D}_{\cO_K}^{r},G}(\psi)\bigg|\leq Cm_\psi^{r+\left|w\right|-1}q_L^{-\frac{m_\psi\gamma}{2}}
\end{equation} 
for all $L\in\tilde{\cL}_{K,M}$ and all additive characters $\psi$ of $L$ of conductor $m_\psi\geq 1$.

Similarly, we use Proposition \ref{smallp} and Lemma \ref{lowerboundlct} for $G$ to deduce that for each $L\in\cL_{K,1}$ and each $0<\sigma<\gamma/2$, there is a constant $C_{G,\tilde{Z},L,\sigma}$ such that
\begin{equation}\label{relation3}
\bigg|E_{L,\tilde{Z}\mathfrak{D}_{\cO_K}^{r},G}(\psi)\bigg|\leq C_{G,\tilde{Z},L,\sigma}q_L^{-m_\psi\sigma}
\end{equation} 
for all additive characters $\psi$ of $L$.

By (\ref{Cm=1}), (\ref{relation1}), (\ref{relation2}), (\ref{relation3}), it remains to prove that $\gamma\geq 2\tilde{\sigma}_{0w}$. By the proof of Lemma \ref{sigw}, we have 
$$\lct(\cJ_{\tilde{G}_{\textbf{a}_0})}\geq \frac{\left|w\right|-s(\tilde{G}_{\textbf{a}_0})}{\deg(\tilde{G}_{\textbf{a}_0})-1}.$$
But it is easy to verify that $\deg(\tilde{G}_{\textbf{a}_0})=d_w(\tilde{g}_{\textbf{a}_0w})$ and $\tilde{G}_{\textbf{a}_0}=\widehat{\tilde{g}_{\textbf{a}_0w}}$ is obtained from $\tilde{g}_{\textbf{a}_0w}$ by using $|w|-n$ basic torus transformations as in Lemma \ref{torus}. Thus, we have $s(\tilde{G}_{\textbf{a}_0})\leq s(\tilde{g}_{\textbf{a}_0w})+|w|-n$. On the other hand, Lemma \ref{singlocus} implies that $s(\tilde{g}_{\textbf{a}_0w})\leq s_{w\ell}$ if $d_w(\tilde{g}_{\textbf{a}_0w})=\ell$. Therefore, we have
$$\lct(\cJ_{\tilde{G}_{\textbf{a}_0}})\geq \frac{n-s(\tilde{g}_{\textbf{a}_0w})}{d_w(\tilde{g}_{\textbf{a}_0w})-1}\geq \min_{\ell\in J}\frac{n-s_{w\ell}}{\ell-1}=2\tilde{\sigma}_{0w}.$$
Hence, $\gamma\geq 2\tilde{\sigma}_{0w}$ as desired.
\end{proof}
\begin{proof}[Proof of Corollary \ref{countinguni}]We consider the $\ZZ$-subscheme $X$ of $\AA_\ZZ^n\times \cA$ given by polynomials $H_{ij}(x,g)=f_{ij}(x)-g_{ij}(x)$ for $i\in J$ and $1\leq j\leq r_i$, where $g=(g_{ij})_{i\in J, 1\leq j\leq r_i}\in\cA$. Let $\varphi$ be the projection from $X$ to $\cA$. Note that $\tilde{\sigma}_{0w}\big((f_{ij}-g_{ij})_{i\in J, 1\leq j\leq r_i}\big)=\tilde{\sigma}_{0w}\big((f_{ij})_{i\in J, 1\leq j\leq r_i}\big)>r$ for all $g=(g_{ij})_{i\in J, 1\leq j\leq r_i}\in \cA(\overline{\QQ})$. Thus, \ref{wnsd}, Propositions \ref{equcondi} and \ref{delta} imply that every  fiber of $\varphi_{\overline{\QQ}}$ is geometrically irreducible and a complete intersection of dimension $n-r$ having only rational singularities. Therefore, $X_\QQ$ is also a complete intersection variety of dimension $n+\dim(\cA_\QQ)-r$. Thus, $X_\QQ$ is Cohen-Macauley. So $\varphi_\QQ$ is flat by using  the miracle flatness theorem and the fact that $\cA_{\QQ}$ is smooth.  Moreover, for each $g=(g_{ij})_{i\in J, 1\leq j\leq r_i}\in \cA(\overline{\QQ})$,  \ref{wnsd} implies that Conjecture \ref{avelocIgu} holds for $\sigma<\tilde{\sigma}_{0w}\big((f_{ij})_{i\in J, 1\leq j\leq r_i}\big)$ and the ideal generated by polynomials $M(f_{ij}-g_{ij})$ for $i\in J, 1\leq j\leq r_i$, where $M$ is a non-zero integer such that the coefficients of $Mg_{ij}$ are integral over $\ZZ$ for all $i,j$. Thus, Corollary \ref{countinguni} follows from repeating the proof of Theorem \ref{improE} for $\varphi$ and $E=\lceil 2\big(\tilde{\sigma}_{0w}\big((f_{ij})_{i\in J, 1\leq j\leq r_i}\big)-r\big)\rceil>0.$

\end{proof}
\section{Some open questions}\label{OP}
\subsection{} First of all, we will formulate a conjecture to relate the minimal exponent of ideals in Section \ref{BM} and the motivic oscillation index in Definition \ref{moidef}.  By the strong monodromy conjecture, we can expect that the following conjecture holds.
\begin{conj}Let $K$ be a number field. Let $Y$ be the closed subscheme of $\AA_{\cO_K}^n$ associated to a non-zero ideal $\cI$. Suppose that $\cI=(f_1,...,f_r)\subset \cO_K[x_1,...,x_n]$ such that $\cI_K\neq (1)$ and $\dim(Y_K)=n-r$ then 
$\moi_K(\cI)=\tilde{\alpha}_{\cI_K}$.
\end{conj}

On the other hand, a lot of recent research has been successful in characterizing higher rational singularities and higher Du Bois singularities of an arbitrary locally complete intersection in terms of its minimal exponent (see \cite{JKSY-k,FL-k1,FL-k2,Must-Pop-k,CDM-k}). Indeed, let $e$ be a natural number and $Z$ be a locally complete intersection subvariety of a smooth irreducible variety $X$ such that $Z$ is of pure codimension $r$. Then $Z$ has only $e$-rational singularities if and only if the minimal exponent $\tilde{\alpha}_Z$ is strictly larger than $e+r$ (see \cite{CDM-k}). Similarly, $Z$ has only $e$-Du Bois singularities  if and only if $\tilde{\alpha}_Z\geq e+r$ (see \cite{Must-Pop-k}). By the philosophy of the strong monodromy conjecture, it is natural to ask how to characterize higher rational singularities and higher Du Bois singularities by using the local motivic oscillation indexes of ideals. Let $Z$ be a locally complete intersection variety of pure dimension over a number field $K$. We set $\kappa(Z)=\dim(Z)+\min_{P\in Z(\overline{K})}\moi_{K(P),P}^{\textnormal{aloc}}(Z)$. Note that the concept of $0$-rational singularities (resp. $0$-Du Bois singularities) agrees with the concept of rational singularities (resp. Du Bois singularities). On the other hand, if $Z$ is normal and Gorenstein,  then $Z$ has only Du Bois singularities if and only if $Z$ has only log canonical singularities (see \cite{K-DuBois,KK-DuBois,KS-DuBois}). Thus, the argument in Section \ref{property} answered our question for $0$-rational singularities and $0$-Du Bois singularities. Namely, $Z$ has only $0$-rational singularities if and only if $\kappa(Z)>0$. Moreover, if $Z$ is normal, then $Z$ has only $0$-Du Bois singularities if and only if $\kappa(Z)\geq 0$. More generally, we propose the following conjecture for our question.
\begin{conj}\label{k-rati} Let $e$ be a positive integer and $Z$ be a locally complete intersection variety of pure dimension over a number field $K$, then $Z$ has only $e$-rational singularities (resp. $e$-Du Bois singularities) if and only if $\kappa(Z)>e$ (resp. $\kappa(Z)\geq e$). 
\end{conj}
\subsection{}For the next question, we would like to ask how to relax the smoothness in the second part of \ref{FRS}. Let us use the notation of \ref{FRS} and the proof of its second part. Clearly, it may not allow us to have too bad singularities for each $X_i$. At least, we may require that $X_i$ is a locally complete intersection of pure dimension and has only rational singularities for all $i$. Looking again at the proof of \ref{FRS}, we see that the equality $I_2(L,m)=0$ is no longer true without the smoothness of $X_i$. If we try to use the strategy for estimating $I_1(L,m_d)$ in the proof of \ref{FRS} to deal with $I_2(m,d)$, the integral
$$\int_{\varpi_L\cO_L^{n}\times (\cO_L^N\setminus \varpi_L\cO_L^{N})}\psi\bigg(\sum_{1\leq i\leq\ell} \sum_{1\leq j\leq N_i} y_{ij}f_{ij}(x)\bigg)\mu_{L^{n+N}},$$
 will appear when we take $y_e=0$ for all $e$. This integral is a linear combination with $\NN$-coefficients of the integrals
$$I_{i,I,J}(L,m_\psi)=\int_{\varpi_L\cO_L^{n}\times (\cO_L^{N_i}\setminus \varpi_L\cO_L^{N_i})\times \cO_L^{\sum_{j\in J}N_j}\times \varpi_L\cO_L^{\sum_{j\in I}N_j}}\psi\bigg(\sum_{1\leq i\leq\ell} \sum_{1\leq j\leq N_i} y_{ij}f_{ij}(x)\bigg)\mu_{L^{n+N}}$$
where $\{i\}\sqcup I\sqcup J=\{1,...,\ell\}$. Recall that $P_i=0\in \AA_{\cO_K}^{n_i}$ for each $i$. By the definition of $\moi_{K, P_i}^{(N_i)}(X_i)$ and the fact that $X_i$ is a locally complete intersection for each $i$, there is a constant $a$ and a strictly increasing sequence $(m_j)_{j\geq 1}$ of positive integers  such that for each $M>0$ and each $j\geq 1$, the size of $\left|I_{i,I,J}(L,m_\psi)\right|$ is at least $$q_L^aq_L^{-(\moi_{K,P_i}^{(N_i)}(X_i)+\sum_{j\neq i}N_j)m_\psi}$$
 for a local field  $L\in\cL_{\QQ,M}$ and an additive character $\psi$ of $L$ of conductor $m_\psi=m_j$. So we may require that $$\moi_{K,P_i}^{(N_i)}(X_i)+\sum_{j\neq i}N_j>r+\sum_{1\leq j\leq \ell}N_j$$
 for all $1\leq i\leq \ell$ to be able to apply  \ref{locpole}. In other words,  we may need the condition that $\kappa(X_i)>r$ for all $i$. In particular, if Conjecture \ref{k-rati} holds, then we may require that $X_i$ has only $r$-rational singularities for all $i$ to relax the smooth condition of \ref{FRS}. We formulate our question as follows.
\begin{Op}\label{non-smooth} Could we have a version of the second part of \ref{FRS} for locally complete intersection varieties of pure dimension and have only $r$-rational singularities instead of smooth varieties? Similarly, we ask the same question for equi-dimensional and locally complete intersection varieties  $(X_i)_{i\geq 1}$ such that $\kappa(X_i)>r$ for all $i\geq 1$.
\end{Op}

\subsection{} For the last question, we would like to know whether we can extend our result for the Waring type problem in \ref{adduni} to the setting of definable sets. Let $\cL_{rings}=\{0,1,+,-, .\}$ be the language of rings. Let $(A_i)_{i\geq 1}$ be $\cL_{rings}$-definable subsets of $\AA_\ZZ^r$ of complexity at most $(N,R,D)$. Here, a $\cL_{rings}$-definable set $X$ is of complexity at most $(N,R,D)$ if the following conditions hold:
\begin{itemize}
\item[(\textit{i}),]$X$ can be defined by a formula $\varphi$ in at most $N$ variables (both free variables and non-free variables).
\item[(\textit{ii}),] There are at most $R$ polynomials appearing in the formula $\varphi$ and each of these polynomials is of degree at most $D$.
\end{itemize}
Suppose that $A_i(\CC)$ is neither finite nor contained in any proper affine subspace of $\AA_\CC^r$ for all $i$. By the work of Chatzidakis, van den Dries and Macintyre in \cite{CDM}, if $p$ is large enough then $A_i(\FF_p)$
 is a finite Boolean combination of $\cL_{rings}$-definable sets of the form $\varphi_{ij}(X_{ij})$ for  algebraic sets $X_{ij}$ and  morphisms $\varphi_{ij}$.  Thus, the first part of \ref{adduni} might probably extend to the situation where 
$(A_i)_{i\geq 1}$ are $\cL_{rings}$-definable sets and $m=1$. We may wonder whether this fact also holds for all $m\geq 1$. More generally, we may ask this question for an arbitrary group scheme $(G,\cdot)$ of finite type over $\ZZ$ instead of $\AA_\ZZ^r$. Let us formulate  precisely an open question about it.
\begin{Op}\label{defi}Let $(G,\cdot)$ be a group scheme of finite type over $\ZZ$. For each triple $(N,R,D)\in\ZZ_{>0}^3$, is there an integer $\ell_0$ depending only on $(N,R,D)$ such that the following sentence is true:
 If $A_1,\dots,A_{\ell_0}$ are $\cL_{rings}$-definable subsets of $G$ of complexity at most $(N,R,D)$ such that $A_i(\CC)$ is neither finite nor contained in any right and left coset of any arbitrary proper subgroup of $G(\CC)$ for all $i$,  then there is an integer $M$ depending only on $A_1,\dots,A_{\ell_0}$ such that $$A_1(\ZZ/p^m\ZZ)\cdot...\cdot A_{\ell_0}(\ZZ/p^m\ZZ)=G(\ZZ/p^m\ZZ)$$ for all $p>M$ and all $m\geq 1$?
\end{Op}

\section*{Acknowledgements}
The author would like to thank T. Browning, R. Cluckers, I. Glazer, F. Loeser, M. Musta\c{t}\u{a} and W. Veys for inspiring discussions during the preparation of this paper. In particular, the author wants to thank Y. Hendel for many constructive comments and his suggestion on the probabilistic Waring type problem. Moreover, the author is grateful to the anonymous referee whose dedicated work improved the quality of the paper especially on its structure together with proposing the second part of Proposition \ref{equcondilog} and its corollaries on the relation between semi-log canonical singularities and (local) motivic oscillation index. The author is partially supported by  Fund for Scientific Research-Flanders (Belgium) (F.W.O.) 1270923N and the Excellence Research Chair “FLCarPA: L-functions in positive characteristic and applications” financed by the Normandy Region.

\bibliographystyle{amsplain}
\bibliography{anbib}
\printindex
\end{document}